\documentclass[11pt]{article}
\UseRawInputEncoding
\usepackage{bbm}
\usepackage{amssymb, amsthm, amsmath, amscd}
\usepackage[usenames,dvipsnames]{color}
\newtheorem{thm}{Theorem}[section]
\newtheorem{lem}[thm]{Lemma}
\newtheorem{prop}[thm]{Proposition}
\newtheorem{cor}[thm]{Corollary}
\theoremstyle{definition}

\theoremstyle{definition}
\newtheorem{df}[thm]{Definition}
\theoremstyle{definition}
\newtheorem{rem}[thm]{Remark}
\newtheorem{nota}[thm]{Notation}
\theoremstyle{definition}
\newtheorem{exm}[thm]{Example}
\newcommand{\red}{\textcolor{red}}
\renewcommand{\phi}{\varphi}

\newcommand{\N}{\mathbb{N}}
\newcommand{\Z}{\mathbb{Z}}
\newcommand{\Q}{\mathbb{Q}}
\newcommand{\R}{\mathbb{R}}
\newcommand{\C}{\mathbb{C}}

\numberwithin{equation}{section}

\newcommand{\Aff}{\operatorname{Aff}}

\newcommand{\Tw}{\overline{T(A)}^w}

\newcommand{\cpc}{c.p.c.~map}
\newcommand{\hm}{homomorphism}
\newcommand{\dt}{\delta}
\newcommand{\ep}{\varepsilon}

\newcommand{\td}{\tilde}

\newcommand{\la}{\langle}
\newcommand{\ra}{\rangle}
\newcommand{\andeqn}{\,\,\,{\rm and}\,\,\,}
\newcommand{\rforal}{\,\,\,{\rm for\,\,\,all}\,\,\,}
\newcommand{\CA}{$C^*$-algebra}
\newcommand{\SCA}{$C^*$-subalgebra}

\newcommand{\af}{{\alpha}}
\newcommand{\bt}{{\beta}}

\newcommand{\diag}{{\rm diag}}

\newcommand{\wtd}{\widetilde}

\newcommand{\wilog}{without loss of generality}
\newcommand{\Wlog}{Without loss of generality}

\newcommand{\beq}{\begin{eqnarray}}
\newcommand{\eneq}{\end{eqnarray}}
\newcommand{\tforal}{\,\,\,\text{for\,\,\,all}\,\,\,}
\newcommand{\tand}{\,\,\,\text{and}\,\,\,}

\usepackage{amsfonts}
\usepackage{mathrsfs}
\usepackage{textcomp}
\usepackage[all]{xy}

\newcommand{\Her}{\mathrm{Her}}
\newcommand{\Cu}{\mathrm{Cu}}

\newcommand{\Qw}{\overline{QT(A)}^w}
\newcommand{\vo}{\varpi}
\newcommand{\Av}{l^\infty(A)/c_{\varpi}(A)}
\newcommand{\Avt}{l^\infty(A)/I_{_{\tau, \varpi}}}
\newcommand{\AvS}{l^\infty(A)/I_{_{S, \varpi}}}
\newcommand{\Nv}{l^\infty(N)/c_{\varpi}(N)}
\newcommand{\Nuv}{l^\infty(N)/I_{_{\varpi,\tau}}(N)}
\newcommand{\Ccq}{l^\infty(A)/I_{_{T(A),\varpi}}}
\newcommand{\Cq}{l^\infty(A)/I_{_{\overline{T(A)}^w,\varpi}}}
\newcommand{\Cqc}{(l^\infty(A)\cap A')/I_{_{\overline{T(A)}^w, \varpi}}}
\newcommand{\Cqcc}{(l^\infty(A)\cap A')/I_{_{T(A), \varpi}}}

\title{Tracial oscillation zero and ${\cal Z}$-stability}
\author{Huaxin Lin}
\date{ }

\begin{document}

\maketitle

\begin{abstract}
Let $A$ be a (not necessarily unital) separable  non-elementary
simple amenable \CA\, whose 
tracial  basis may not have finite covering dimension and may not be compact
but 
satisfies  {{a}} certain condition (C). 
We show that $A$ is ${\cal Z}$-stable if and only if $A$ has strict comparison 
for positive elements. 
Extremal boundaries of simplexes  which satisfy condition (C) 
may contain
countable  disjoint unions of $n$-dimensional cubes ($n\in \N$) as a subset.
%Non-unital version of the latter   statement also holds. 
\end{abstract}

\section{Introduction}

The recent progress in the classification of simple \CA s are for separable simple 
amenable \CA s which tensorially absorb the Jiang-Su algebra. These \CA s are 
called ${\cal Z}$-stable \CA s. 
In fact, almost all classification results in the Elliott program, the program of 
classification of separable amenable \CA s,  are for ${\cal Z}$-stable \CA s 
(\cite{Ell93}, \cite{Ellicm}, \cite{Rro2}, \cite{P},  \cite{KP}, \cite{EG},  \cite{EGL},  \cite{Linduke}, \cite{Lininv},
\cite{GLNI}, \cite{EGLN}, \cite{TWW}, to name
only  a few).
The Jiang-Su algebra ${\cal Z}$ is an infinite dimensional 
simple \CA\, with a unique tracial state whose ordered $K$-theory 
is exactly the same as that of the complex field (see \cite{JS}). 
If $A$ is a separable simple \CA\, with weakly unperforated $K_0(A),$ 
then $A$ and $A\otimes {\cal Z}$ have exactly the same Elliott invariant
(see \cite{GJS}).
In other words, the current restriction to ${\cal Z}$-stable \CA s in the Elliott program 
is not the limitation of methods but rather a discovery.

The fact that regularity conditions are required in the study of the classification program
has a long history. It was in the horizon since B. Blackadar's strict comparison was introduced in 1980's (see \cite{Bcomp}).
Perhaps, the first deep result in regularity conditions for simple \CA s
is the introduction of the notion of slow dimension growth in \cite{DNNP}. 
Stable rank (\cite{Rff}), real rank (\cite{BP}),  tracial rank (\cite{Lintrk}),  decomposition rank
(\cite{KW}), as well as nuclear dimension (\cite{WZ})
have been introduced to the \CA\, theory. They may all be regarded as regularity 
conditions. 

Toms-Winter's conjecture states (for stably finite case)
 that for a separable non-elementary stably finite simple amenable \CA\, $A,$ 
the following are equivalent:
% (\cite{TW?}):

(1) $A$ has strict comparison,

(2)  $A\cong A\otimes {\cal Z},$

(3) $A$ has finite nuclear dimension.

By now, the equivalence of (2) and (3) has been established (see \cite{CE}, \cite{CETWW}, 
\cite{Winter-Z-stable-02}, \cite{T-0-Z}, \cite{MS2}).  
The implication (2) $\Rightarrow$ (1) is established earlier (\cite{Rrzstable}).

The remaining direction is the implication (1) $\Rightarrow$ (2) (or (3)). 
A breakthrough was made by the celebrated work of Matui and Sato (\cite{MS})  which shows, 
among other directions of Toms-Winter's conjecture, that, if $A$ is a unital stably finite separable non-elementary 
amenable simple \CA\, with finitely many extremal traces and has strict comparison, 
then $A\cong A\otimes {\cal Z}.$  The newly invented methods was later developed 
further (see \cite{KR}, \cite{S-2} and \cite{TWW-2}) to show that Toms-Winter's conjecture holds for 
those unital separable simple amenable \CA s $A$ whose
 tracial state space $T(A)$ is a Bauer simplex with finite dimensional  extremal boundary.
% (there is also a paper \cite{Zw}, see Remark \ref{RwZ}). 
%(see also \cite{S-2}, \cite{TWW-2}).   
Wei Zhang (\cite{Zw}) showed that Toms-Winter conjecture
holds for unital separable amenable simple \CA s whose  extremal tracial states are  
finite-dimensional with  
 %may not be Bauer simplices
%but have  finite dimensional extremal boundaries  which  have the 
tightness property (but not closed).
 More recently in \cite{CETW}, it is shown 
that Toms-Winter's conjecture holds under the assumption of uniform  property $\Gamma$
(see also \cite{Lingamma}).
%
%%%%%%%%%%%%%%%%%%
\iffalse
To be more precise, let $\partial_e(T(A))$ be the extremal boundary of the Choquet simplex 
$T(A).$  Matui-Sato (\cite{MS}) showed that Toms-Winter's conjecture holds for 
those unital separable simple amenable \CA s $A$ whose  $\partial_e(T(A))$ has 
only finitely many points.  In \cite{KR}, \cite{S-2} and \cite{TWW-2}, Kirchberg-R\o rdam,
Sato and Toms-White-Winter showed that the condition 
can be eased to that $\partial_e(T(A))$ is closed and finite-dimensional.
Wei Zhang \cite{Zw}  late showed that Toms-Winter's conjecture 
holds for those unital separable simple amenable \CA s whose  extremal boundary is finite dimensional and tight. 
\fi
%%%%%%%%%%%%%%%%%%%%%%%%%%%%%%%%%%%%
%and 
%has finite dimensional  

In this paper, we show that the conjecture holds  under the additional 
conditions  that $A$ has stable rank one
%tracial approximate oscillation zero 
and 
% in two cases: the case that 
%$\partial_e(T(A))$ 
the tracial cone $\wtd{T}(A)$  (the cone of densely defined traces)  has a basis 
with condition (C) (see Definition \ref{Dc1} and Theorem \ref{TM-1}).  
If one of the basis has countable extremal boundary, then the condition 
that $A$  has stable rank one can be dropped  (see Corollary  \ref{CCM-1}). 
One example of Choquet simplex $T$ which satisfies the condition (C)
is the simplex whose extremal boundary $\partial_e(T)$ is the
one point compactification of countable  disjoint union of finite (but not bounded)
dimensional cubes (or any metric spaces with finite covering dimension).  So $\partial_e(T)$ is not finite dimensional.
There are also examples of Choquet simplexes which satisfy condition (C) but 
are not Bauer simplexes  and their  extremal boundaries  are not finite dimensional (see examples in \ref{exm-1}). 
It should be noted that extremal boundaries $\partial_e(T(A))$
of $T(A)$ in \cite{MS},  \cite{KR}, \cite{S-2} and \cite{TWW-2}), as well as in \cite{Zw} are all finite dimensional.
Except those studied in \cite{Zw}, tracial state spaces $T(A)$ studied in these papers
(\cite{MS}, \cite{KR}, \cite{S-2} and \cite{TWW-2})  all 
satisfy condition (C).

%t is shown in \cite{FLosc} that, for a finite separable simple \CA\,  $A$ with strict comparison,
%$A$ has tracial approximate oscillation zero if and only if $A$ has stable rank one. 
%
%countable extremal boundary (see 
%Definition 
%\ref{Dcextremal} and 
%a basis $S$  whose extremal boundary $\partial_e(S)$  has countably many points
%Theorem \ref{TM-1}), and the case  that $A$ is unital,  has stable rank one, and its tracial state 
%space has strong countable-dimensional extremal boundary (Theorem \ref{TM-2} and 
%also Definition \ref{Dcountable}). 
%In both cases, we do not assume that $T(A)$ is a Bauer simplex. 
%These results in this paper are only small steps 
%This result is only a step 
%in the direction
%of (1) $\Rightarrow$ (2). %
%In fact, in the unital case, this has been shown by Wei Zhang. 
%Nevertheless these are   generalization of  Matui-Sato's result in \cite{MS} and complimentary to that 
%in \cite{KR},  \cite{S-2} and \cite{TWW-2}. 
%(We have an additional condition that $A$ has stable rank one
%in Theorem \ref{TM-2}).
% as well as that of W. Zhang. 
%Perhaps, 
%It is worth  noting   that Theorem \ref{TM-1} in this paper does not require the tracial state space to be 
%a Bauer simplex  (see also  Remark \ref{RwZ} for Wei  Zhang's  paper \cite{WZ}).
%that this in the direction of (1) $\Rightarrow$ (2) allows 
%non-Bauer simplices to be the tracial state spaces of simple \CA s in an abstract way.  
%

If $A$ is not unital but $A\otimes {\cal K}$ has a nonzero projection $p,$ then 
the previous results 
%can be applied 
apply to \CA s $B=p(A\otimes {\cal K})p$ which  are
unital. Since $A$ is assumed to be separable and simple, by Brown's stable isomorphism 
theorem (see \cite{Br}), one obtains $B\otimes {\cal K}\cong A\otimes {\cal K}.$ Therefore the
results mentioned in \cite{MS}, \cite{KR}, \cite{S-2} and \cite{TWW-2} can be applied to the case that $A$ is not unital 
but $K_0(A)_+\not=\{0\}.$  
However, this  stabilization method fails when $A$ is a
stably projectionless \CA.  Theorem \ref{TM-1} in this paper  includes  the case that $A$  is a  stably projectionless  separable 
simple amenable \CA\,  whose   $\wtd{T}(A)$ has  a basis which satisfies condition (C)
%cone of densely defined traces has a 
%compact basis $S$ whose extremal boundary has 
%countable  extremal boundary (see \ref{Dcextremal} 
(see \ref{Dc1} and also \ref{Pextrb}).

This paper uses the notion of  tracial approximate oscillation zero. 
Let us, for simplicity,  assume that $A$ is a unital separable simple \CA. 
Denote by  $\Gamma: {\rm Cu}(A)\to {\rm LAff}_+(T(A))$ (the set of strictly positive  lower-semicontinuous affine functions
on the tracial state space $T(A)$)  the canonical map defined by $\la a\ra \mapsto \widehat{\la a\ra}=d_\tau(a)$ (the dimension function of $a$).
In general $\widehat{\la a\ra}$ is not continuous. Let $\omega(a)$ be the oscillation of the 
lower-semicontinuous function $\widehat{\la a\ra}.$ 
Roughly speaking, a separable simple \CA\, has tracial approximate oscillation zero, if 
every positive element in $A$ can be approximated (tracially)
%in the sense of Cuntz equivalence) 
by 
those elements whose tracial oscillations tend to zero.   This notion will be discussed in detail in \cite{FLosc}.
Among other things, we show that if $A$ is a (not necessarily unital) separable simple \CA\, which has 
strict comparison, then $A$ has tracial approximate oscillation zero if and only if $A$ has stable rank one 
(and therefore, by \cite{APRT}, the canonical map $\Gamma$ is surjective). 
In the case that $\partial_e(T(A))$ has countably many points, $A$ always has 
tracial approximate oscillation zero. 
%The notion  of tracial oscillation zero plays the key role in this paper. 
Perhaps, the notion of tracial oscillation zero may play further role in the study of simple \CA s.

The paper is organized as follows. 
Section 2 serves  as the  preliminary of the paper, where, 
among other items, we provide examples of non-Bauer simplexes 
%$S$ 
which satisfy condition (C).  Some of them are not finite dimensional.
%
%such that
%its 
%whose extremal boundary $\partial_e(S)$ has  countably many points. 
We also show 
%hat, 
that, for non-unital  separable simple \CA s, the condition that 
%the extremal boundary of 
one of its bases   
for the cone of traces satisfies condition (C)
does not depend on the choice of the basis.  
In section 3 we recall some variations of results related to Matui-Sato's paper \cite{MS} and we mix them with the notion of 
tracial approximate oscillation zero. In section 4, we present some technical results 
which allow us to add maps 
approximately 
orthogonally,  using the notion of tracial 
%approximate 
oscillation zero. 
The last section contains the proof of the main result of the paper. 

\vspace{0.2in}

{\bf Acknowlegements}
  This research
    is partially supported by a NSF grant (DMS 1954600)
and  the Research Center for Operator Algebras in East China Normal University
which is partially supported by Shanghai Key Laboratory of PMMP, Science and Technology Commission of Shanghai Municipality (STCSM), grant \#13dz2260400. 
The author would like to thank the referee for many valuable suggestions.

\section{Preliminary}

\begin{df}\label{D1}
Let $A$ be a \CA\,
and
$S\subset A$  a subset of $A.$
Denote by
${\rm Her}(S)$
% (or just $\Her(S),$ when $A$ is clear)
the hereditary $C^*$-subalgebra of $A$ generated by $S.$
Denote by $A^{\bf 1}$ the closed unit ball of $A,$  and 
by $A_+$ the set of all positive elements in $A.$
Put $A_+^{\bf 1}:=A_+\cap A^{\bf 1}.$
%and by
%$A_{sa}$ the set of all self-adjoint elements in $A.$
Denote by $\wtd A$ the minimal unitization of $A.$
%When $A$ is unital, denote by $GL(A)$ the group of invertible elements of $A,$
 %and denote by  $U(A)$ the unitary group of $A.$
 {{Let  ${\rm Ped}(A)$ denote
the Pedersen ideal of $A$ and ${\rm Ped}(A)_+:= {\rm Ped}(A)\cap A_+$.  Denote by $T(A)$ the tracial state space of $A.$}}
\end{df}

\begin{df}
Let $A$ and $B$ be \CA s and 
$\phi: A\rightarrow B$ be a  linear map.
% between \CAs.
The map $\phi$ is said to be positive if $\phi(A_+)\subset B_+.$  
%and 
The map $\phi$ is said to be completely positive contractive, abbreviated to c.p.c.,
if $\|\phi\|\leq 1$ and  
$\phi\otimes \mathrm{id}: A\otimes M_n\rightarrow B\otimes M_n$ is 
positive for all $n\in\mathbb{N}.$ 
A c.p.c.~map $\phi: A\to B$ is called order zero, if for any $x,y\in A_+,$
$xy=0$ implies $\phi(x)\phi(y)=0$ (see  Definition 2.3 of \cite{WZ}).
If $ab=ba=0,$ we also write $a\perp b.$

In what follows, $\{e_{i,j}\}_{i,j=1}^n$ (or just $\{e_{i,j}\},$ if there is no confusion) stands for  a system of matrix units for $M_n$ and $\iota\in C_0((0,1])$ 
is the identity function on $(0,1],$  i.e., $\iota(t)=t$ for all $t\in (0,1].$
\end{df}

\begin{nota}
Throughout the  paper,
the set of all positive integers is denoted by $\N.$ 
{{The set of all compact operators on a separable 
infinite-dimensional Hilbert 
space is denoted by ${\cal K}.$}} 

Let  $A$ 
be a normed space and ${\cal F}\subset A$ be a subset. For any  $\epsilon>0$ and 
$a,b\in A,$
we  write $a\approx_{\epsilon}b$ if
$\|a-b\|< \epsilon$.
We write $a\in_\ep{\cal F}$ if there is $x\in{\cal F}$ such that
$a\approx_\ep x.$

%Let $A$ be a \CA\ and $x\in A.$ Let $|x|:=(x^*x)^{1/2}.$

\end{nota}

\begin{nota}\label{Nfg}
Let $\epsilon>0.$ Define a  continuous function
$f_{\epsilon}
: {{[0,+\infty)}}
\rightarrow [0,1]$ by
\beq\nonumber
f_{\epsilon}(t)=
%\left\{\begin{array}{ll}
\begin{cases}
0  &t\in {{[0,\epsilon/2]}},\\
1 &t\in [\epsilon,\infty),\\
\mathrm{linear } &{t\in[\epsilon/2, \epsilon].}
\end{cases}
%\andeqn
%%%%g_\dt(t)=\begin{cases}
%\left\{\begin{array}{ll}
%%%%0  &t\in \{0\}\cup [\dt, 
%%%%\infty),\\
%%%%%1 &{t\in[\dt/8, \dt/2]}\\
%%%%\mathrm{linear}  &t\in [0,\dt/8]\cup [\dt/2, \dt].
%\mathrm{linear } 
%1 &{t\in[\dt/8, \dt/2].}
%%\end{cases}
%{array}\right
\eneq

\end{nota}

\begin{df}\label{Dcuntz}
Let $A$ be a \CA\, and 
%and  let $M_{\infty}(A)_+:=\bigcup_{n\in\mathbb{N}}M_n(A)_+$.
%For $x\in A\otimes {\cal K},$
%we identify $x$ with ${\rm diag}(x,0)\in M_{n+m}(A)$
%for all $m\in \N.$
$a,\, b\in (A\otimes {\cal K})_+.$ 
We 
%may write  $a\oplus b:=\mathrm{diag}(a,b)\in M_{n+m}(A)_+.$
%If $a, b\in M_n(A),$
%we 
write $a \lesssim b$ if there is 
$x_i\in A\otimes {\cal K}$  for all $i\in \N$ 
such that
$\lim_{i\rightarrow\infty}\|a-x_i^*bx_i\|=0$.
We write $a \sim b$ if $a \lesssim b$ and $b \lesssim a$ both  hold.
%We also write $a\lesssim b$ and $a\sim b,$
%when {{there is no confusion.}}
The Cuntz relation $\sim$ is an equivalence relation.
Set $\Cu(A)=(A\otimes {\cal K})_+/\sim.$
%W(A):=M_{\infty}(A)_+/\sim$.
Let $[a]$ denote the equivalence class of $a$. 
We write $[a]\leq [b] $ if $a \lesssim  b$.
%$(W(A),\leq)$ is a partially ordered abelian semigroup.
%Let ${\rm Cu}(A)=W(A\otimes {\cal K}).$ 
%$W(A)$ (resp. ${\rm Cu}(A)$) is called almost unperforated,
%if for any $\la a \ra, \la b\ra\in W(A)$ (resp. ${\rm Cu}(A)$),
%and for any $k\in\N$,
%if $(k+1)\la a\ra \leq k\la b\ra$,
%then $\la a \ra \leq \la b\ra$
%(see \cite{Rordam-1992-UHF2}).
%Denote by  $V(A)$  the subset of those elements in $W(A)$ represented by projections.
% {\blue{if $A$ is a stably finite $C^*$-algebra.}}
\end{df}

\begin{df}\label{Dqtr}
{{Let $A$ be a $\sigma$-unital \CA. 
A densely  defined 2-quasi-trace  is a 2-quasi-trace defined on ${\rm Ped}(A)$ (see  Definition II.1.1 of \cite{BH}). 
Denote by ${\widetilde{QT}}(A)$ the set of densely defined quasitraces 
on %of 
$A\otimes {\cal K}.$  
%Suppose, for the convenience of this paper, that all 2-quasi-traces defined on ${\rm Ped}(A)$ 
%are traces. 
%
Denote by ${\widetilde{T}}(A)$ the set of densely defined traces 
on %of 
$A\otimes {\cal K}.$  
 In what follows we will identify 
$A$ with $A\otimes e_{1,1},$ whenever it is convenient. 
Let $\tau\in {\widetilde{QT}}(A).$  Note that $\tau(a)\not=\infty$ for any $a\in {\rm Ped}(A)_+\setminus \{0\}.$

We endow ${\widetilde{QT}}(A)$ 
{{with}} the topology  in which a net 
%$\tau_i$
${{\{}}\tau_i{{\}}}$ 
 converges to $\tau$ if 
 %$\tau_i(a)$ 
${{\{}}\tau_i(a){{\}}}$ 
 converges to $\tau(a)$ for all $a\in 
 {\rm Ped}(A)$ 
 (see also (4.1) on page 985 of \cite{ERS}).
 A  convex subset $S\subset {\wtd{QT}}(A)\setminus \{0\}$ is a basis
 for ${\wtd{QT}}(A),$ if for any $t\in {\wtd{QT}}(A)\setminus \{0\},$ there exists a unique  pair 
 $r\in \R_+$ (non-negative real numbers) and 
 $s\in S$ such that $r\cdot s=t.$ 
% Let $e\in {\rm Ped}(A)_+\setminus \{0\}$ be a full element of $A.$ 
 Denote by $QT(A)$ the set of normalized 2-quasitraces of $A$ ($\|\tau\|=1$).
 Let $e\in {\rm Ped}(A)_+\setminus \{0\}$ be a full element of $A.$ 
 Then $S_e=\{\tau\in {\wtd{QT}}(A): \tau(e)=1\}$ is a Choquet simplex and is a basis 
 for the cone ${\wtd{QT}}(A)$ (see Proposition 3.4 of \cite{TT}).
% {\red{By Theorem 4.4 of \cite{ERS}, 
% if $A$ is separable, this topology is second countable. 
% Suppose that $T\subset {\widetilde{QT}}(A)$ is a compact subset.  Then 
% $T$ is metrizable.}} 

% We  will write  $\tau(a)$ for $\tau|_{{\rm Ped}(A)}(a)$ for any $a\in A.$
%{\green{(do you mean $a\in {\rm Ped}(A)?$)}}
%We will continue to use $\tau$ for its extension on $M_n(A)$
%for each $n\ge 1.$ 
}}

Note that, for each $a\in ({{A}}%A_+
\otimes {\cal K})_+$ and $\ep>0,$ $f_\ep(a)\in {\rm Ped}(A\otimes {\cal K})_+.$ 
Define 
\beq
\widehat{[a]}(\tau):=d_\tau(a)=\lim_{\ep\to 0}\tau(f_\ep(a))\rforal \tau\in {\widetilde{QT}}(A).
\eneq
Except in subsection \ref{sub2}, we will assume that all 2-quasitraces of a separable \CA\, $A$ 
in this paper  
are in fact traces for the convenience. This is the case if $A$ is exact (by \cite{Haagtrace}).

\end{df}

\begin{df}
Let $A$ be a simple \CA\,with $\wtd{QT}(A)\setminus \{0\}\not=\emptyset.$
% whose  2-quasi-traces are all traces.   
Then
$A$
is said to have (Blackadar's) strict comparison, if, for any $a, b\in (A\otimes {\cal K})_+,$ 
condition 
\beq
d_\tau(a)<d_\tau(b)\rforal \tau\in {\widetilde{QT}}(A)\setminus \{0\}
\eneq
implies that 
$a\lesssim b.$ 
%provided
%\beq
%d_\tau(a)<d_\tau(b)\rforal \tau\in {{{\widetilde{T}}(A)\setminus \{0\}.}}
%\eneq
\end{df}

\begin{df}\label{DGamma}
Let $A$ be a \CA\, with ${\wtd{QT}}(A)\setminus\{0\}\not=\emptyset.$
Let $S\subset {\wtd{QT}}(A)$ be a convex subset. 
Let $\Aff(\wtd{QT}(A))$ be the set of 
real affine  continuous functions {\red{$f$}} on $\wtd{QT}(A)$ such that
{\red{$f(0)=0.$}}
%$f(rs)=rf(s)$ for all $s\in \wtd{QT}(A)$ and $r\in \R_+.$
Set 
%(if $0\not\in S,$ we ignore the condition $f(0)=0$)
\beq
\Aff_+(S)&=&\{f|_S: f\in \Aff(\wtd{QT}(A)):  f(s)>0\,\,{\rm for}\,\,s\not=0 \}
% f(0)=0\}
\cup \{0\},\\
%
%{\rm LAff}_f(S)_+&=&\{f:S\to [0,\infty): \exists \{f_n\}, f_n\nearrow f,\,\,
 %f_n\in \Aff(S)_+\},\\
%{\rm LAff}_{f,+}(S)&=&\{f:S\to [0,\infty): \exists \{f_n\}, f_n\nearrow f,\,\,
 %f_n\in \Aff_+(S)\},\\
%{\rm LAff}(S)_+&=&\{f:S\to [0,\infty]: \exists \{f_n\}, f_n\nearrow f,\,\,
% f_n\in \Aff(S)_+\},\\
{\rm LAff}_+(S)&=&\{f: S\to [0, \infty]:  \exists \{f_n\}, \,\,f_n\nearrow f,\,\,  f_n\in \Aff_+(S)\}.
 \eneq
%Note that, if $0\in S,$ $f(0)=0$ for all $f\in {\rm LAff}_+(S).$
For a  simple \CA\, $A$ and   each $a\in (A\otimes {\cal K})_+,$ the function $\hat{a}(\tau)=\tau(a)$ ($\tau\in S$) 
is in general in ${\rm LAff}_+(S).$   If $a\in {\rm Ped}(A\otimes {\cal K})_+,$
then $\hat{a}\in \Aff_+(S).$
For 
$\widehat{[a]}(\tau)=d_\tau(a)$ defined above,   we have 
$\widehat{[a]}\in {\rm LAff}_+({\wtd{QT}}(A)).$

We write $\Gamma: \Cu(A)\to {\rm LAff}_+({\wtd{QT}}(A))$ for 
the canonical map defined by $\Gamma([a])(\tau)=\widehat{[a]}=d_\tau(a)$ 
for all $\tau\in {\wtd{QT}}(A).$

In the case that $A$ is  algebraically simple (i.e., $A$ is a simple \CA\,  and $A={\rm Ped}(A)$), 
$\Gamma$ also induces a canonical map 
$\Gamma_1: \Cu(A)\to {\rm LAff}_+(\Qw),$
where $\Qw$ is the weak*-closure of $QT(A).$  Since, in this case, 
$\R_+\cdot \Qw={\wtd{QT}}(A),$ the map $\Gamma$ is surjective if and only if $\Gamma_1$
is surjective.  We would like to point out that, in this case, $0\not\in \Qw$ (see Lemma 4.5 of \cite{eglnp}).
In the case that $A$ is stably finite and simple,
denote by $\Cu(A)_+$ the set of purely non-compact elements (see Proposition 6.4 of \cite{ERS}).
Suppose that $\Gamma$ is surjective.  Then  $\Gamma|_{\Cu(A)_+}$ is surjective as well  (see Theorem 7.12 \cite{FLosc}, for example).
%{TOstosurj-1}
%
%%%%%%%%%%%%%%%
%
\iffalse
Let $p\in (A\otimes {\cal K})_+$ be a projection (so $p\in {\rm Ped}(A\otimes {\cal K})$).
There are $a_n\in (A\otimes {\cal K})$ with $0\le a_n\le 1$ such 
that $\widehat{[a_n]}=(1/2^n)\widehat{[p]},$ $n\in \N.$ 
Define $b=\diag(a_1/2, a_2/2^2,...,a_n/2^n,...)\in  A\otimes {\cal K}.$ 
Then $0$ is a limit point of ${\rm sp}(b).$ Therefore $[b]$ cannot be represented by 
a projection. In other words, $[b]\in \Cu(A)_+.$  
We compute that $\hat{[b]}=\hat{[p]}.$ 
It follows that  $\Gamma|_{\Cu(A)_+}$ is surjective. }}
\fi
%%%%%%%%%%%%%%%%%%%%%%%%%%%%%%%%%%%%%%%%
%
%%%%%%%%%%%%%%%%%%
\end{df}

\begin{df}\label{DNcu}
Let $l^\infty(A)$ be the \CA\ of bounded sequences of $A.$
{{Recall}} that 
$c_0(A):=\{\{a_n\}\in l^\infty(A): \lim_{n\to\infty}\|a_n\|=0\}$ 
is a (closed two-sided) ideal of $l^\infty(A).$ 
%Let $A_\infty:=l^{\infty}(A)/c_0(A)$  and 
% $\Pi_\infty: l^\infty(A)\to 
%l^\infty(A)/c_0(A)
%A_\infty$ be the quotient map. 
%Let $\iota: A \hookrightarrow l^\infty(A)$ be the canonical embedding, 
%let $\iota_\infty:=\pi_\infty\circ\iota.$ 
%We identify $A$ with $\iota(A)$ and $\iota_\infty(A).$
We view $A$ as a subalgebra of $l^\infty(A)$ 
via the canonical map $\iota: a\mapsto\{a,a,,...\}$ for all $a\in A.$
In what follows, we may identify $a$ with the constant sequence $\{a,a,...,\}$ in 
$l^\infty(A)$ without further warning.

Put $A'=\{x=\{x_n\}\in l^\infty(A):  \lim_{n\to\infty}\|x_na-ax_n\|=0\}.$ 
%It is called the central sequence algebra of $A.$
%We also use $A'$ for $\Pi_\infty(A').$ 
\end{df}

%%%%%%%%%%%%
%\iffalse
\begin{df}\label{Dc0}
Let $A$ be a $\sigma$-unital simple \CA,
 %and 
%$\{a_n\}\in l^\infty(A)_+.$ 
%We write $a_n\cto 0,$ if for any $b\in A_+\setminus \{0\},$ there exists $n_0\in \N$ such 
%that $a_n\lesssim b$ for all $n\in \N$ and $n\ge n_0.$  
%The sequence $\{x_n\}$ is call Cuntz-null
%if $f_\ep(x_n^*x_n)\cto 0$ for all $\ep\in (0,1).$ 
%Denote by $N_{cu}:=N_{cu}(A)$  the subset of $l^\infty(A)$ consisting of Cuntz-null 
%sequences. 
%It is proved in \cite{FLL} that, if $A$ is a non-elementary simple \CA,
%  $N_{cu}(A)$ is a (closed two-sided) ideal of $l^\infty(A)$ containing $c_0(A)$ (see Proposition 3.5 of \cite{FLL}). 
%
$e\in A_+^{\bf 1}$ be a strictly positive element
and $e_n=f_{1/2n}(e),$ $n\in \N.$ Recall that (see Definition 2.5 of \cite{Lin91cs})
$A$ is said to have continuous scale if and only if, for any finite subset 
${\cal F}\subset A_+\setminus \{0\},$ there exists $n_0\in \N$ such that,
for any $n\ge n_0,$ 
$
e_m-e_n\lesssim b\rforal b\in {\cal F}.
$
%whenever $n\ge n_0.$
 %for any $m(n)>n,$ $a_n=(e_{m(n)}-e_n)\cto 0.$  
%This is also equivalent to saying 
%that $a_n\in N_{cu}(A)$ for $m(n)>n$ (with a little work---see 2.20 of \cite{FLosc}, for example).  
\end{df}
%\fi
%%%%%%%%%%%%%%%%%%%
%
%
%%%%%%%%%%%%%
\begin{df}
Let $A$ be a separable simple \CA\, with ${\rm Ped}(A)=A$ and $\wtd{QT}(A)\setminus \{0\}\not=\emptyset.$
%not=\{0\}.$
%Suppose that all 2-quasitraces of $A$ are traces. 
Let $\tau\in {\wtd{QT}}(A)\setminus \{0\}.$
Define, for each $x\in A,$
\beq
\|x\|_{_{2,\tau}}=\tau(x^*x)^{1/2}.
\eneq
Let $S\subset {\wtd{QT}}(A)\setminus \{0\}$ be a compact subset. Define 
\beq
\|x\|_{_{2, S}}=\sup\{\tau(x^*x)^{1/2}:\tau\in S\}.
\eneq
%Put $I_\tau=\{\{x_n\}\in l^\infty(A): \lim_{m\toinfty} \|x\|_{2, \tau}=0\}$ and 
Put $I_{S,\N}=\{\{x_n\}\in l^\infty(A): \lim_{n\to\infty}\|x\|_{_{2, S}}=0\}.$ 
%%%%%
%removed 23-04-9
%It is shown (Proposition 3.8 of  \cite{FLL}) that, with additional assumption that $A$  has strict comparison, 
%$N_{cu}(A)=I_{_{\Qw,\N}}.$
\end{df}

\begin{df}\label{Dultrafiler}
Let $\varpi\in \bt(\N)\setminus \N$ be a free ultrafilter. 
Set
\beq
c_{0,\varpi}=\{\{x_n\}\in l^\infty(A): \lim_{n\to\omega}\|x_n\|=0\}.
\eneq
Let $S\subset {\wtd{QT}}(A)$ be a  compact subset.
Define 
\beq
I_{_{S, \varpi}}=\{\{x_n\}\in l^\infty(A): \lim_{n\to\varpi}\|x_n\|_{_{2, S}}=0\}.
\eneq
It is a (closed two-sided) ideal. 
In the case that $A={\rm Ped}(A),$ we usually consider 
$I_{_{\Qw, \varpi}}.$ If $A$ has continuous scale, 
we consider $I_{_{QT(A), \varpi}}.$  

Denote by $\Pi_\varpi: l^\infty(A)\to l^\infty(A)/_{I_{\Qw, \varpi}}$ the quotient map. 
If $\tau\in {\wtd{QT}}(A),$  for $x=\Pi_\varpi(\{x_n\}),$ define 
\beq
\tau_\varpi(x)=\lim_{n\to\varpi}\tau(x_n).
\eneq
We may view $\tau_\varpi$ as a trace on $l^\infty(A)/_{I_{\Qw, \varpi}}.$

If $S\subset QT(A)$ is a compact subset,  denote by $\pi_S: l^\infty(A)\to l^\infty(A)/I_{_{S, \varpi}}$  
and by $\Phi_S: \Cq\to l^\infty/I_{_{S, \varpi}}$ the quotient maps, respectively.
Note that $S$ could be a single point $\tau.$
For convenience, abusing the notation, 
we may also write $A'$ for $\pi_S(A').$ 
\end{df}

%and $l^\infty(A)/c_0(A)$
%in the canonical way (i.e. the constant sequences). 

\subsection{Extremal Boundaries} 

Throughout this paper, by a Choquet simplex, we mean a metrizable  (compact) Choquet simplex. 

\begin{df}\label{Dc1}
Let $T$ be a Choquet simplex and $X=\partial_e(T)$ be the extremal boundary.
Recall that $\partial_e(T)$ is a $G_\dt$-set (see Cor. I.4.4 of \cite{Alf}).
By the Choquet theorem, for each $t\in T,$ 
there is a  unique Borel probability measure $\mu_t$ on $X$ such that $t$ is the barycenter of   $\mu_t,$ i.e., 
\beq\label{Dc1-0}
f(t)=\int_X f(x) d\mu_t \rforal f\in \Aff(T).
\eneq
%(Throughout this paper, we may identify $t$ with $\mu_t$ whenever it is convenient.)
Let $S\subset \partial_e(T)=X$ be a Borel subset. Denote by
$M_S=\{\mu_t: t\in T\andeqn \mu_t|_{X\setminus S}=0\}.$ In other words,
$M_S$ is the set of those Borel probability measures on $X$ concentrated on $S.$
In what follows we may identify $t$ with $\mu_t$ and $M_S$ with 
the subset of $T$ whose associated extremal boundary measure concentrated on $S.$
Denote by ${\rm conv}(S)$ the convex hull of $S.$

We say $T$ satisfies condition (C), if $T$ has the following properties:

(1) $\partial_e(T)=X=\cup_{n=1}^\infty X_n,$ where each $X_n$ is a  compact subset of $X$ with finite covering 
dimension  and $X_i\cap X_j=\emptyset,$ if $i\not=j,$ and $i,j\ge 2.$

(2) for any $k, m\ge 1,$  $\overline{{\rm conv}(\cup_{j=1}^m X_{j+k}) }=M_{\cup_{j=1}^m X_{j+k}}$ and 
$\overline{{\rm conv}(X\setminus \cup_{j=1}^m X_{j+k})}=M_{X\setminus \cup_{j=1}^m X_{j+k}}.$

\vspace{0.1in}

\end{df}

%{\red{We remark that if $\partial_e(T)$ is a compact space with finite covering dimension,
%then $T$ satisfies condition (C). 
In (1) above,  we note that  $\{{\rm dim}X_n\}$ may not be bounded.
%{\red{Example 3 in 
%
By an elementary measure theory argument, if $S$ is a Borel subset of $\partial_e(T),$
then one always has 
${\rm conv}(S)\subset M_S\subset \overline{{\rm conv}(S)}.$

\begin{prop}\label{P-L}
If $S$ is  a compact subset of $\partial_e(T),$ 
then $\overline{{\rm conv}(S)}=M_S.$ 
\end{prop}

\begin{proof}
By Corollary 11.19 of \cite{kG}, 
$\overline{{\rm conv}(S)}$ is a closed face of $T$ and $\partial_e(\overline{{\rm conv}(S)})=S.$
By the Choquet-Bisop-de Leeuw theorem (cf. Theorem I.4.8  of \cite{Alf}),
for each $x\in \overline{{\rm conv}(S)},$ there exists a measure $m_x$ on $S$ 
such that, 
for any  $g\in \Aff(\overline{{\rm conv}(S)}),$
$g(x)=\int_S g(s) dm_x.$  Define ${\bar m}_x$ on $\partial_e(T)$ 
by ${\bar m}_x(B)=m_x(B\cap S)$ for all Borel subsets  $B$ of $\partial_e(T).$ 
Let $\mu_x$ be as defined in \eqref{Dc1-0}.
Then, by the Choquet theorem, for any $f\in \Aff(T),$ 
\beq
f(x)=\int_{\partial_e(T)}f(t) d\mu_x=\int_{\partial_e(T)} f(t) d{\bar m}_s=\int_S f(s) dm_s.
\eneq
This implies that   $\overline{{\rm conv}(S)}=M_S.$ 
\end{proof}

%In other words, 

\begin{rem}\label{Rmark1}
 (i) By Proposition \ref{P-L},  in the definition of condition (C), one always has 
$\overline{{\rm conv}(\cup_{j=1}^m X_{j+k}) }=M_{\cup_{j=1}^m X_{j+k}}.$
%is automitic. 

(ii) Note also that, if $S_1, S_2\subset X$ is a pair of disjoint subsets, then 
$M_{S_1}\cap M_{S_2}=\emptyset.$   In particular, condition (2) 
implies that  
\beq
\overline{{\rm conv}(\cup_{j=1}^m X_{j+k}) }\,\cap \, \overline{{\rm conv}(X\setminus \cup_{j=1}^m X_{j+k})}=
\emptyset.
\eneq
It might be helpful to note that $k\ge 1$ in (2) above. 

(iii) Let $T$ be a Choquet simplex. If $T$ is a Bauer simple, i.e., $\partial_e(T)$ is a compact space, 
 with finite covering dimension,
%then $T$ satisfies condition (C). 
then $T$ satisfies condition (C).

%(v) Let $T$ be a Choquet simplex. Suppose that there a  finite dimensional compact subset 
%$X_1$ 
%of $\partial_e(T)$ such that $\overline{\partial_e(T)}\setminus \partial_e(T)\subset 
%{\rm conv}(X_1)$ and $\partial_e(T)\setminus X_1$ is a countable set.
%Then $T$ satisfies condition (C). 

(iv) If $T$ is a Bauer simplex, then $X=\partial_e(T)$ is a compact subset of $T.$
Suppose that  $X=\cup_{n=1}^\infty X_i,$ where 
each $X_i$ is compact and finite dimensional such 
that $X_i\cap X_j=\emptyset,$ if $i\not=j$ and $i,j\ge 2.$
Then 
 $X$ satisfies condition (C) if and only if 
 %$X=\cup_{n=1}^\infty X_n,$
   %where each $X_n$ is a finite dimensional compact subset of $X,$  
   $X_n$ is also relatively open
   for $n\ge 2.$
   %$ and $X_i\cap X_j=\emptyset,$
   %if $i\not=j,$ $i,j\ge 2.$     
   To see this, let us assume that $X$ satisfies condition (C).
   %We then write $X=\cup_{n=1}^\infty X_n,$ where $X_n$ satisfies (1) and (2) in \ref{Dc1}.
   %
   Then, for any $k,m\in\N,$ 
   $\overline{X\setminus \cup_{j=1}^m X_{j+k}}\cap (\cup_{j=1}^m X_{j+k})=\emptyset.$
   It follows that each $\cup_{j=1}^m X_{j+k}$ is open.
   % and hence clopen. 
   In particular, $X_j$ is clopen for all $j\ge 2.$
   
   For the converse,  suppose that each $X_i$ is a clopen subset of the compact set $X,$ $i\ge 2.$ 
   Then $X\setminus \cup_{j=1}^m X_{j+k}$ 
   is a closed subset of $X,$ and hence compact. By Proposition \ref{P-L},
   $\overline{{\rm conv}(X\setminus \cup_{j=1}^m X_{j+k})}=M_{X\setminus \cup_{j=1}^m X_{j+k}}.$
   Thus $X$ satisfies condition (C). 
   
(v) Let $\{K_n\}$ be a sequence of compact metric spaces with finite covering dimension.
Let $X$ be the one - point compactification of disjoint union of $\{K_n\}.$ 
Put $D_1=C(X).$ Then $\partial_e(T(D_1))=\cup_{n=1}^\infty X_n,$ 
where $X_1={\rm the\,\, infinite\,\, point}=\{\xi_\infty\}$ and $X_{n+1}=K_n,$ $n\in \N.$ 
Note $\partial_e(T(D_1))=X$ is compact. 
By (iv), 
%By Proposition \ref{P-L}, 
%for each $k,\,m\ge 1,$ $\overline{{\rm conv}(\cup_{j=1}^m X_{j+k})}=M_{(\cup_{j=1}^m X_{j+k}) }.$
%Since  $X\setminus (\cup_{j=1}^m X_{j+k})$ is also compact, 
%by Proposition \ref{P-L} again, $\overline{{\rm conv}(X\setminus \cup_{j=1}^m X_{j+k})} =
%M_{X\setminus (\cup_{j=1}^m X_{j+k}) }.$ Hence 
$T(D_1)$ satisfies the condition (C). 
%
%It is easy to {\red{see}} that $T(D_1)$ satisfies condition (C). 
We would like to mention that  we may choose each $K_n$  the $n$-dimensional cube,
or the $n$-dimensional spheres, 
$n\in \N,$  for example. Then $X$ is compact and 
has countable (but not  necessarily finite) dimension (see also (8) in Example \ref{exm-1}
below for  examples of simple \CA s whose tracial state space is $X$). 

(vi) 
%We do not have an example of Choquet simplex with countable dimension which 
%does not satisfy condition (C). However, we believe that such example exists. 
Examples in   (1) and  (2) 
of Example \ref{exm-1}  below show
that there are different Choquet simplexes with homeomorphic extremal boundaries.
When $\partial_e(T)$ is not compact, we think that 
the condition that $\overline{{\rm conv}(X\setminus \cup_{j=1+k}^m X_j)}=M_{X\setminus \cup_{j=1+k}^m X_j}$
is an affine condition
 %We think 
%that condition (C) is 
% a property for Choquet simplex 
(not merely a topological condition on its extremal points). 
\end{rem}
%One also sees immediately from (2) that $X_1\cap X_n\not=\emptyset$ if $n\not=1.$
%In fact, (2) also implies that $X_i\cap X_j=\emptyset,$ if $i\not=j.$
%Note that $(\cup_{j=1}^m X_{j+k})$ is also locally compact.

%It follows that $\partial_e(T)$ is locally compact and $\sigma$-compact.
%
%\end{df}

\begin{prop}\label{Padd-compact}
Let $T$ be a Choquet simplex and $\partial_e(T)$ be a disjoint union of 
finite dimensional compact sets 
$\{ X_n\}$ such that
 each $X_{1+n}$ is  relatively open (in $\partial_e(T)$)
for all $n\in \N.$ 
%and   $X_i\cap X_j=\emptyset,$ if $i\not=j$ and $i,j\in \N.$ 
Suppose that $\overline{\partial_e(T)}\setminus \partial_e(T)\subset \overline{{\rm conv}(X_1)}.$
Then $T$ satisfies condition (C).
\end{prop}

\begin{proof}
%%%%%
\iffalse
Claim (1):  for any $\eta>0,$ there is an integer $N_0\ge 1$ such that, 
  for all $n\ge N_0,$ ${\rm dist}(X_n, {\rm conv}(X_1))<\eta.$
 % if $t_n\in X_n$ and $n\ge N_0.$
  
  Otherwise, there would be a subsequence $\{n_i\}\subset \N$
  such that   $t_{n_i}\in X_{n_i}$ and 
  ${\rm dist}(t_{n_i}, {\rm conv}(X_1))\ge \eta$ for all $i\in \N.$
  Since $T$ is compact, we may assume that
  $t_{n_i}\to t_0\in T,$ as $i\to\infty.$ 
  It would follow that $t_0\in \overline{{\rm conv}(X_1)}.$ Thus it would imply, for all large $i,$
  ${\rm dist}(t_{n_i}, {\rm conv}(X_1))<\eta/2.$ A contradiction. This proves the claim.
\fi

For  any $k, m\in \N,$ put 
$Y:=\partial_e(T)\setminus \cup_{j=1}^m X_{j+k}.$
To show that $T$ satisfies condition (C), it suffices to  show that $\overline{{\rm conv}(Y)}=M_Y.$
Suppose that $x_n\in {\rm conv}(Y)$ such that
$x_n\to \xi\in T$ as $n\to\infty.$
%To show that $T$ satisfies condition (C), it suffices to 
We will show that $\xi\in M_Y.$

Write $\xi=\af\xi_0+(1-\af)\xi_1,$ where $\xi_0$ associated with a boundary measure 
which concentrated on $\cup_{j=1}^m X_{j+k}$ and $\xi_1$ on $Y,$ and 
$0\le \af\le 1.$    
In another words, 
\beq\label{501-n1}
\mu_\xi(\cup_{j=1}^m X_{j+k})=\af\andeqn \mu_\xi(Y)=(1-\af).\
\eneq
% Since $\sum_{j=1}^m X_{j+k}$ is relatively clopen, $\af\not=1.$
%
Assume that $\af>0.$

Note that $Z_1:=X_1\sqcup \cup_{j=1}^m X_{j+k}$ is compact. 
%By Corollary 11.19 of \cite{kG}, $F=\overline{{\rm conv}(X_1\cup \cup_{j=1}^m X_{j+k})}$
%is a closed face of $T$ and $\partial_e(F)=X_1\cup \cup_{j=1}^m X_{j+k}$ which is a Bauer simplex. 
Let $g\in C(Z_1)$ be such that $0\le g\le 1,$ $g|_{X_1}=0$ and $g|_{\cup_{j=1}^m X_{j+k}}=1.$ 
%Choose $f'\in \Aff(F)_+$ such that $0\le f'\le 1,$ 
%$f'|_{\cup_{j=1}^m X_{j+k}}=0$ and $f'|_{X_1}=1.$
By Theorem 11.14 of \cite{kG}, there is $f\in \Aff(T)$ such that
$0\le f\le 1$ and ${f}|_{Z_1}=g.$ 
%%%%%%%%
%
%%%%%%%%%%%%%%%
%\iffalse
%$f_1\in \Aff(T)_+$ such that ${f_1}|_F=f'$ and $0\le f_1\le 1.$ 
%The same argument 
%provides another function $f_2\in \Aff(T)$ such that $0\le f_2\le 1,$
%${f_2}|_{X_1}=1$ and ${f_2}|_{\cup_{j=1}^m X_{j+k}}=0.$
%Put $f_3=1_{T}.$
%\fi
%%%%%%%%
%
%
%%%%%%%%%%%%%%%%%%%%%%%%%%%%
Let $0<\ep<\af/16.$
%$ and $\eta={\ep\af\over{16}}.$

Claim (1): 
% for any $\eta>0,$
 there is an integer $N_0\ge 1$ such that, 
  for all $n\ge N_0,$ 
  and any  $z\in X_n,$ 
  \beq
 \inf \{|f(z)-f(x)|: x\in {\rm conv}(X_1)\}<\ep/2.
 %\rforal.
 % z\in X_n.
  \eneq
  %${\rm dist}(X_n, {\rm conv}(X_1))<\eta.$
 % if $t_n\in X_n$ and $n\ge N_0.$
  
  Otherwise, there would be a subsequence $\{n_i\}\subset \N$
  such that   $t_{n_j}\in X_{n_j},$ $n_j\to\infty,$ and 
  \beq\label{cdadd4-28}
  \inf\{|f(t_{n_j})-f(x)|: x\in {\rm conv}(X_1)\}\ge \ep/2.
  %,\,\,\, i=1,2,3.
  \eneq
 % and 
 % ${\rm dist}(t_{n_i}, {\rm conv}(X_1))\ge \eta$ for all $i\in \N.$
  Since $T$ is compact, we may assume that
  $t_{n_i}\to t_0\in T,$ as $i\to\infty.$ 
  Since each $X_{1+n}$ is relatively open for all $n\in \N$ and $\overline{\partial_e(T)}\setminus \partial_e(T)\subset \overline{\rm conv}(X_1),$
  it would follow that $t_0\in \overline{{\rm conv}(X_1)}.$
  Then 
  \beq
  \lim_{i\to\infty}|f(t_{n_i})-f(t_0)|=0.
  \eneq
  This contradicts with \eqref{cdadd4-28}. 
  The claim is proved. 
  %This proves the claim.
  We may assume that $N_0>m.$
  % Thus it would imply, for all large $i,$
  %${\rm dist}(t_{n_i}, {\rm conv}(X_1))<\eta/2.$ A contradiction. This proves the claim.

Since $x_n\to \xi,$ we also have
\beq\label{429-n1}
\lim_{n\to \infty}|f(x_n)-f(\xi)|=0.
\eneq
%
%There exists an integer $N_1$ such that
%\beq
%|f_i(x_n)-f_i(\xi)|<\eta\rforal n\ge N_1\andeqn i=1,2,3.
%\eneq
%
%$
% Suppose that ${\rm dist}(x_n, \xi)<\ep/2$
%for all $n\ge  N_1$ (for some $N_1\in \N$).
%By Claim (1), there exists $N_0\ge 2$ such that
%\beq
%x
%{\rm dist}(X_n, {\rm conv}(X_1))<\ep/2\rforal n\ge N_0.
%\eneq
%
%We may assume that $N_0>m.$
    Put $Y_0=X_1\cup (\cup_{j=N_0+1}^\infty X_{j+k}).$ 
    %\cup_{j=m+1}^{N_0}X_{j+k}.$ 
    Recall that $x_n\in {\rm conv}(Y)$ for all  $n\in  \N.$ 
  We may write $x_n=\af_n \xi_n^0+(1-\af_n)\xi_n^1,$ where $0<\af_n\le 1,$
  $\xi_n^0\in M_{Y_0}$ and $\xi_n^1\in M_{\small{\cup_{j=m+1}^{N_0}X_{j+k}}}.$ 
  %$n\in \N.$ 
  Since $\cup_{j=m+1}^{N_0}X_{j+k}$ is compact, 
  %by ?, 
  %$\overline{\cup_{j=m+1}^{N_0}X_{j+k})}$ is a closed face
  %and, 
  by Proposition \ref{P-L}, $\overline{{\rm conv}(\cup_{j=m+1}^{N_0}X_{j+k})}=M_{\cup_{j=m+1}^{N_0}X_{j+k}}.$
  We may assume (by passing to a subsequence)
  % if necessarily,  
  that $\xi_n^1\to \xi^1\in M_{\cup_{j=m+1}^{N_0}X_{j+k}}$
  and $(1-\af_n)\to 1-\bt,$ where $0\le \bt \le 1.$  Hence 
  we may also assume that $\xi_n^0\to \xi^0$ and $\af_n\to \bt.$
  It follows that $\xi=\bt \xi^0+(1-\bt)\xi^1.$ 
  Hence 
  \beq
  \lim_{n\to\infty}|f((1-\af_n)\xi_n^1)-f((1-\bt)\xi^1))|=0.
  \eneq
  Combining with \eqref{429-n1}, we obtain
  \beq
  \lim_{n\to\infty}|f(\af_n\xi_n^0)-\bt f(\xi^0)|=0.
  \eneq
  Since $\af_n\to \bt,$ we actually have
  \beq\label{429-n3}
  \lim_{n\to\infty}\bt |f(\xi_n^0)-f(\xi^0)|=0.
  \eneq
 % If $\bt=0,$ then $\xi=\xi^1\in M_{\cup_{j=m+1}^{N_0}X_{j+k}}\subset M_Y.$
  %We are done.
  %We claim that $\bt \ge \af.$ 
  %If $\bt=1,$ then $\bt\ge \af.$
  %then $\bt\ge \af.$ 
  %%%%%%%%
  \iffalse
  by \eqref{429-n3},  
  \beq\label{430-n1}
  \lim_{n\to\infty} |f(\xi^0_n)-f(\xi^0)|=0.
  \eneq
  \fi
  %%%%%%%
 % So we assume that $0<\bt<1.$ 
Write $\xi^0=t\xi^0_0+(1-t)\xi^0_1,$ where $\xi^0_0\in M_{\cup_{j=1}^m X_{j+k}},$
  $\xi^0_1\in M_Y$ and $0\le t\le 1.$ 
  We claim that $\bt t\ge \af.$  If   $\bt t=1,$ then $\bt t\ge \af.$ 
  Suppose that $0<\bt t<1.$ Then 
  % that (note that, in this case, $\bt t\not=1.$)
  \beq
  \xi&=&\bt \xi^0+(1-\bt)\xi^1=\bt t \xi^0_0+\bt (1-t)\xi^0_1+ (1-\bt)\xi^1\\
    &=&  \bt t \xi^0_0+(1-\bt t)(({\bt (1-t)\over{1-\bt t}})\xi^0_1+({1-\bt\over{(1-\bt t)}})\xi^1).
      \eneq
  Note that ${\bt (1-t)\over{1-\bt t}}+{1-\bt\over{(1-\bt t)}}=1.$
  It follows that  $({\bt (1-t)\over{1-\bt t}})\xi^0_1+({1-\bt\over{(1-\bt t)}})\xi^1\in M_Y.$
  Hence, by \eqref{501-n1},  $\bt t\ge \af>0.$ 
  %in particular, $\bt \ge \af.$
  %$\bt\xi^0=(\bt t)\xi^0_+\bt (1-t)\xi^0_1$
  %%%%%%%%%%%
  \iffalse
  Since $\xi^1\in M_{\cup_{j=m+1}^{N_0}X_{j+k}}$ 
  and $\cup_{j=m+1}^{N_0}X_{j+k}\subset Y,$ we conclude that $\bt\ge \af$
  and $\bt\xi^0=\af\xi_0+\bt_1\xi_1',$ where $0\le \bt_1\le 1$ and 
  $\xi_1'\in M_Y.$
  \fi
  %%%%%%%%%%%%%%
  %Since $\bt\ge \af>0,$  b
  
  By \eqref {429-n3},
  there is  $N_1\in \N$  such that
  % By passing to a subsequence, if necessary, 
  %we may assume $\af_n\to \af$ for some $0\le \af\le 1.$
  %
 % Suppose $0<\af\le 1.$ 
%By?,
\beq\label{429-n2}
|f(\xi_n^0)-f(\xi^0)|<\ep\rforal n\ge N_1.
\eneq
Write $\xi_n^0=\sum_{i=1}^{m(n)}\bt_{n,j} z_{n,j},$
where $z_{n,j}\in Y_0,$  $0\le \bt_{n,j}\le 1$   and $\sum_{j=1}^{m(n)} \bt_{n,j}=1.$
By Claim (1), there are $\zeta_{n,j}\in {\rm conv}(X_1)$
such that
\beq
|f(z_{n,j})-f(\zeta_{n,j})|<\ep/2, \,\, 1\le j\le  m(n),\,\, n\in \N.
\eneq
Put $\zeta_n^0=\sum_{i=1}^{m(n)} \bt_{n,j}\zeta_{n,j}.$ Then $\zeta_n^0\in {\rm conv}(X_1).$
We also have that
%\beq
%
%\eneq 
%This implies that
%\beq
%|f_i(\xi_n^0)-f_i(\xi^0)|<\eta/\bt <{\ep\af\over{16\bt}}\le \ep/16
%\eneq
%
%By Claim (1),  there are $\zeta_n^0\in {\rm conv}(X_1)$  and $N_1\in N$ such that
\beq
|f(\xi_n^0)-f(\zeta_n^0)|<\ep/2.
%\rforal n\ge N_1 \andeqn i=1,2.
\eneq
Since $f(\zeta_n^0)=0,$ we obtain 
$
f(\xi_n^0)<\ep/2.
$
%\eneq
By \eqref{429-n2}, 
%\beq
$f(\xi^0)<3\ep/2.$
%\ep/16+\eta.
%\eneq
Hence 
\beq
f(\bt \xi^0)=\bt f(\xi^0)<\bt \cdot 3\ep/2<3\ep/2<\af/8.
%(\ep/16+\eta)<\af/32+\af/16<\af/4
\eneq
Since $\bt\xi^0=\bt t\xi^0_0+\bt (1-t)\xi^0_1,$ we have 
%Then 
\beq
f(\bt \xi^0)=\bt t f(\xi_0)+\bt (1-t)f(\xi_1^0)\ge \bt t\ge \af.
\eneq
This  is a contradiction. Therefore $\af=0.$ Hence $\xi\in M_Y.$ 
\end{proof}

%
%%%%%%%%%%%%%
\iffalse
contradicts with ?.

Let $y_n=\bt \zeta_n^0+(1-\bt) \xi^1.$

Claim (2): ${\rm dist}(\xi_n^0, {\rm conv}(X_1)<\ep/2$ for all $n\ge N_1.$

We may write $\xi_n^0=\sum_{j=1}^{m(n)} \bt_{n,j} t_{n,j},$
where $0 <\bt_{n,j}\le 1,$ $\sum_{j=1}^{m(n)}\bt_{n,j}=1$ 
and $t_{n,j}\in X_{n(j)}$ and ${n(j)}\ge N_0+1.$
By Claim  (1), there are $\zeta_{n,j}\in {\rm conv}(X_1)$   such that
\beq
{\rm dist}(t_{n,j}, \zeta_{n,j})<\ep/2. 
\eneq
It follows that
\beq
{\rm dist}(\xi_n^0, \sum_{j=1}^{m(n)}\bt_{n,j} \zeta_{n,j})<\ep/2.
\eneq
 Since  ${\rm conv}(X_1)$ is convex, $\sum_{j=1}^{m(n)} \bt_{n,j}\in {\rm conv}(X_1).$
 This proves Claim (2). 
 
 Hence, for all $n\ge N_1,$
 \beq
 {\rm dist}(x_n,  {\rm conv}({\rm conv}(X_1),  {\rm conv}(Y_0)))<\ep,
 \eneq
 Note that $ {\rm conv}({\rm conv}(X_1),  {\rm conv}(Y_0))={\rm conv}(Y_0).$
 However, $Y_0$ is compact. By Proposition \ref{P-L}, 
 $\overline{{\rm conv}(Y_0)}=M_{Y_0}\subset .$
 
 Therefore (since, by the assumption, $X_1\subset \partial_e(T)\setminus \cup_{n=1}^m X_{j+k}$ )
  \beq
 {\rm dist}(x_n,  {\rm conv}(\partial_e(T)\setminus \cup_{j=1}^m X_{j+k}))<\ep\rforal n\ge N_1.
 \eneq
It follows that $\xi\in \overline{ {\rm conv}(\partial_e(T)\setminus \cup_{j=1}^m X_{j+k}}.$
In other words, $T$ satisfies condition (C).
 %
\end{proof}
\fi
%%%%%%%%%%%%%%
%
%%%%%%%%%%%%%%%%%
\begin{cor}
\label{Padd-countable}
Let $T$ be a Choquet simplex with countable extremal boundary
$\partial_e(T).$
%=\{t_n\}_{n\in \N}$ ($t_i\not=t_j$ if $i\not=j$).
  Suppose that  there  is  a compact 
  subset $T_0\subset \partial_e(T)$ such that 
  %$\overline{\partial_e(T)}'\cap \partial_e(T)
  %\subset T_0$ and 
  $\overline{\partial_e(T)}'\subset \overline{{\rm conv}(T_0)},$
  where $\overline{\partial_e(T)}'$ is the set of cluster points of $\overline{\partial_e(T)}.$
  %which is finite or is
  %a convergent subsequence  together with its limit in 
  % $T_0=\{t_{n_k}\}$ or a finite subset  $T_0$ of 
  %$\partial_e(T)$ such that 
  %there are finitely many points $t_{n_1}, t_{n_2},...,t_{n_m}\in \partial_e(T)$  such that 
  %$\overline{\partial_e(T)}\setminus \partial_e(T)\subset 
%\overline{{\rm conv}(T_0)}$ and $\partial_e(T)\setminus T_0=\{t_m\}_{m\in N}$ consists of
 %isolated points. 
%(\{t_{n_1},t_{n_2},...,t_{n_m}\}}.$ 
Then $T$ satisfies condition (C).
\end{cor}

\begin{proof}
We may assume that  $\partial_e(T)=\{t_m\}_{m\in \N}\sqcup T_0.$ 
%$\overline{\partial_e(T)\setminus \partial_e(T)}\subset \overline{{\rm conv}(\{t_1, t_2,...,t_N\})}$ for some $N\in \N.$
By Corollary 11.19 of \cite{kG}, 
and Proposition \ref{P-L}, 
$\overline{{\rm conv}(T_0)}=M_{T_0}$ is a closed face and $\partial_e(M_{T_0})=T_0.$
Put $X_1=T_0$ and $X_{n+1}=\{t_n\},$ $n\in \N.$ 
We claim that each $t_n$ a relatively isolated point in $\partial_e(T).$
Otherwise, 
%Since any 
if $t_n$ is a cluster point of $\partial_e(T)$
for some $n,$ then it  is a cluster point of $\overline{\partial_e(T)}.$
By the assumption, $t_n\in M_{T_0}.$ But $t_n$ is an extremal point,  this would imply
$t_n\in T_0$ (as $M_{T_0}$ is a closed face). This proves the claim. 
%each $t_n$ is a relatively isolated point in $\partial_e(T).$
 Hence $X_{n+1}$ is relatively clopen for all $n\in \N.$ 
Moreover every point of $\overline{\partial_e(T)}\setminus \partial_e(T)$ is a cluster point of $\partial_e(T).$
Thus $\overline{\partial_e(T)}\setminus \partial_e(T)\subset \overline{{\rm conv}(T_0)}.$ 
%Since each $t_n$ is an isolated point, $X_{n+1}$ is a clopen set. 
%Note that $T_0$ is a compact subset of $\partial_e(T).$ 
%\{t_1, t_2,...,t_N\},$ and $X_n=\{t_{n+N}\},$ $n\in \N.$
We then apply Proposition \ref{Padd-compact}.
\end{proof}

Note that $T_0$ in \ref{Padd-countable} could be a union of finitely many convergent sequences together with their limit points (or a finite subset). 
 Let us mention  that  $\Q$ is not a $G_\dt$-set. Therefore, 
  $\Q$ with the usual topology cannot be realized as the extremal boundary of a (metrizable)
  Choquet simplex (see Cor. I.4.4 of \cite{Alf}).
  
  \begin{prop}\label{P52}
  Let $T$ be a Choquet simplex with $\partial_e(T)=\cup_{n=1}^\infty X_n,$
  where $X_n$ satisfies (1) and (2) in the condition (C).
  Then $\overline{\partial_e(T)}\setminus \partial_e(T)\subset 
  \overline{{\rm conv}(X_1)}.$
  %where each $X_n$ is compact and finite dimensional,  
  %$X_{1+n}$ is relatively clopen and $X_i\cap X_j=\emptyset,$
  %if $i\not=j,$ $i,j\ge 2.$
  \end{prop}
  
  \begin{proof}
  
  Let $\xi\in \overline{\partial_e(T)}\setminus \partial_e(T).$
Then there exists a sequence $x_n\in \partial_e(T)$ such that
$x_n\to \xi.$  Since, for each $m,$   $\cup_{j=2}^m X_j$ is compact, 
we may assume that $x_n\in X\setminus \cup_{j=2}^{m(n)} X_j$
for some $m(n)\to\infty$ as $n\to\infty.$
It follows that, for any $m\ge 2,$ $\xi\in \overline{X\setminus \cup_{j=2}^m X_j}.$
Hence, by condition (C),
\beq
\xi\in \cap_{m=1}^\infty (\overline{{\rm conv}(X\setminus \cup_{j=2}^m)})=\cap_{m=1}^\infty M_{X\setminus
\cup_{j=2}^m X_j}.
\eneq
Then, for each fixed $m,$ $\mu_\xi(\sum_{j=2}^m X_j)=0.$ 
It follows that $\mu_\xi(\cup_{j=2}^\infty X_j)=0.$ Hence
$\xi\in M_{X_1}\subset \overline{{\rm conv}(X_1)}.$

\end{proof}

Next we would like to provide more and concrete  examples of Choquet simpleces $S$  which 
satisfy condition (C) but are not Bauer simplexes. 

%each of which 
%is not a Bauer simplex  
%and $\partial_e(S)$ has  countably many points.
%only finitely many limit points.
%Concrete examples of \CA s with non-Bauer simpleces 
%as their tracial spaces brought to our close attention 
%are those tracial spaces from one-dimensional NCCW complexes (see 3.8 of \cite{GLNI}). 
%These can be ``countablized"  as follows.

\begin{exm}\label{exm-1}

(1) This example is the same as  Example 3.3 of \cite{CETW}.
Denote by ${\bf c}$ the \CA\,  of convergent sequences.
Let $X=\{0\}\cup \{1/k\}_{k\in \N}\subset [0,1].$
We may view ${\bf c}=C(X)$ as the algebra of continuous functions on $X.$ 
%on $\N$ that converges at $\infty.$ If $x\in {\bf c},$ $x_\infty$ stands for the limit of $x$ at infinity.
Let 
\beq
E_4=\{x\in {\bf c}\otimes M_2: x(0)=\begin{pmatrix} a & 0\\
                0& b\end{pmatrix},\, a,b\in \C\}.
\eneq
This is one of the examples presented by L. G. Brown in \cite{Brs} (p.~868) to demonstrate different 
phenomena. 
Since $E_4$ is unital, $S=T(E_4)$ is a Choquet simplex. 
Let $\tau_n: E_4\to \C$ be defined by $\tau_n(x)=t_{M_2}(x(1/n))$ for all $x\in E_4,$ $n\in \N.$
Set $\tau^+(x)=a$ and $\tau_-(x)=b$ (if $x(0)=\diag(a, b)$).  \CA\, 
$E_4$ is used in  Example 3.3 of \cite{CETW} to show, among other things, that its tracial state space 
is not a Bauer simplex. As in \cite{CETW},  
%\beq
$\partial_e(S)=\{\tau_n: n\in \N\}\cup\{\tau^+, \tau^-\}.$
%\eneq
Note  (as in 3.3 of \cite{CETW}) that $\{\tau_n\}_{n\in \N}$ converges to $(1/2)(\tau^++\tau^-).$ 
So $\partial_e(S)$ is not closed. 
Let $X_1=\{\tau^+, \tau^-\}$ and $X_{n+1}=\{\tau_n\},$ $n\in \N.$ 
Then it is easy to see that $S$ satisfies condition (C) (see Proposition 
\ref{Padd-countable}) but is not a Bauer simplex.
%
\iffalse
Let $E_4^{(n)}$ be the direct sum of $n$ copies of $E_n.$ Then $T=T(E_4^{(n)})$ 
is a Choquet simplex such that $\overline{\partial_e(T)}$ has exactly $n$ limit points 
and none of them is in $\partial_e(T).$ There are other examples of Choquet simplices $S$
which are not Bauer simplices  with $n$ many limit points in $\overline{\partial_e(S)}.$
\fi

\iffalse
Define 
\beq
B=\{x\in {\bf c}\otimes E_4: x(0)=\begin{pmatrix} a & 0\\
                0& a\end{pmatrix},\, a\in \C\}.
\eneq
Put $I=c_0\otimes E_4=C_0(X, E_4).$   Then $B=\wtd I.$
Since $B$ is unital, $T(B)$ is a Choquet simplex.

Denote by $\tau_0$ the trace of $B$ which vanishes on $I.$
We may write 
$$\partial_e(T(B))=\tau_0\sqcup \cup_{n=1}^\infty \{1/k\}\times \partial_e(T(E_4)).$$
Note that  
%$T(B)$ is a Choquet simplex and 
$\partial_e(T(B))$ has only countably many points but has 
infinitely many limit points and is not compact.

Let $X_1=\{\tau_0\},$ $X_{1+1}=\{1\}\times \{\tau^+, \tau^-\},$ 
$X_{1+2}=\{1\}\times \{\tau_1\},$ $X_{1+3}=\{1/2\}\times \{\tau^+, \tau^-\},$
$X_{1+4}=\{1/2\}\times \{\tau_1\},$ $X_{1+5}=\{1\}\times \{\tau_2\},$
$X_{1+6}=\{1/3\}\times \{\tau^+, \tau^-\},$ $X_{1+7}=\{1/3\}\times \{\tau_1\},$
$X_{1+8}=\{1/2\}\times \{\tau_2\},$ $X_{1+9}=\{1\}\times \{\tau_3\},$ 
$X_{1+10}=\{1/4\}\times \{\tau^+,\tau^-\},....$ 
\fi
%%%%%%%%%%%%%%%%%

(2) This is a slight modification of example in (2). 
Let 
\beq
E_4'=\{x\in {\bf c}\otimes M_3: x(0)=\begin{pmatrix} a & 0 &0\\
                0& b& 0\\
                0&0& c\end{pmatrix},\, a,b\in \C\}.
\eneq
Let $\tau^1(x)=a$ for all $x\in E_4',$ $\tau^2(x)=b$ for all $x\in E_4'$ and 
$\tau^3(x)=c$ (if $x(0)=\diag(a, b,c)$). 
Define $\tau_n(x)=t_{M_3}(x(1/n))$ for all $x\in E_4'.$ 
Define $X_1=\{\tau^1, \tau^2, \tau^3\},$ $X_n=\{\tau_n\},$ $n\in \N.$ 
Note that $\tau_n$ converges $(1/3)(\tau^1+\tau^2+\tau^3).$

One notes that, as topological spaces, $\overline{\partial_e(E_4')}$ is homeomorphic to 
$\overline{\partial_e(E_4)}$ (which is homeomorphic to $\{0\}\cup\{1/k\}_{k\in \N}$).
 Moreover, a homeomorphism can be given such that whose restriction 
on $\partial_e(E_4')$ gives a homeomorphism from 
$\partial_e(E_4')$ onto $\partial_e(E_4).$ 
 Both $T(E_4)$ and $T(E_4')$ satisfy condition (C). 
 However $T(E_4)$ and $T(E_4')$ have different affine structure.

(3) We now consider a \SCA\, $B$ of 
${\bf c}\otimes E_4.$

Let $X=\{0\}\cup \{1/k\}_{k\in \N}$ be as a subset of $[0,1].$ 
We may write ${\bf c}\otimes E_4=C(X, E_4).$ 
Let $D_0=\{g\in C(X, M_2): g(x)=g(0)\in \begin{pmatrix} a & 0\\
                0& b\end{pmatrix},\, a,b\in \C\, \rforal x\in X\}.$
               % By identifying elements in $D_{00}$ with constant 
                %diagonal matrices, 
                We may viewed $D_0$ as a unital \SCA\, of $E_4.$
 %               
%
%Consider 
%$D_0\subset C(X, E_4)$ 
%\beq
%D_0=\{f\in C(X, E_4): f(x)=f(0)\in D_{00}, \, \rforal x \in X\}.
%\eneq
%(
So $D_0$ is the \SCA\,consisting of  constant diagonal $2\times 2$ matrices
and $D_0\cong \C\oplus \C.$

%In the next formula, we identify each $x\in {\bf c}\otimes E_4$ as 
%an $E_4$-values of function defined on $0\cup \{1/k\}_{k\in \N}.$
%For example, $x(t)\in E_4$ for each $t\in 0\cup \{1/k\}_{k\in \N}.$  If we  write $x(0)=\begin{pmatrix} a & 0\\
 %               0& b\end{pmatrix},$ we mean that $x(0)$ is a constant matrix in $E_4$
 %               (which is a \SCA\, of ${\bf c}\otimes M_2$).
Define 
\beq
B=\{x\in C(X, E_4): x(0)\in D_0\}.
%=\begin{pmatrix} a & 0\\
%                0& b\end{pmatrix},\, a,b\in \C\}.
\eneq
Since $B$ is unital, $T(B)$ is a Choquet simplex. 
Put $I=c_0\otimes E_4.$ We may view $I$ as an ideal of $B.$ Then $B/I\cong \C\oplus \C.$
Note that  
%$T(B)$ is a Choquet simplex and 
$\partial_e(T(B))$ has only countably many points but has 
infinitely many limit points and is not compact.
%Put $X=\{\infty\}\cup \{k\}_{k\in \N}.$ Let us identify ${\bf c}$ with $C(X)$ 
%and 
We may write 
$$\partial_e(T(B))=\{0\}\times \{\tau^+, \tau^-\}\sqcup \cup_{n=1}^\infty \{1/k\}\times \partial_e(T(E_4)).$$
Note that, for each fixed  $n,$  $\{1/k\}\times \tau^+\to \{0\}\times \tau^+$ 
and $\{1/k\}\times \tau^-\to \{0\}\times \tau^-,$  as $k\to\infty,$ respectively.
Moreover,  for fixed $m\in N,$
$\{1/k\}\times \{1/m\}\to \{0\}\times  (1/2)(\tau^++\tau^1)$ as $k\to \infty,$
$\{1/k\}\times \{1/m\}\to \{1/k\}\times (1/2)(\tau^++\tau^1)$ as $m\to\infty,$ 
and, for any subsequence $\{m_k\},$\\
$\{1/k\}\times \{1/m_k\}\to \{0\}\times  (1/2)(\tau^++\tau^1)$ as $k\to\infty.$
Define 
$X_1=X
%\{0\}\times \partial_e(T(E_4))\cup(\cup_{k=1}^\infty 
 \times \{\tau^+, \tau^-\}.$
Then $X_1$ is compact and zero dimensional. 
Put $Y_{k,n}=\{1/k\}\times \{\tau_n\}.$   There is a bijection $s: \N\to \{(k,n): k,n\in \N\}.$
Define $X_{n+1}=Y_{s(n)},$ $n\in \N.$ Then $X_{n+1}$ is compact and zero dimensional,
and $\partial_e(T(B))=\cup_{n=1}^\infty X_n.$
Moreover, $X_{1+n}$ is relatively open, and 
$\overline{\partial_e(T(B))}\setminus \partial_e(T(B))\subset \overline{{\rm conv}(X_1)}.$
By 
%Corollary \ref{Padd-countable},}}
Proposition \ref{Padd-compact} (or Corollary \ref{Padd-countable}), 
%
\iffalse
Denote by $E_{4,n}=\{f\in E_4: f(1/j)=0\rforal j\not=n\}.$   Then $E_{4,n}\cong M_2$ and is 
a direct summand of $E_4.$  Moreover, $E_4/E_{4,n}\cong E_{4}.$
%Now we write $B=C(X, E_4).$ 
For each $i,j\in \N,$
\beq
B_{i,j}=\{f\in B: f(1/k)=0,\,\, k\not=i, f(1/i)\in E_{4,j}\}\cong E_{4,j}\cong M_2.
\eneq
{{Hence, for any $i\in \N,$ $B_{s(i)}=B_{i', j'}$ for some $i', j'.$ It is a direct summand 
of $B.$   We may write $\partial_e(T(B_{s(i)}))=Y_{s(i)}=X_{i+1}.$
%
We note that $B/B_{s(i)}\cong B.$ 
In fact, for any $k, i\in \N,$}} 
$\bigoplus_{i=1}^m B_{s(i+k-1)}$ is  isomorphic to a finite direct sum of $E_{4,j}.$  
These finite direct sums are also ideals of $B.$ 
Moreover,  $B/\bigoplus_{i=1}^m B_{s(i+k-1)}\cong B.$ 
Hence we may write  (note that $i, k\ge 1$)
\beq
B=B_0\bigoplus (\bigoplus_{i=1}^m B_{s(i+k-1)}),
\eneq
where $B_0\cong B.$ 
Note that, for any \CA\, $C,$ $\overline{{\rm conv}(\partial_e(T(C)))}=M_{\partial_e(T(C))}.$
From this, it is easy to see that 
for $k, m\in \N,$  
\beq
%\overline{{\rm conv}(\cup_{i=1}^mX_{i+k})}=M_{\cup_{i=1}^mX_{i+k}}
%\andeqn\\
\overline{{\rm conv}(X\setminus \cup_{i=1}^mX_{i+k})}=M_{X\setminus \cup_{i=1}^mX_{i+k}}
\eneq
\fi
 $T(B)$ satisfies condition (C) but not a Bauer simplex.

%(see the last paragraph 

%
\iffalse
(3)  Let $\{K_n\}$ be a sequence of compact metric spaces with finite covering dimension.
Let $X$ be the one point compactification of disjoint union of $\{K_n\}.$ 
Put $D_1=C(X).$ Then $\partial_e(T(D_1))=\cup_{n=1}^\infty X_n,$ 
where $X_1={\rm the\,\, infinite\,\, point}=\{\xi_\infty\}$ and $X_{n+1}=K_n,$ $n\in \N.$ 
It is easy to {\red{see}} that $T(D_1)$ satisfies condition (C). 
We would like to mention that  we may choose each $K_n$  the $n$-dimensional cube,
or the $n$-dimensional spheres, 
$n\in \N,$  for example. Then $X$ is compact and 
has countable (but not  necessarily finite) dimension. 
\fi
%%%%%%%%%%
(4) 
Let  $X$ be as in (v) of Remark of \ref{Rmark1}. 
Define  
\beq
D_2=\{f\in C(X, M_2):  f(\xi_\infty)=\begin{pmatrix} a & 0\\
                0& b\end{pmatrix},\, a,b\in \C\}.
\eneq
Then $D_2$ is a unital \CA,  $T(D_2)$ is a Choquet simplex and $\partial_e(T(D_2))=\{\tau^+, \tau^-\}\sqcup (\cup_{n=1}^\infty K_n)$ which is not compact.
%In fact, for any sequence $x_m\in \cup_{n=1}^\infty K_n$ such that 
%$x_m\to\xi_{\infty}$ as points  in $X.$ 
%
Put $X_1=\{\tau^+, \tau^-\}$ and $X_{n+1}=K_n,$ $n\in \N.$  
Then $\partial_e(T(D_2))=\cup_{n=1}^\infty X_n,$ where each $X_n$ is   compact   and has finite covering dimension. Moreover $X_{1+n}$ is relatively open for $n\in \N.$
Note that $\overline{\partial_e(T(D_2))}\setminus \partial_e(T(D_2))=\{(1/2)(\tau^++\tau^-)\}\subset {\rm conv}(X_1).$
%Similar to what we discussed in (1) above, 
By Proposition \ref{Padd-compact}, $T(D_2)$ satisfies condition (C) and is not a Bauer simplex.
Moreover $T(D_2)$ is not finite dimensional, if $\{{\rm dim} X_n\}$ is not bounded.

(5) Let $Y$ be a connected  and locally connected compact metric space (these conditions
are for convenience not necessarily needed) with finite covering dimension  and $s: [0,1]\to Y$ 
be a continuous surjective map (by Hahn--Mazurkiewicz Theorem). Then we obtain a unital embedding $j_Y: C(Y)\to C([0,1]).$
%(by Hahn--Mazurkiewicz Theorem). 
%Let $\mu_Y$ be the strictly positive Borel probability measure given by 
%$\iota$ and the Lesbegue  measure on $[0,1].$  
Let $Q$ be the UHF-algebra with 
$(K_0(Q), K_0(Q)_+, [1_Q])=(\Q, \Q_+, 1)$ and $j_I: C([0,1])\to Q$ be  a unital embedding.
We then obtain a unital embedding $\iota_Y:=j_I\circ j_Y: C(Y)\to A.$ 
Put $C=\iota_Y(C(Y)).$ Let $\tau_Q$ be the unique tracial state of $Q.$ 
Define $\tau_Y: C(Y)\to \C$ by $\tau_Y(f)=\tau_Q\circ \iota_Y(f)$ for all $f\in C(Y).$ 

Let $X$ be as in (v) of Remark of \ref{Rmark1}. We may choose $X$ so that it  is not finite dimensional. 
Similar to  (4) above, define 
\beq
D_3=\{ f\in C(X, Q): f(\xi_\infty)\in C\}.
\eneq
Then $D_3$ is unital and $T(D_3)$ is a Choquet simplex.

Let $\pi_\infty: D_3\to C\cong C(Y)$ be the quotient map and $y\in Y.$
Then $y$ is identified with $\tau_y: D_3\to \C$ defined 
by $\tau_y(f)=\pi_\infty(f)(y)$ for all $f\in D_3.$
Let $X_1=Y,$ $X_{n+1}=K_{n},$ $n\in \N.$ Then $\partial_e(T(D_3))=\cup_{n=1}^\infty X_n$
and $X_i\cap X_j=\emptyset,$ if $i\not=j$ and $i,j\in \N.$ 
Note that $X_{1+n}$ is also relatively open.

Choose, for each $n,$  $x_n\in X_{n+1}.$ 
Here $x_n$ is identified with the tracial state $\tau_n: D_3\to \C$
by $\tau_n(f)=\tau_Q(f(x_n))$ for $f\in D_3,$ $n\in \N.$ 
Then we have $x_n\to \tau_C,$ where 
$\tau_C: D_3\to \C$ is defined by $\tau_C(f)=\tau_Q\circ f(\xi_\infty)$ for $f\in D_3.$
Note that $\tau_C\in \overline{{\rm conv}(X_1)}$  and $\tau_C\not\in \partial_e(T(D_3)).$
In particular, $T(D_3)$ is not a Bauer simplex. However, 
$\tau_C\in \overline{{\rm conv}(Y)}.$  In fact 
$\{\tau_C\}=\overline{{\rm conv}(Y)}\setminus {\rm conv}(Y)\subset \overline{{\rm conv}(X_1)}.$
Thus, by Proposition \ref{Padd-compact},  $\partial_e(T(D_3))$ 
%To see that  $Z:=\partial_e(T(D_3))$ 
satisfies condition (C).
%
%%%%%%%%%%%%
\iffalse
fix $k\in \N.$ We need to show that, for any $m\in \N,$ 
$\overline{{\rm conv}(Z\setminus \cup_{j=1}^m X_{j+k})}=M_{Z\setminus \cup_{j=1}^m X_{j+k}}.$
Let $z_n\in {\rm conv}(Z\setminus \cup_{j=1}^m X_{j+k})$ such 
that $z_n\to z.$  By  the Choquet  theorem, there is a probability Borel measure 
$\mu_z$ on $Z$ such that
\beq
f(z)=\int_Z fd\mu_z\rforal f\in \Aff(Z).
\eneq 
Write $\mu_z=\af\cdot \mu_1+(1-\af) \mu_2,$ where $0\le \af\le 1,$ 
$\mu_1$ is concentrated on $\cup_{j=1}^m X_{j+k}$ and 
$\mu_2$ is concentrated on $Z\setminus \cup_{j=1}^m X_{j+k}.$ 
Suppose that $\af>0.$ 

Let $g\in C(X, Q)$ be such that $g(x)=1_Q$ for all $x\in {\cup_{j=1}^m X_{j+k}}$ and 
$g(x)=0$ for all $x\in X\setminus \cup_{j=1}^m X_{j+k}.$ 
Then  $g\in A_+.$
Put $f(\tau)=\tau(g)$ for all $\tau\in T(D_3).$ Then $f\in \Aff(T(D_3))=\Aff(Z)$  and 
%We may view $f$ as an element in $\Aff(T(D_3)).$
%Then
\beq
f(z_n)\to f(z).
\eneq
In this case $f(z_n)=0$ for all $n\in \N.$  But 
\beq
f(z)&=&\af \int_{\cup_{j=1}^m X_{j+k}} f d\mu_1+(1-\af)\int_{Z\setminus \cup_{j=1}^m X_{j+k}}
d \mu_2\\
&=&\af\int_{\cup_{j=1}^m X_{j+k}} f d\mu_1=\af>0.
\eneq
This contradict to the fact that $f(z_n)\to f(z).$
Hence $\af=0$ which implies that
\beq
{\overline{{\rm conv}(Z\setminus \cup_{j=1}^m X_{j+k}})}=
M_{Z\setminus \cup_{j=1}^m X_{j+k}}.
\eneq
Thus $\partial_e(T(D_3))$ satisfies 
condition (C), and, has infinite covering dimension  and is not a Bauer simplex. 
\fi

%%%%%%%%%%%%%%%%%%%%%%%%%%

%
%%%%%%%%%%%%%%%%%%%%%%%%%%%%%%%%%%%%%
(6) Consider a unital \SCA\, $D_3^d$ of $D_3$  as follows
\beq
D_3^d=\{f\in C(X, Q): f(x)=f(\xi_\infty)\in C\rforal x\in X\}.
\eneq
Note $D_3^d\cong C\cong C(Y).$
Let $\{F_n\}$ be a sequence of finite dimensional compact metric spaces.
Let $F$ be the one-point compactification of disjoint union of $\{F_n\}.$
Denote by $\zeta_\infty$ the one point. 
Define 
\beq
D_4=\{g\in C(F, D_3): f(\zeta_\infty)\in D_3^d\}.
\eneq
$D_4$ is a unital \CA\, and $T(D_4)$ is a Choquet simplex.
Let $J=\{g\in C(F, D_3): g(\zeta_\infty)=0\}.$ One has the following short exact sequence:
\beq
0\to J\to D_4\to D_3^d\cong C(Y)\to 0.
\eneq
One then computes that
\beq
\partial_e(T(D_4))&=&(\cup_{n=1}^\infty F_n\times \partial_e(T(D_3)))\cup \{\zeta_\infty\}\times Y\\
&=&
(\cup_{n, m=1}^\infty F_n\times X_m)\cup (\cup_{n=1}^\infty F_n\times Y)\cup \{\zeta_\infty\}\times Y\\
&=&(\cup_{n, m=1}^\infty F_n\times X_m)\cup (F\times Y).
\eneq 
%Note that, if $x_n\in \cup_{m=1}^\infty F_m,$ and $x_n\to \zeta_\infty,$
%and, if $y_n\in Y,$ and $y_n\to y\in Y,$ 
%then, $x_n\times y_n\to \zeta_\infty\times y\subset \{\zeta_\infty\}\times Y.$
Let $r: N\to \N\times \N$ be a bijection. Let us write $r(n)=(r_1(n), r_2(n))$ for $n\in \N.$ 
Put $\Omega_1=F\times Y, $ $\Omega_{1+n}=F_{r_1(n)}\times  X_{r_2(n)+1},$ $n\in \N.$
%and $n\to r_1(n)\times r_2(n)$ gives a bijection from $\N$ to $\N\times \N.$ 
Note that $\Omega_n$ is compact for each $n\in \N,$
$\partial_e(T(D_4))=\cup_{n=1}^\infty \Omega_n,$ 
$\Omega_i\cap \Omega_j=\emptyset,$ if $i\not=j$ and $i,j\in \N.$ 
Each $\Omega_{1+n}$ is also relatively open, $n\in \N.$
%Moreover 
%one can show that 
%$\partial_e(T(D_4))$ satisfies condition (C).  

Note that, for each $x\in \cup_{n=1}^\infty F_n,$  
$(x\times \tau_C)(g)=\tau_Q(g(x)(\xi_\infty))$ for all $g\in D_4.$
If $x_k\in \cup_{n=1}^\infty F_n$ such that $x_k\to \xi\in F$ and $z_k\in X_{1+k},$  $k\in \N,$
then $x_k\times z_k\to \xi \times \tau_C.$ 
%Hence $T(D_4)$ is not a Bauer simplex. 
However
\beq
\xi \times \tau_C\in\overline{\partial_e(T(D_4))}\setminus \partial_e(T(D_4))
\eneq
(for any $\xi \in F$).
%\cup_{n=1}^\infty F_n$).
Note also, for any $f\in D_4,$ 
$f(x_n)\to f(\zeta_\infty),$ if $x_n\in F$ and $x_n\to \zeta_\infty.$
It follows that, for any $z\in \partial_e(T(D_3)),$ $x_n\times z\to \zeta_\infty\times \tau_C.$
Hence $T(D_4)$ is not a Bauer simplex. 
In fact $\overline{\partial_e(T(D_4))}\setminus \partial_e(T(D_4))=F\times \{\tau_C\}$
which is a subset of $\overline{{\rm conv}(F\times Y)}=\overline{{\rm conv}(\Omega_1)}.$
Hence, by Proposition \ref{Padd-compact},  $T(D_4)$ satisfies condition (C). 
In this case 
$\overline{\partial_e(T(D_4))}\setminus \partial_e(T(D_4))$
may contains interesting topological spaces.
% ($\cup_{n=1}^\infty F_n$).
% contains 
%infinitely many points. 

%%%%%%%%%%%%
\iffalse
(6) Similar to the construction of $B,$ one may also have a Choquet simplex $T$ which satisfies 
condition (C), $\overline{\partial_e(T)}\setminus \partial_e(T)$ has infinitely many points and $\partial_e(T)$ 
is not finite dimensional. If instead of  attaching  two points to $\cup_{n=1}^\infty X_n,$ 
we may attach a finite dimensional compact metric space $Y$ to $\cup_{n=1}^\infty X_n.$
In that way, some more interesting examples of Choquet {{simpleces}} satisfying condition (C) can be constructed. 
\fi
%%%%%%%%

(7) %To see the limitation of Corollary \ref{CCM-1}, we would like to present the following 
%example:
%Let $Q$ be the UHF-algebra with $(K_0(Q), K_0(A)_+, [1_Q])=(\Q, \Q_+, 1).$
%Denote by $\tau_Q$ the unique tracial state of $Q.$
We will modify the example (3) and 
% of \ref{exm-1}.
 identify $M_2(Q)$ with $Q.$
Denote by $j_1: M_2(Q)\to Q$ the isomorphism such that $j_1(\diag(1, 0))=e_1\in Q,$ where 
$\tau_Q(e_1)=1/2.$ 
%and $j_2: M_4(Q)\to Q$ 
% e_1Qe_1\to Q$ the isomorphism 
%(such that $\tau_Q(j_2(p))=2\tau_Q(p)$ for all projection $p\in e_1Qe_1$).
%In what follows by $Q\oplus Q\subset Q$ we mean 
%an embedding 
Let 
$\iota_1: Q\oplus Q\to Q$  be defined by
$\iota_1(a\oplus b)=j_1(\diag(a, b))$ for $a\oplus b\in Q\oplus Q.$
%Let 
%=\diag(a, b)\subset M_2(Q).$

Let $X=\{0\}\cup \{1/k\}_{k\in \N}\subset [0,1]$ and 
\beq
C=\{f\in C(X, Q): f(0)\in \iota_1(Q\oplus Q)\}.
\eneq
Let $\pi^+: \iota_1(Q\oplus Q)\to Q$ be defined by
$\pi^+(\iota_1(a\oplus b))=a$ and $\pi^-: \iota_1(Q\oplus Q)\to Q$ by 
$\pi^-(\iota_1(a\oplus b))=b$ for all $a, b\in Q,$ respectively.
Denote by $\tau^+: C\to \C$ the tracial state on $C$ defined 
by $\tau^+(f)=\tau_Q(\pi^+(f(0)))$ and by $\tau^-: C\to \C$ the tracial state on $C$ 
defined by $\tau^-(f)=\tau_Q(\pi^-(f(0))$
for all $f\in C,$ respectively. 
Let $D\subset C$ be such that
\beq\nonumber
D=\{f\in C(X, Q): f(x)=f(0)\in  \iota_1(\iota_1(Q\oplus Q)\oplus \iota_1(Q\oplus Q))\rforal x, \in X\}\subset C.
\eneq
We will identify $D$ with $\iota_1(\iota_1(Q\oplus Q)\oplus \iota_1(Q\oplus Q))\cong Q\oplus Q\oplus Q\oplus Q.$
Define 
\beq
A=\{g\in C(X, C): g(0)\in D\}.
\eneq
Note that, for each $x\in X,$ $g(x)\in C.$ 
For $j, k\in \N,$  define  $\tau_{1/j,1/k}: A\to \C$ by $\tau_{1/j, 1/k}(g)=\tau_Q(g(1/j)(1/k))$ for 
$g\in A.$  For $j\in \N,$  define $\tau_{1/j}^+$ by
$\tau_{1/j}^+(g)=\tau^+(g(1/j))$  and $\tau_{1/j}^-$ by
$\tau_{1/j}^-(g)=\tau^-(g(1/j))$ 
 for all $g\in A.$   Define 
 $$
 \pi^+_+: \iota_1(\iota_1(Q\oplus Q)\oplus \iota_1(Q\oplus Q))\to Q
\,\, {\rm by}\,\,\,\pi^+_+(\iota_1(\iota_1(a\oplus b)\oplus \iota_1(c\oplus d)))=a,
$$
 define $\pi^+_-$ by $\pi^+_-((\iota_1(\iota_1(a\oplus b)\oplus \iota_1(c\oplus d)))=b,$
 define $\pi^-_+$ by $\pi^-_+((\iota_1(\iota_1(a\oplus b)\oplus \iota_1(c\oplus d)))=c,$
 define $\pi^-_-$ by $\pi^-_-((\iota_1(\iota_1(a\oplus b)\oplus \iota_1(c\oplus d)))=d$
 for all $a, b, c,d\in Q,$ respectively. 
 Define $\tau^+_+: A\to \C$ by $\tau^+_+(g)=\tau_Q\circ \pi^+_+(g(0)),$
 %Define $\tau^+_+: A\to \C$ by 
 %$\tau^+_+(g)=\tau_Q(\pi^+_+(g(0))),$ 
 $\tau^+_-: A\to \C$ by $\tau^+_-(g)=\tau_Q(\pi^+_-(g(0))),$
 $\tau^-_+:A\to \C$ by $\tau^-_+(g)=\tau_Q(\pi^-_+(g(0)))$ and
 $\tau^-_-: A\to \C$ by $\tau^-_-(g)=\tau_Q(\pi^-_-(g(0)))$
 for all $g\in A,$ respectively. 
 
 Then 
 \beq
 \partial_e(T(A))=\{\tau_{1/j,1/k}: j,k\in \N\}\cup\{\tau_{1/j}^+, \tau_{1/m}^-
 : j,m\in\N\}
 \cup\{\tau^+_+, \tau^+_-, \tau^-_+, \tau^-_-\}.
 \eneq

We compute that, for each fixed $j\in \N,$ $\tau_{1/j, 1/k}\to (1/2)(\tau_{1/j}^++\tau_{1/j}^-),$ as 
$k\to\infty,$ 
for each fixed $k\in \N,$ $\tau_{1/j, 1/k}\to (1/4)(\tau^+_++\tau^+_-+\tau^-_++\tau^-_-)$  as $j\to \infty,$
and, $\tau_{1/j, 1/k}\to (1/4)(\tau^+_++\tau^+_-+\tau^-_++\tau^-_-)$  as $j,k\to \infty.$
Moreover,
$\tau_{1/j}^+\to (1/2)(\tau^+_++\tau^+_-),$ as $j\to\infty,$ 
and  $\tau_{1/j}^-\to (1/2)(\tau^-_++\tau^-_-).$

We claim that $\partial_e(T(A))$ does not satisfy condition (C).
Note that each point in $\partial_e(T(A))$ is a relatively isolated point and hence any infinite 
subset of $\partial_e(T(A))$ has a limit point in $T(A)\setminus \partial_e(T(A))$
(recall that $T(A)$ is compact as $A$ is unital).
So every compact subset of $\partial_e(T(A))$ is finite. 

Suppose that $\partial_e(T(A))=\cup_{n=1}^\infty X_n$
such that  each $X_n$ is compact, and 
$X_i\cap X_j=\emptyset,$ if $i\not=j$ and $i,j\ge 2.$ 
We have shown that each $X_n$ has finitely many points. 
Therefore, we may choose  $j_0\in \N$ such that $\tau_{1/j_0}^+\not\in X_1.$
There are some $k, m\in \N$ such that $\tau_{1/j_0}^+\in \cup_{j=1}^mX_{j+k}.$
Hence $\tau_{1/j_0}^+\not\in X\setminus \cup_{j=1}^mX_{j+k}.$
Since $\cup_{j=1}^mX_{j+k}$ has only finitely many points, 
$X\setminus \cup_{j=1}^mX_{j+k}$ contains infinitely many 
points with the form $\tau_{1/j_0, 1/k}.$ 
Hence, $(1/2)(\tau_{1/j_0}^++\tau_{1/j_0}^-)\in \overline{{\rm conv}(X\setminus \cup_{j=1}^mX_{j+k})}.$
Since $\tau_{1/j_0}^+$ is an extremal point which is not in $X\setminus \cup_{j=1}^mX_{j+k},$
we have 
$(1/2)(\tau_{1/j_0}^++\tau_{1/j_0}^-)\not\in M_{X\setminus \cup_{j=1}^mX_{j+k}}.$
In other words, $\partial_e(T(A))$ does not satisfy condition (C).
Note that $\partial_e(T(A))$ is countable.

However, if we defined 
$D=\{f\in C: f(x)=f(0)\in \iota_1(Q\oplus Q)\, \rforal x\in X\}\subset C$ 
and 
\beq
B_1=\{g\in C(X, C): g(0)\in D\},
\eneq
then   $T(B_1)$ is affinely homeomorphic to $T(B)$
(the same $B$ as in Example (3) above which satisfies condition (C). 
%
%%%%%%%%%%%%%%%%%%
\iffalse
(8) Let $T$ be a Choquet simplex with countable extremal boundary.
Let $\partial_e(T)=\{t_n\}_{n\in \N}$ ($t_i\not=t_j$ if $i\not=j$).
  Suppose that $\overline{\partial_e(T)}\setminus \partial_e(T)\subset 
\overline{{\rm conv}(\{t_1,t_2,...,t_N\}}$ for some $N\in \N.$
Then $T$ satisfies condition (C).  To see this, define $X_1=\{t_1, t_2,...,t_N\},$
$X_{1+n}=\{x_{n+N}\}, $ $n\in \N.$  Then $X_i\cap X_j=\emptyset,$ if $i\not=j$ and 
$i,j\ge 2.$ Moreover, for any $k, m\in \N,$
$\partial_e(T)\setminus \cup_{j=1}^m X_{j+k}.$
\fi
%%%%%%%%%%%%%%%

(8) \CA s $E_4,$ $E_4',$  $B,$  and  $D_2, D_3$  and $D_4$  are not simple. 
From \cite{Btrace}, any metrizable Choquet  simplex  $T$
%As mentioned earlier 
can be realized as a tracial state space of a unital simple AF-algebra.
% (see \cite{Btrace}). 
In fact $T$ can also be realized as a tracial state space of a unital (or stably projectionless) 
 simple \CA\, $A$ which has  arbitrary  $K_1(A)$-group (and any compatible $K_0(A)$)
 (see, for example, \cite{GLNI}, \cite{eglnkk0} and \cite{GLIII}). 
%%%%%%%%%%%

\end{exm}

\begin{lem}\label{Lclaimaff}
%Let $A$ be a separable non-elementary simple \CA\,  with $A={\rm Ped}(A)$ 
Let $T$ be a Choquet simplex such that
 (1) $\partial_e(T)=Y\cup Z$ and $Y\cap Z=\emptyset,$ 
 (2)   
 $Y$ and $Z$ are  Borel and,
 (3)   $M_Y=
 \overline{{\rm conv}(Y)}$ and  $M_Z= \overline{{\rm conv}(Z)}.$

Suppose that $x_n=\af_n y_n+(1-\af_n) z_n\in T$ with $\af_n\in [0,1]$ and $y_n\in M_Y$
and $z_n\in M_Z,$ $n\in \N,$ and if 
$x_n\to \af y+(1-\af)z\in T,$ where $\af\in [0,1],$ $y\in M_Y$ and $z\in M_Z,$
then $\af_n\to \af, $ $y_n\to y$ and $z_n\to z.$ 
\end{lem}

\begin{proof}
%Let $M_Y$  and $M_Z$ be the sets of probability Borel measure on $\partial_e(T)$
%concentrated on $Y$ and $Z,$ respectively.  
By the Choquet theorem, we identify
$M_Y$ and $M_Z$ as convex subsets of $T.$
By (1) and (3),
% which are convex subsets.
$T=\{\af \mu_y+(1-\af)\mu_z: 0\le \af\le 1,\,\, \mu_y\in M_Y\andeqn \mu_z\in M_Z\}.$
Let $t\in T$ be represented by the Borel probability  measure 
$\mu_t.$ Then $\mu_t=\af \mu_y+(1-\af)\mu_z$ for some  $\mu_y\in M_Y,$  $\mu_z\in M_Z$
and $\af\in [0,1].$ 
In other words, $\mu_t(Y)=\af$ and $\mu_t(Z)=(1-\af).$ 
By the Choquet  theorem, this decomposition is unique.
(It is also helpful to note that (1) and (3) imply that $\overline{Y}\cap \overline{Z}=\emptyset.$)

%Clearly ${\rm conv}(Y)\subset M_Y$ and ${\rm conv}(Z)\subset M_Z.$
%An elementary measure theory argument  shows that 
%$M_Y\subset \overline{{\rm conv}(Y)}$ and $M_Z\subset \overline{{\rm conv}(Z)}.$ 

Now suppose that $x_n=\af_n y_n+(1-\af_n) z_n\in T$ with $\af_n\in [0,1]$ and $y_n\in M_Y=\overline{{\rm conv}(Y)}$
and $z_n\in M_Z=\overline{{\rm conv}(Z)},$ $n\in \N,$ and  
$x_n\to \af y+(1-\af)z\in T,$ where $\af\in [0,1],$ $y\in \overline{{\rm conv}(Y)}$ and $z\in \overline{{\rm conv}(Z)}.$
We will show that $\af_n\to \af, $ $y_n\to y$ and $z_n\to z.$ 

Otherwise there is a subsequence $\{n_k\}\subset \N$ such that
\beq\label{Lclaimaff-2}
{\rm dist}(z_{n_k}, z)\ge \sigma
\eneq
for some $\sigma>0.$ We may assume that $\af_{n_k}\to \af_0$ for some $\af_0\in [0,1].$ 
Since $T$ is compact, both $\overline{{\rm conv}(Y)}$ and $\overline{{\rm conv}(Z)}$ are compact. 
Hence we may assume that $z_{n_k}\to z_0\in \overline{{\rm conv}(Z)}$ and $y_{n_k}\to y_0\in \overline{{\rm conv}(Z)}.$
Then 
\beq
\af_{n_k}y_{n_k}+(1-\af_{n_k})z_{n_k}\to \af_0 y_0+(1-\af)z_0.
\eneq
It follows that 
\beq
\af_0y_0+(1-\af_0)z_0=\af y+(1-\af)z_0.
\eneq
Note that $y_0, y\in M_Y$ and $z, z_0\in M_Z.$ By the conclusion of the first paragraph 
of the proof, $\af_0=\af,$ $y_0=y$ and $z=z_0.$ 
This leads a contradiction (to \eqref{Lclaimaff-2}) which proves the lemma.
\end{proof}

\begin{cor}\label{Laff}
%Let $A$ be a separable non-elementary simple \CA\,  with $A={\rm Ped}(A)$ 
Let $T$ be a Choquet simplex such that
 (1) $\partial_e(T)=Y\cup Z$ and $Y\cap Z=\emptyset,$ 
 (2)   
 $Y$ $Z$ are Borel and,
 (3)   $M_Y=
 \overline{{\rm conv}(Y)}$ and  $M_Z= \overline{{\rm conv}(Z)}.$

Then there is  $f\in \Aff(T)$  with $0\le f\le 1$ such that
$f|_{M_Y}=1$ and $f|_{M_Z}=0.$
\end{cor}

\begin{proof}
By (1) and  (3), $\overline{Y}\cap \overline{Z}=\emptyset$ (in $T$). 
Define a real affine function $f$  on $T$ such that $f|_Y=1$ and $f_Z=0$ and, 
for each $t\in T,$  
\beq
f(t)=\int_{\partial_e(T)} f(x) d\mu_t.
\eneq
Clearly $f$ is affine (and continuous on $\partial_e(T)$). {{Note that
$f$ is merely an affine extension of $f|_{\partial_e(T)}.$ In fact, 
$f(t)=\mu_t(Y)$ for all $t\in T.$}}
Since $Y\cap Z=\emptyset,$ we also have $f|_{M_Y}=1$ and $f|_{M_Z}=0$

 The point of the proof is to  show that $f$ is continuous.
Let $\af_n\in [0,1],$ $y_n\in M_Y$ and $z_n\in M_Z$ such
that $\af_ny_n+(1-\af_n)z_n\to \af y+(1-\af)z,$ where $\af\in [0,1],$ $y\in M_Y$ and $z\in M_Z.$
Note that 
\beq
f(\af_n y_n+(1-\af_n)z_n)=\af_n\andeqn f(\af y+(1-\af)z)=\af.
\eneq
By Lemma \ref{Lclaimaff}, $\af_n\to \af.$ Therefore 
\beq
\lim_{n\to \infty}f(\af_n y_n+(1-\af_n)z_n)=f(\af y+(1-\af)z).
\eneq
Thus $f$ is continuous. The lemma follows.
\end{proof}

\begin{prop}\label{Plimits}
Let $T$ be a Choquet simplex with condition (C)
such that $\partial_e(T)=\cup_{n=1}^\infty X_n$ which satisfies (1) and (2) in \ref{Dc1}.
Suppose that $x_n\in M_{\cup_{i=n}^{m(n)}X_i}$ (for some $m(n)\ge n$) and $x_n\to x\in T$ (as $n\to\infty$).
Then $x\in M_{X_1}.$
\end{prop}

\begin{proof}
Put $X=\partial_e(T)$ and fix
$m\in \N\setminus \{1\}.$ 
Since $(\cup_{i=n}^{m(n)}X_i)\cap (\cup_{i=2}^m X_i)=\emptyset$ when $n>m,$ 
$x_n\in \overline{{\rm conv}(X\setminus \cup_{i=2}^mX_i)}$ 
for all $n>m.$
It follows that $x\in \overline{{\rm conv}(X\setminus \cup_{i=2}^mX_i)}.$
Put $Z_m=X\setminus \cup_{i=2}^mX_i$ and $Y_m=\cup_{i=2}^mX_i.$  Then, by \ref{Dc1}, 
$x\in M_{Z_m}.$  It follows that the extremal  boundary (Borel probability) measure 
$\mu_x$ concentrates on $Z_m.$ In other words, ${\mu_x}|_{Y_m}=0.$
This holds for each $m.$ Therefore $x\in M_{X_1}.$
\end{proof}

In the case that $A$ is not unital, we will consider 
the case that  the cone ${\wtd{T}}(A)$  has a basis $S$ which is a Choquet simplex satisfying 
condition (C).
%$\partial_e(S)$  has only countably  many points.
We would like to point out  that this condition does not depend on the choice of basis $S.$ 
This is clarified by the following proposition.

\begin{prop}\label{Pextrb}
Let $A$ be a separable simple \CA\, with ${\wtd{T}}(A)\setminus \{0\}\not=\emptyset.$ 
Suppose that $S\subset \wtd{T}(A\otimes {\cal K})\setminus \{0\}$  is  a compact Choquet simplex which is also a  basis for the cone 
$\wtd{T}(A\otimes {\cal K}).$ 
Suppose that $e\in 
{\rm Ped}(A\otimes {\cal K})\setminus \{0\}$ with $0\le e\le 1$ and 
$T_e=\{\tau\in \wtd{T}(A\otimes {\cal K}): \tau(e)=1\}.$ 

(1) Then there is a homeomorphism (but not necessarily affine) $\gamma$ from $S$ onto $T_e$ 
which maps $\partial_e(S)$ onto $\partial_e(T_e).$

(2) $S$ satisfies condition (C) if and only if $T_e$ satisfies condition (C). 

(3) If $S$ is a Bauer simplex, so is $T_e.$  Moreover, in this case, there is 
an affine homeomorphism $\gamma: S\to T_e$ 
such that
\beq
\gamma(\tau)(a)={\tau(a)\over{\tau(e)}}\tforal a\in {\rm Ped}(A\otimes {\cal K})\tand 
\tau\in \partial_e(S).
\eneq

\end{prop}

\begin{proof}
Let $\gamma: S\to T_e$ be defined by
$\gamma(s)(a)={s(a)\over{s(e)}}$ for all $a\in {\rm Ped}(A
\otimes {\cal K})$ and $s\in S.$
It is an injective continuous map (recall that $A$ is simple, so $s(e)>0$).
 To see it is surjective, 
let $\tau\in T_e.$ Then, since $S$ is a basis for the cone $\wtd{T}(A),$ 
there is a unique $r\in \R_+\setminus \{0\}$ such that $r\cdot \tau\in S.$
Thus 
\beq
\gamma(r\cdot \tau)(a)={r\cdot \tau(a)\over{r\cdot \tau(e)}}={\tau(a)\over{\tau(e)}}=\tau(a)
\eneq
for all $a\in {\rm Ped}(A\otimes {\cal K}).$ This shows that $\gamma$ is  surjective 
and hence a homeomorphism (between compact Hausdorff spaces).

To see $\gamma$ maps $\partial_e(S)$ to $\partial_e(T_e),$ 
let $\tau\in \partial_e(S).$ Suppose that there are $\tau_1, \tau_2\in T_e$ such that
$\gamma(\tau)=\af\tau_1+\bt \tau_2,$ where $0\le \af,\, \bt<1$ and $\af+\bt=1.$ 
%Since $pAp$ is a unital hereditary \SCA, one can extend $\tau_1$ and $\tau_2$ to two traces on $A$ (see 5.2.7 of \cite{Ped}).
%Note also that such extensions are unique, again, because $pAp$ is a hereditary \SCA.
%Since $A$ is  a unital separable  simple \CA,  by \cite{Brs}, there exists a unitary $W\in M(A\otimes {\cal K}),$ 
%an integer $m\ge 1$ and a projection $e\in M_m(pAp)$ such that
%$W^*1_AW=e.$ 
%Denote $s'$ the extension of $\af\tau_1+\bt\tau_2$ this way. 

Define $f_S\in {\rm Aff}(\wtd T(A\otimes {\cal K}))$ by $f_S(r\cdot s)=r$ for all $r\in \R$ and $s\in S.$
Recall that $S$ is a basis for the cone $\wtd{T}(A).$
In particular, if $s_1, s_2\in S$ and $\af' s_1+\bt' s_2=s\in S$ for some 
$\af', \bt'>0,$ then $\af'+\bt' =1.$ 
To see this, let $\dt=\af'+\bt'>0.$ Then $(\af'/\dt)s_1+(\bt'/\dt)s_2=s/\dt.$
Since $S$ is convex, $s_3=s/\dt\in S.$ Since $S$ is a basis, $\dt=1.$ 
From what we have  just verified, 
%Since $S$ is a basis, 
$f_S$ is a well-defined real continuous affine function, and 
$t\in S$ if and only if $f_S(t)=1.$  Moreover, 
$f_S(\tau)>0,$ for any $\tau \in \wtd{T}(A\otimes {\cal K})\setminus \{0\}.$

Put $s(a)={\gamma(\tau)(a)\over{f_S(\gamma(\tau))}}$ for $a\in {\rm Ped}(A\otimes {\cal K}).$ 
Then $s\in S.$  However,  since $\gamma(s)(e)=1=\gamma(\tau)(e),$ 
\beq
\gamma(s)(a)={\gamma(\tau)(a)\over{f_S(\gamma(\tau))s(e)}}=
{\gamma(\tau)(a)\over{f_S(\gamma(\tau))({\gamma(\tau)(e)\over{f_S(\gamma(\tau))}})}}
%\over{\gamma(\tau)(e))}})}}
=\gamma(\tau)(a)
\rforal a\in {\rm Ped}(A\otimes {\cal K}).
%
%\gamma(s)(a)={\gamma(\tau)(a)\over{f_S(\gamma(\tau))(\gamma(s)(e))}}={\gamma(\tau)(a)\over{f_S(\gamma(\tau))({f_S(\gamma(\tau))\over{\gamma(\tau)(e))}})}}=\gamma(\tau)(a)
%\rforal a\in {\rm Ped}(A\otimes {\cal K}).
\eneq
Since $\gamma$ is bijective,  $s=\tau.$

%\beq
%s(a)=\over{\gamma(\tau)(a)\over{\gamma(\tau)(f_s))=
%{(\af\tau_1+\bt \tau_2)(a)\over{

On the other hand, define  $t_i(a)=\tau_i(a)/s_i(f_S)$ for all $a\in {\rm Ped}(A\otimes {\cal K})$ and $i=1,2.$ 
Put $\lambda=\af f_S(\tau_1)+\bt f_S(\tau).$   Then
$$
s=((\af\tau_1+\bt \tau_2)/f_S(\af \tau_1+\bt \tau_2))=(\af\tau_1+\bt \tau_2)/\lambda=({\af f_S(\tau_1)\over{\lambda}})t_1+({\bt f_S(\tau_2)\over{\lambda}})t_2.
$$
%Note $\gamma(f_s)=\af\tau_1(f_s)+\bt \tau_2(f_s).$
Since $s=\tau$ is  assumed to be in  $\partial_e(S)$ and  
$({\af\tau_1(f_S)\over{\lambda}})+({\bt\tau_2(f_S)\over{\lambda}})=1,$
either 
$$
{\af f_S(\tau_1)\over{\lambda}}=0,\,\,\, \text{or} \,\,\, {\bt\ f_S(\tau_2)\over{\lambda}}=0,
$$
which forces either $\af=0$ or $\bt=0.$  This proves  the fact that $\gamma|_{\partial_e(S)}$ maps $\partial_e(S)$ to 
$\partial_e(T_e).$ 
Note that $\gamma^{-1}(t)(a)={t(a)\over{f_S(t)}}$ for all $t\in T_e$ and $a\in {\rm Ped}(A\otimes {\cal K}).$
Similar argument shows that $\gamma^{-1}$ maps 
{{$\partial_e(T_e)$ to $\partial_e(S).$}} This implies that $\gamma|_{\partial_e(S)}$ is a homeomorphism 
from $\partial_e(S)$ onto $\partial_e(T_e).$ 
This proves the first part of the proposition.

For part (2),  let $K$ be a Borel subset of $\partial_e(S).$ 
%and $Z=S\setminus K.$
Note that, if $\tau_1, \tau_2\in K$ and $\af\in [0,1],$
then, with $\tau=\af\tau_1+(1-\af)\tau_2,$ 
\beq\nonumber
&&\gamma(\af\tau_1+(1-\af)\tau_2)={\af\tau_1+(1-\af)\tau_2\over{\af\tau_1(e)+(1-\af)\tau_2(e)}}
\\\nonumber
&&=({\af \tau_1(e) \over{\tau(e)}})({\tau_1\over{\tau_1(e)}})+({(1-\af) \tau_2(e) \over{\tau(e)}})({\tau_2\over{\tau_2(e)}})=
({\af \tau_1(e) \over{\tau(e)}})\gamma(\tau_1)+({(1-\af) \tau_2(e) \over{\tau(e)}})\gamma(\tau_2).
\eneq
Since $\gamma(\tau)(e)=1,$  we have
\beq
1=\gamma(\af\tau_1+(1-\af)\tau_2)(e)=({\af \tau_1(e) \over{\tau(e)}})+({(1-\af) \tau_2(e) \over{\tau(e)}}).
\eneq
This shows that $\gamma$ maps ${\rm conv}(K)$ into ${\rm conv}(\gamma(K)).$
In fact,  
since $\gamma(\tau)={\tau\over{\tau(e)}},$ from measure theory, we also have 
$\gamma(M_K)\subset M_{\gamma(K)}.$
Similarly, replacing $\hat{e}$ by $f_S$ (in the proof of the first part),  the same lines of the argument shows that 
$\gamma^{-1}$ maps ${\rm conv}(\gamma(K))$ into ${\rm conv}(K).$
It follows that $\gamma$ maps ${\rm conv}(K)$ onto ${\rm conv}(\gamma(K)).$
Since $\gamma$ is a homeomorphism, 
$
\gamma
$
maps $\overline{{\rm conv}(K)}$ onto $\overline{{\rm conv}(\gamma(K))}.$
Since ${\rm conv}(\gamma(K))\subset M_{\gamma(K)}\subset \overline{{\rm conv}(\gamma(K))},$ 
if $M_K=\overline{{\rm conv}(K)},$ then $M_{\gamma(K)}=\overline{{\rm conv}(\gamma(K))}.$

Now if $S$ satisfies condition (C), 
we may write  $S=\cup_{n=1}^\infty X_n$ which satisfies (1) and (2) in \ref{Dc1}.
Then $T_e=\cup_{n=1}^\infty \gamma(X_n),$  {{$\gamma(X_i)\cap \gamma(X_j)=\emptyset,$
if $i,j\ge 2$ and $i\not=j,$}}
and each $\gamma(X_n)$ is compact and 
has finite covering dimension (as  $X_n$ is compact and has finite covering dimension). 
From what has been shown,  if $k, m\in \N,$
\beq\nonumber
\overline{{\rm conv}(\cup_{j=1}^m \gamma(X_{j+k})) }=M_{\gamma(\cup_{j=1}^m X_{j+k})}\andeqn 
\overline{{\rm conv}(X\setminus \cup_{j=1}^m \gamma(X_{j+k}))}=M_{T_e\setminus \cup_{j=1}^m \gamma(X_{j+k})}.
\eneq
Thus $T_e$ satisfies condition (C).  The same argument shows that, if $T_e$ satisfies condition (C) so does 
$S.$

%If $M_K=\overline{{\rm conv}(K)}$
%suppose that $S=\cup_{n=1}^\infty X_n$ which satisfies (1) and (2) in \ref{Dc1}. 

For (3), let us assume that  $S$ is a Bauer simplex, i.e., $\partial_e(S)$ is compact. By the first part
of the proposition, so is $\partial_e(T_e).$

For each $\tau\in S,$ by the  Choquet theorem, there is a unique  probability Borel measure 
$\mu_\tau$ on $\partial_e(S)$ such that 
$$
f(\tau)=\int_{\partial_e(S)} fd\mu_\tau\rforal f\in \Aff(S).
$$
Define $\kappa(\mu_\tau)(F)=\mu_\tau(\gamma^{-1}(F))$ for any Borel subset $F\subset \partial_e(T_e).$ 
For each $g\in \Aff(T_e),$ $g|_{\partial_e(T_e)}$ is continuous.
For each $\tau\in S,$ define
%since $T_e$ is a Bauer simplex,
\beq\label{Pbauer-1}
g(\wtd\gamma(\tau))=\int_{\partial_e(T_e))} g d\kappa(\mu_\tau)\rforal g\in \Aff(T_e).
\eneq
Since $T_e$ is a Bauer simplex,  the formula above gives 
the unique continuous affine extension of $g|_{\partial_e(T_e)}.$
  Note that the map
$\gamma^\sharp: \Aff(T_e)\to \Aff(S)$ by 
\beq
\gamma^{\sharp}(f)=\int_{\partial_e(S)} f\circ \gamma d\mu_\tau\rforal f\in \Aff(T_e)
\eneq
is continuous in weak-* topology of $\Aff(T_e).$ It follows  that $\wtd\gamma(\tau)\in T_e$ for each $\tau\in S$ and 
that the map 
$\wtd \gamma: S\to T_e$ defined by $\tau\mapsto  \wtd\gamma(\tau)$ is continuous and affine.
Since $\wtd \gamma|_{\partial_e(S)}=\gamma|_{\partial_e(S)}$ is a homeomorphism,    $\wtd \gamma$ maps $S$ onto $T_e,$ 
which is 
an affine homeomorphism. 
\end{proof}

%%%%%%%%%%%%%%%%%
\iffalse
\begin{df}\label{Dcextremal}
Let $A$ be a separable simple \CA\, with $\wtd{T}(A)\setminus \{0\}\not=\emptyset.$
We say  that $\wtd{T}(A)$  has countable extremal boundary, if $\wtd{T}(A)$ has a Choquet simplex 
$S$ as its basis whose extremal boundary $\partial_e(S)$ has countably many points.
\end{df}
\fi
%%%%%%%%%%%
%\end{df}
\begin{df}\label{DconC}
Let $A$ be a separable \CA\, such that 
whose Pedersen ideal contains a full element (this includes the case that $A$ is unital).
 We say that extremal boundaries of the tracial cone of $A$ satisfies condition 
(C) if  $\wtd{T}(A)\setminus \{0\}\not=\emptyset$ and 
%extremal boundary $\partial_e(S)$ of 
one of bases $S$ of $\wtd{T}(A)\setminus \{0\}$  
satisfies condition (C). Note that in this case, $\wtd{T}(A)$ always has a Choquet simplex as 
a base for the cone $\wtd{T}(A)$ (see Proposition 3.4 of \cite{TT}).
\end{df}

\subsection{Oscillation and orthogonal complements}\label{sub2}

\begin{df}[see  Definition A.1 of \cite{eglnkk0}]\label{DefOS1}
Let $A$ be a \CA\,  
%whose 2-quasitraces are all  traces.  
%${\wtd{QT}}(A)\setminus \{0\}\not=\emptyset.$ 
Let $S\subset
 {\wtd{QT}}(A)
%\wtd{T}(A)
\setminus \{0\}$ be a compact subset. 
Define, for each $a\in {\rm Ped}(A\otimes {\cal K})_+,$ 
\beq
\omega(a)|_S&=&\inf\{\sup\{d_\tau(a)-\tau(c): \tau\in S\}: c\in \overline{a(A\otimes {\cal K})a}, \,0\le c\le 1\}.
\eneq

Recall (from Theorem 4.7 of \cite{eglnp}) that a $\sigma$-unital \CA\, $A$ is called compact, if ${\rm Ped}(A)=A.$
Then, by Lemma 4.5 of \cite{eglnp}, $0\not\in \Qw.$ 
If $A$ is compact,  we may choose $S=\Qw.$
In that case, we will omit $S$ in the notation.
Note that $\omega(a)|_S=0$ if and only if $d_\tau(a)$ is continuous on $S.$

Let $A$ be a $\sigma$-unital simple \CA\, with $\wtd{QT}(A)\not=\{0\}.$
%Suppose that all 2-quasi-traces of $A$ are traces.  
Let $e\in A$ be a strictly positive element.
If $A$ has continuous scale,  then $\omega(e)=0$ and 
$QT(A)$ is compact (see, for example, Proposition 5.4 of \cite{eglnp}).
  If $A$ also has strict comparison, then $A$ has continuous scale 
 if and only if $\omega(e)=0$ (see Proposition 5.4 and Theorem 5.3 of \cite{eglnp}). 
 
 Note that, if $A$ is compact and $\wtd{QT}(A)\setminus \{0\}\not=\emptyset,$
 then $S:=\Qw$ is a compact subset of $\wtd{QT}(A)\setminus \{0\}$ and $\R_+\cdot S=\wtd{QT}(A).$
%We will also use the fact that, in this case, $A$ has continuous scale if and only if $T(A)$ is compact (see Proposition 
%5.4 of \cite{eglnp}, for example).  
\end{df}

\begin{df}[Definition 4.7 of \cite{FLosc}]\label{defOs2}
Let $A$ be a  $\sigma$-unital \CA\, with $\wtd{QT}(A)\setminus \{0\}\not=\emptyset$ and $S\subset \wtd{QT}(A)\setminus \{0\}$  be a compact subset such that $\R_+\cdot S=\wtd{QT}(A)$
(such as unital \CA s $A$ with $QT(A)\not=\emptyset$).
%
%compact \CA\, with $QT(A)\not=\emptyset.$
%We assume that all 2-quasitraces of $A$ are traces.
Let $a\in {\rm Ped}(A\otimes {\cal K})_+$ and let
$\Pi: l^\infty(A)\to l^\infty(A)/I_{_{S,\N}}$
%I_{_{\Qw,\N}}$ 
be the quotient map. Define
%(here we assume that $b_n\in {\rm Ped}(A)_+.$)
% {\red{If $A$ is not simple, then one has to make 
%some global condition----more carefully}}
\beq\nonumber
%\Omega^T(a)&=& \lim\inf\{\|a-b_n\|_{2, \Tw}: b_n\in A_+, \|b_n\|\le M_S\|a\|,\,\,  \lim_{n\to\infty}\omega(b_n)=0\},\\\nonumber
%\Omega^T(a)|_S&=& \lim\inf\{\|a-b_n\|_{2, S}: 
%\{b_n\}\in l^\infty({\rm Her}(a)))_+, \,\, \lim_{n\to\infty}\omega(b_n)|_S=0\},\\\nonumber
\Omega^T(a)|_S=\inf\{\|\Pi(\iota(a)-\{b_n\})\|: b_n\in \Her(a_n)_+, \|b_n\|\le \|a\|, \,\, \lim_{n\to\infty}\omega(b_n)|_
%{\Qw}
{S}=0\}.
\eneq
%\andeqn\\\nonumber
%\hspace{-0.4in}\Omega^T_{a,{\bf 1}}(a)|_S&=& \lim\inf\{\|a-b_n\|_{2, S}: b_n\in {\rm Her}(a)_+, \|b_n\|\le \|a\|, \lim_{n\to\infty}\omega(b_n)|_S=0\},\\\nonumber
%\Omega_C^T(a)|_S&=&\inf\{\|\Pi_{cu}(\iota(a)-\{b_n\}\|: \{b_n\}\in l^\infty(A), \lim_{n\to\infty} \omega(b_n)|_S=0\}.
%\eneq
%\andeqn\\\nonumber
%\Omega^{C}_a(a)|_S&=&\inf\{\Pi(\iota(a)-\{b_n\}\|: b_n\in {\rm Her}(a)_+, \,\, \|b_n\|\le M_S\|a\|,\,\,\lim_{n\to\infty} \omega^c(b_n)|_S=0\}\\
%\Omega^N(a)|_S&=&\inf \{\|a-b_n\|: b_n\in l^\infty({\rm Her}(a))_+, \lim_{n\to\infty} \omega(b_n)|_S=0\}.
%\eneq
One may call $\Omega^T(a)|_S$ the tracial approximate oscillation of $a.$
If $\Omega^T(a)|_S=0,$ we say  that the element $a$ has  approximately tracial  oscillation zero.
Note that $\Omega^T(a)|_S=0$ if and only if there exists $b_n\in \Her(a)_+$ with $\|b_n\|\le \|a\|$ such that  
\beq
\lim_{n\to\infty}\|a-b_n\|_{_{2, S}}=0\andeqn\lim_{n\to\infty}\omega(b_n)=0
\eneq
(see   Proposition 4.8 
%{Omega-1}
 of \cite{FLosc}).   It follows from Proposition 4.9  of \cite{FLosc} that 
that $a$ has approximately tracial oscillation zero does 
 not depend on the choice of $S$ (see also (2) of Proposition 4.8 of \cite{FLosc}). So under the assumption above on $A,$ 
 we may write $\Omega^T(a)=0$ (instead of $\Omega^T(a)|_S=0$).
%In the case that $A$ is compact, we often choose $S$ to be $\Qw.$ 

We say that $A$ has tracial approximate oscillation zero, if $\Omega^T(a)=0$ 
for all $a\in  {\rm Ped}(A\otimes {\cal K})_+$  (see Definition 5.1 of \cite{FLosc}).
%($A$ still has the same assumption as the first line 
%of this definition).
If we view $\|\cdot \|_{_{2, S}}$ as an  $L^2$-norm, 
 the condition of $\Omega^T(a)=0$ has an analogue  to that ``almost'' continuous functions are $L^2$-norm dense.
 
 Let $A$ be a $\sigma$-unital simple \CA\,
 %For the notion of tracial approximate oscillation zero, we are  mainly interested in 
 %the case that $A$ is $\sigma$-unital simple \CA\, 
 with $\wtd{QT}(A)\setminus \{0\}\not=\emptyset.$ 
 Choose any $a\in {\rm Ped}(A\otimes {\cal K})_+\setminus \{0\}.$
 Then $\Her(a)$ is compact. Hence  $S:=\overline{QT(\Her(a))}^w$ is a compact
 subset of $\wtd{QT}(\Her(a))\setminus \{0\}$ and 
 $\R\cdot S=\wtd{QT}(A).$ Note that, by \cite{Br}, $\Her(a)\otimes {\cal K}\cong A\otimes {\cal K}.$ 
 In other words, $A$ is stably isomorphic to an algebraically simple \CA. 
 %In this case, we often choose $S$ to be $\overline{\wtd{QT}(\Her(a))}.$
 Therefore, in  this case, we often assume that $A$ is also compact, we choose $S=\Qw.$  
% which has strict comparison. In this case, we often choose $S$ to be 

% {\rm Ped}(A\otimes {\cal K})_+.$
%If, in addition,  $A$ has strict comparison, then  we say $A$ has tracial approximate oscillation zero.
%We will omit $S$ in the case 
%that $S=T_e$ for some full element $e\in {\rm Ped}(A)_+^{\bf 1}$ 
%and the 
%In the case that  
%$A={\rm Ped}(A)$ and $S=\Qw.$

%{\red{In general, if $A$ is a $\sigma$-unital \CA\, with $\wtd{T}(A)\setminus \{0\}\not=\emptyset,$
%then we say $A$ has tracial approximate oscillation zero if, for every 
%$a\in {\rm Ped}(A\otimes {\cal K})_+\setminus \{0\},$ $\Her(a)$ has tracial approximate oscillation 
%zero in the sense above (see section 5 of \cite{FLosc}).}}
%Tracial approximate oscillation  can be defined for \CA s which do not have the property that ${\rm Ped}(A)=A.$
%We may also allow 2-quasitraces instead of traces
% (see section 5 of  \cite{FLosc}). 
%In this paper, however, we avoid studying semi-finite traces, or quasi-traces. 

We will use the following fact: If $A$ is a separable simple \CA\, with $\wtd{QT}(A)\not=\{0\}$ and 
the extremal boundary $\partial_e(S)$ of a basis $S$ for the cone $\wtd{QT}(A)$ has only countably many points,
then $A$ has tracial approximate oscillation zero (Theorem 5.9  of \cite{FLosc}). 
\end{df}

\begin{df}\label{Ppermproj}
Let $\varpi\in \bt(\N)\setminus \N$ be a free ultrafilter.
Let $p\in l^\infty(A)/I_{_{\Qw, \varpi}}$ be a projection. An element $\{e_n\}\in l^\infty(A)_+^{\bf 1}$ is called 
a permanent projection lifting of $p,$ if, for any sequence of integers $\{m(n)\},$
$\Pi_\varpi(\{e_n^{1/m(n)}\})=p.$
\end{df}

\begin{prop}[Proposition 6.2 of \cite{FLosc}]\label{Dpproj}

Let $A$ be a separable  compact \CA\, (see \ref{DefOS1}) with $T(A)\not=\emptyset$ and
let $\{e_n\}\in  l^\infty(A)_+^{\bf 1}.$ 
Suppose that $\varpi\in \bt(\N)\setminus \N$ is a free ultrafilter.
%$p=\Pi_\varpi(\{e_n\})$ be a projection.

(1)  Suppose that $p=\Pi_\varpi(\{e_n\})$ is a projection. Then $\{f_{\dt}(e_n)\}$ is a permanent projection lifting of $p$ for any $0<\dt<1/2$
and  
\beq\label{Dperm-1}
\lim_{n\to\varpi}\sup\{\tau((e_n-f_\dt(e_n)e_n)): \tau\in \Qw\}=0.\,\,\,
\eneq

(2)    The following are equivalent:

(i)  $\{e_n\}$ is a permanent projection lifting of $p=\Pi_\varpi(\{e_n\}),$ 

(ii)  
$\lim_{n\to\varpi} \sup\{d_\tau(e_n)-\tau(e_n^2):\tau\in \Qw\}=0,$

(iii)  for all $\dt\in (0,1/2),$
$$\lim_{n\to\varpi} \sup\{d_\tau(e_n)-\tau(f_\dt(e_n)):\tau\in \Qw\}=0.$$

\end{prop}

\begin{proof}

(1) Note that $\Pi(f_\dt(\{e_n\}))=f_\dt(\Pi(\{e_n\}))=p$ for any $0<\dt<1/2.$
Therefore $\Pi(f_{\dt/2}(\{e_n\})=p.$ Put $b_n=f_\dt(e_n),$ $n\in \N.$ 
For any  sequence of integers $\{m(n)\},$ 
%\beq
$b_n^{1/m(n)}\le f_{\dt/2}(e_n),\,\,n\in \N.$
%\eneq
It follows that
\beq
p=\Pi_\varpi(\{f_\dt(e_n)\})\le \Pi_\varpi(\{b_n^{1/m(n)}\})\le \Pi_\varpi(\{f_{\dt/2}(e_n)\})=p.
\eneq
So $\{b_n\}$ is a permanent projection lifting of $p.$  Moreover
$(\{e_n\}-\{f_\dt(e_n)e_n\})^{1/2}\in I_{_{\Qw, \varpi}}$ 
and \eqref{Dperm-1} also holds.
This proves part  (1) of the lemma.

(2)   
%Note first that (ii) $\Rightarrow$ (iii) is obvious.
((ii) $\Rightarrow$ (i))  If 
$$\lim_{n\to\varpi} \sup\{d_\tau(e_n)-\tau(e_n^2):\tau\in \Qw\}=0,$$
then $\{e_n\}-\{e_n^2\}\in I_{_{\Qw, \varpi}}.$ So $p=\Pi_\varpi(\{e_n\})$ is a 
projection. 
Morover, for any
$\{m(n)\},$ 
\beq
\sup\{\tau(e_n^{1/m(n)})-\tau(e_n^2): \tau\in \Qw\}\le \sup\{d_\tau(e_n)-\tau(e_n^2):\tau\in \Qw\}\to^\varpi 0.
\eneq
It follows that $\{e_n^{1/m(n)}-e_n^2\}\in I_{_{\Qw,\varpi}}.$   Then 
$\Pi_\varpi(\{e^{1/m(n)}\})
%=\Pi_\varpi(\{e^{2/m(n)}\})
=\Pi_\varpi(\{e_n\}).$
Therefore 
%$p=\Pi_\varpi(\{e_n\})$ is indeed a projection and 
$\{e_n\}$
is a permanent projection lifting of $p.$
%Hence $\{e_n\}$ is a permanent projection lifting.

((iii) $\Rightarrow$ (i)) Assume, for any $\dt\in (0,1/2),$ 
\beq\label{Dpproj-2}
\lim_{n\to\varpi} \sup\{d_\tau(e_n)-\tau(f_\dt(e_n)):\tau\in \Qw\}=0.
\eneq

Let $m(n)\in \N.$  
There is, for each $n,$   an integer $k(n)\ge m(n)$ such that
\beq\label{Dpproj-1}
\|e_n^{1/k(n)}f_\dt(e_n)-f_\dt(e_n)\|<(1/2n)^2.
\eneq
Put $c_n=e_n^{1/k(n)}f_\dt(e_n)-f_\dt(e_n).$ Then $\Pi_\varpi(\{c_n\})=0.$
It then follows that
\beq\label{225-1}
\Pi_\varpi(\{e_n^{1/k(n)}\}-\{f_\dt(e_n)\})=\Pi_\varpi(\{e_n^{1/k(n)}-e_n^{1/k(n)}f_\dt(e_n)\}).
\eneq
However, for each $\tau\in \Tw,$
\beq
&&\hspace{-1in}\tau((e_n^{1/k(n)}-e_n^{1/k(n)}f_\dt(e_n))^2)\le \tau((e_n^{1/k(n)}- e_n^{1/k(n)}f_\dt(e_n)))\\
&&\le d_\tau(e_n)- \tau(e^{1/k(n)}f_\dt(e_n))\le d_\tau(e_n)-\tau(f_\dt(e_n)) +(1/2n)^2.
\eneq
Hence 
%$\Pi_\varpi(\{e_n^{1/k(n)}\}-\{f_\dt(e_n)\})=0$ and 
\beq
\|e_n^{1/k(n)}-e_n^{1/k(n)}f_\dt(e_n)\|^2_{_{2, \Tw}}\le \sup\{d_\tau(e_n)-\tau(f_\dt(e_n)):\tau\in \Qw\}+(1/2n)^2.
\eneq
%
%%%%%%%%%%%%%%%%%%%%
%
\iffalse
If there is  an integer $K\ge m(n)$ for all $n\in \N,$
then 
\beq
p\le \Pi_\varpi(\{e_n^{1/(m(n)}\})\le \Pi_\varpi(\{e_n^{1/K}\})=p^{1/K}=p.
\eneq
\fi
%%%%%%%%%%%%%%%%%%%%
%as $\Pi_\varpi(\{e_n\})$ is a projection.
\iffalse
Let $q_n$ be the range projection of $f_{\dt/2}(e_n)$ in $A^{**}.$ 
Then,  there is, for each $n,$ $k(n)\ge m(n)$ such that
\beq\label{Dpproj-1}
\|e_n^{1/k(n)}q_n-q_n\|<(1/2n)^2.
\eneq
Put 
\beq\label{Dpproj-cn}
c(n,\dt) &:=&2f_\dt(e_n)q_n(e^{1/k(n)}-q_n)=2f_\dt(e_n)f_{\dt/2}(e_n)(e^{1/k(n)}-q_n)\\
&=&2f_\dt(e_n)f_{\dt/2}(e_n)(e^{1/k(n)}-f_{\dt/2}(e_n)).
\eneq
By \eqref{Dpproj-1}
\beq\label{Dpproj-3}
\|c(n, \dt)\|<1/2n^2,\,\,n\in \N.
\eneq
%Combining with \eqref{Dpproj-2},
We compute that, for any $\tau\in \Tw,$
\beq\nonumber
&&\hspace{-0.4in}\tau((e_n^{1/k(n)}-f_\dt(e_n))^*(e_n^{1/k(n)}-f_\dt(e_n)))
=\tau(e_n^{2/k(n)}-2f_\dt(e_n)e_n^{1/k(n)}+f_\dt(e_n))^2)\\\nonumber
&&=\tau(e_n^{2/k(n)}-2f_\dt(e_n)q_n+f_\dt(e_n)^2)+\tau(2f_\dt(e_n)(e_n^{1/k(n)}-q_n))\\\nonumber
&&\le \tau(e_n^{2/k(n)}-2f_\dt(e_n)q_n+f_\dt(e_n))+\tau(2f_\dt(e_n)q_n(e_n^{1/k(n)}-q_n))\\\label{Dpproj-4}
&&=\tau(e_n^{2/k(n)}-f_\dt(e_n))+\tau(c(n,\dt)).
\eneq
%then, for any
%$\{m(n)\},$ 
By \eqref{Dpproj-2},
\beq
&&\lim_{n\to\varpi}\sup\{\tau(e_n^{2/k(n)})-\tau(f_\dt(e_n)): \tau\in \Qw\}\\
&&\le \lim_{n\to\varpi} \sup\{d_\tau(e_n)-\tau(f_\dt(e_n)):\tau\in \Qw\}= 0.
\eneq
\fi
%%%%%%%%%
%Therefore,  combining with \eqref{Dpproj-3}   and \eqref{Dpproj-4}, 
It follows   that  $\Pi_\varpi(\{e_n^{1/k(n)}-e_n^{1/k(n)}f_\dt(e_n)\})=0.$ 
By \eqref{225-1},
%\beq
%\lim_{n\to\varpi} \|e_n^{1/k(n)}-e_n^{1/k(n)}f_\dt(e_n)\|_{_{2, \Tw}}=0.
%\eneq
%Thus 
$\{e_n^{1/k(n)}-f_\dt(e_n)\}\in I_{_{\Qw,\varpi}}.$ But $\Pi_\varpi(\{f_\dt(e_n)\})=\Pi_\varpi(\{e_n\}).$
It follows that $\Pi_\varpi(\{e_n^{1/k(n)}\})=p.$
Recall  that
\beq
p\le \Pi_\varpi(\{e_n^{1/m(n)}\})\le \Pi_\varpi(\{e_n^{1/k(n)}\})=p.
\eneq
We conclude that  $p$ is a projection and  $\{e_n\}$ is a permanent projection lifting.

For the converse, we claim that if (i) holds then 
$\lim_{n\to\varpi}\omega(e_n)=0.$

Otherwise,  there exists $\sigma>0$  satisfying the following:
 For any ${\cal P}\in \varpi,$ there  is   an integer $n({\cal P})\in {\cal P}$ 
 such that
that $\omega(e_{n({\cal P})})>\sigma.$ 
Fix any $\dt\in (0,1/4),$ by Proposition 4.6 of 
%\ref{Ppartomega} of \
\cite{FLosc},  for each of these $n({\cal P}),$ there is 
an integer  
$r(n({\cal P}))$ such 
that
\beq\label{Dpproj-5}
\sup\{\tau(e_{_{n({\cal P})}}^{1/r(n({\cal P}))})-\tau(f_{\dt}(e_{_{n({\cal P})}})): \tau\in \Tw\}>\omega(e_{_{n(\cal P)}})-\sigma/4>\sigma/2.
\eneq
We now define, for each $k\in \N,$ an integer $m(k)$ as follows:
If $k=n({\cal P})$ for some ${\cal P}\in \varpi,$ 
define 
\beq
m(k)=\min\{r(n({\cal P})):  k=n({\cal P})\}. 
\eneq
This is well defined. 
If $k\not=n({\cal P})$ for any ${\cal P}\in \varpi,$ define 
$m(k)=k.$ 

Then, for any ${\cal P}\in \varpi,$ if $k=n({\cal P})\in {\cal P},$ by \eqref{Dpproj-5}, 
\beq
\sup\{\tau(e_k^{1/m(k))})-\tau(f_{\dt}(e_{k})): \tau\in \Qw\}>\sigma/2.
\eneq
In other words,
\beq
\lim_{n\to\varpi}\|(e_n^{1/m(n)}-f_{\dt/2}(e_n))^{1/2}\|_{_{2, \Qw}}\not< \sigma/2.
\eneq
%
%Choose a sequence $m(n)$ of integers which extends $m(n({\cal P})).$ 
%Then
%\beq
%\liminf_n\|e_n^{1/m(n)}-f_{\dt/2}(e_n)\|_{2, \Tw}\ge \sigma/2.
%\eneq
Therefore $\Pi(\{e_n^{1/m(n)}\})\not=\Pi(f_\dt(\{e_n\}))=p.$ A contradiction.
Hence
$\lim_{n\to\varpi}\omega(e_n)=0.$

%It follows that, there exists a sequence 

For (i) $\Rightarrow$ (ii),
and (i) $\Rightarrow$ (iii), 
we first note that 
there exists 
%$\{m(n)\}$ and 
a sequence $\dt_n\in (0,1/2)$  such that
$
%&&\sup\{d_\tau(e_n)-\tau(e_n^{1/m(n)}): \tau\in \Tw\}<\omega(e_n)+1/n,\,\,n\in\N,\\
\lim_{n\to\infty} \dt_n=0$
and
$$
\sup\{d_\tau(e_n)-\tau(f_{\dt_n}(e_n)): \tau\in \Qw\}<\omega(e_n)+1/n,\,\,n\in \N.
%\|c(n, \dt)\|<(1/2n)^2,\,\,n\in \N,
$$
As in the proof of (iii) $\Rightarrow$ (i), we may choose integers $m(n)$ such that 
\beq
&&\|e_n^{1/m(n)}f_{\dt_n}(e_n)-f_{\dt_n}(e_n)\|<1/n,\,\, n\in \N\andeqn\\
&&\sup\{d_\tau(e_n)-\tau(e_n^{1/m(n)}): \tau\in \Qw\}<\omega(e_n)+1/n,\,\,n\in\N.
\eneq
Put $d_n=e_n^{1/m(n)}f_{\dt_n}(e_n)-f_{\dt_n}(e_n),$ $n\in \N.$
Then $\{d_n\}\in c_0(A).$
%where $c(n, \dt)$ is as defined in \eqref{Dpproj-cn}.
Note  that  ($e_n^2e_n^{1/m(n)}=e_n^{1/m(n)}e_n^2$)
%(using CBS-inequality)
% since $\{e_n^{1/m(n)}-e_n\}\in I_{_{\Qw}}$
 %(for any $\{m(n)\}$), we also have that
\beq\nonumber
&&\hspace{-0.5in}\sup\{d_\tau(e_n)-\tau(e_n^2): \tau\in \Qw\}\\\nonumber
&&\le 
\sup\{d_\tau(e_n)-\tau(e_n^{1/m(n)}): \tau\in \Qw\}+\sup\{\tau(e_n^{1/m(n)}-e_n^2):\tau\in \Qw\}\\
&&<\omega(e_n)+1/n+\|e_n^{1/m(n)}-e_n^2\|_{2, \Qw}.
%\,\,\to 0 \,\,\,\, {\rm as\,\,m\to\infty}.
\eneq
By the claim just proved, if (i) holds, 
then (ii) follows.

Similarly,   we have
\beq\nonumber
&&\hspace{-0.4in}\sup\{d_\tau(e_n)-\tau(f_\dt(e_n)): \tau\in \Qw\}\\\nonumber
&&\le 
\sup\{|d_\tau(e_n)-\tau(f_{\dt_n}(e_n))|: \tau\in \Qw\}+\sup\{|\tau(f_{\dt_n}(e_n))-\tau(f_\dt(e_n))|:\tau\in \Qw\}\\\label{Dpproj-10}
&&<\omega(e_n)+1/n+\|f_{\dt_n}(e_n)-f_\dt(e_n)\|_{_{2, \Qw}}.
%\,\,\to 0 \,\,\,\, {\rm as\,\,m\to\infty}.
\eneq
We note that
\beq
e_n^{1/m(n)}\ge e_n^{1/m(n)}f_{\dt_n}(e_n)=f_{\dt_n}(e_n)+d_n,\,\,n\in \N.
\eneq
Since $d_n\in c_0(A),$
$$
\Pi_{\varpi}(\{f_{\dt_n}(e_n)\})=\Pi_\varpi(\{f_{\dt_n}(e_n)+d_n\})=\Pi_\varpi(e^{1/m(n)}f_{\dt_n}(e_n)\})
\le \Pi_\varpi(\{e_n^{1/m(n)}\})=p.
$$
Since $\lim_{n\to\infty}\dt_n=0,$
we also have  (for large $n$)
\beq
p=\Pi_\varpi(\{f_\dt(e_n)\})\le \Pi_{\varpi}(\{f_{\dt_n}(e_n)\})\le p.
\eneq
Consequently 
$\{f_{\dt_n}(e_n)-f_\dt(e_n)\}\in I_{_{\Tw,\varpi}}.$
Then,  by \eqref{Dpproj-10} and the claim that $\lim_{n\to\infty}\omega(e_n)=0,$  we have
\beq
\lim_{n\to\varpi}\sup\{d_\tau(e_n)-\tau(f_\dt(e_n)): \tau\in \Qw\}=0.
\eneq
%
%%%%%%%%%%%
\iffalse
But we also have 
$\{e_n^{1/m(n)}-e_n\}, \{f_\dt(e_n)-e_n\} \in I_{_{\Tw,\varpi}}$
 (for any $\{m(n)\}$).
 It follows that  $\Pi_\varpi(\{f_{\dt_n}(e_n)-f_\dt(e_n)\}\in I_{_{\Tw,\varpi}}.$ Therefore
 \beq
 \lim_{n\to\varpi}\|f_{\dt_n}(e_n)-f_\dt(e_n)\|_{2, \Tw}=0.
 \eneq
 Consequently, 
\beq
\lim_{n\to\varpi}\sup\{d_\tau(e_n)-\tau(f_\dt(e_n)): \tau\in \Tw\}=0.
\eneq
\fi
%%%%%%%%%%%%%%%%%%%%%%%%%%%%%%%
This implies (iii) holds. The proposition follows.
\end{proof}

{\it For the rest of the paper, we will assume that \CA\, $A$ has the property that all 2-quasitraces 
are traces.}

Proposition \ref{LOs-n1}  and Corollary \ref{CCOs}  are not  directly used 
in the proof of the next two sections. However, some ideas behind the proof of section 4
are inspired  by (a somewhat more complicated version of) these  easy facts. 
We believe that it is appropriate  to share these with the reader. 

\begin{rem}\label{Rclosure}
Let $A$ be a separable \CA\, and $a\in A_+^{\bf 1}.$ 
Put $C=\Her(a)$ and $C^\perp=\{b\in A: ba=ab=0\}.$ 
A natural but naive question is how large is  $C+C^\perp?$ 
The answer is, of course, very disappointing, as $C^\perp$ may well be zero.
Let $p_a$ be the open projection associated with $a.$ It is well known 
that $p$ may well be dense in $A$ (even in the commutative case). 

We would like to offer Proposition \ref{LOs-n1}  and Corollary \ref{CCOs}
to rescue us from the total disaster
% as suggested  by  Proposition \ref{Pcountex-1}
as we will see $C+C^\perp$ could be approximately  as large as possible (at least tracially).
\end{rem}

\begin{prop}\label{LOs-n1}
Let $A$ be a compact \CA\,  with $\widetilde T(A)\setminus \{0\} \not=\emptyset$
and let $e_B\in A_+\setminus \{0\}.$  
Suppose that $0\le a\le 1$ is in $B={\rm Her}(e_B).$ 
%with $\omega(a)\ge 0.$
Then, for any $\ep>0,$ there exists $\dt_0>0$ such 
that, for all $0<\dt< \dt_0,$
\beq
\hspace{-0.3in}d_\tau(f_\dt(a))+d_\tau(e_C)&>&d_\tau(e_B)-\liminf_{\eta\to \dt/2}\{d_\tau(f_\eta(a))-d_\tau(f_\dt(a)): 0<\eta<\dt/2\}\\\
\label{LOs-n1-ne-1}
&>&d_\tau(e_B)-\omega(a)-\ep\hspace{0.3in} \tforal \tau\in \Tw,
\eneq
where $e_C$ is a strictly positive element of $C={\rm Her}(f_\dt(a))^\perp\cap B.$
Moreover,
\beq
\omega(f_\dt(a)+e_B)<\omega(e_C)+\omega(a)+\ep.
\eneq
\end{prop}

\begin{proof}
Let $\ep>0.$
Choose $\dt_0>0$ such that, for all $0<r\le  2\dt_0,$ 
\beq\label{LOs-n1-nn1}
\sup\{d_\tau(a)-\tau(f_{r}(a)):\tau\in \Tw\}<\omega(a)+\ep/4.
\eneq
%Let $e_A\in A$ with $0\le e_A\le 1$  be a strictly positive element.
Set $0<\dt<\dt_0.$  Choose any $0<\eta<\dt/2.$
Define, for $m,n\in \N,$  
$$
d_{m,n}:=(1-f_{\eta}(a)^{1/n})e_B^{1/m}(1-f_{\eta}(a)^{1/n}).
$$
Then   $0\le d_{m,n}\le 1$ and 
$
d_{m,n}f_{\dt}(a)=f_\dt(a)d_{m,n}=0.
$
In other words, $d_{m,n}\in C:={\rm Her}(f_\dt(a))^\perp\cap B.$
Let $e_C\in C$ be a strictly positive element.
We compute that, for all $\tau\in \Tw,$
\beq
d_\tau(e_C)\ge \tau(d_{n,m})=\tau(e_B^{1/m})-2\tau(e_B^{1/m}f_{\eta}(a)^{1/n})+\tau(f_{\eta}(a)^{1/n}e_B^{1/m}f_{\eta}(a)^{1/n}). 
\eneq
%Then,  for fixed $\tau\in \Tw$ and for all $n,m\in \N,$ 
%\beq
%d_\tau(e_C)\ge \tau(e_B^{1/m})-2\tau(e_B^{1/m}f_{\eta}(a)^{1/n})+\tau(f_{\eta}(a)^{1/n}e_B^{1/m}f_{\eta}(a)^{1/n}). 
%\eneq
Let $m\to \infty,$   for all $\tau\in \Tw,$
\beq
d_\tau(e_C)\ge d_\tau(e_B)-2\tau(f_{\eta}(a)^{1/n})+\tau(f_{\eta}(a)^{2/n}).
\eneq
Then,  by letting $n\to\infty,$   we obtain
%\beq\nonumber
$d_\tau(e_C)\ge d_\tau(e_B)-d_\tau(f_{\eta}(a)).$
%\eneq
Therefore
\beq
d_\tau(f_\dt(a))+d_\tau(e_C)\ge d_\tau(e_B)-(d_\tau(f_{\eta}(a))-d_\tau(f_\dt(a))\rforal \tau\in \Tw.
\eneq
Letting $\eta\to \dt/2,$ we obtain, for all $\tau\in \Tw,$
\beq
d_\tau(f_\dt(a))+d_\tau(e_C)\ge d_\tau(e_B)-\liminf_{\eta\to \dt}\{d_\tau(f_{\eta}(a))-d_\tau(f_\dt(a)):0<\eta<\dt/2\}.
\eneq
However,  for any $\tau\in \Tw,$
\beq
\liminf_{\eta\to \dt}\{d_\tau(f_{\eta}(a))-d_\tau(f_\dt(a)):0<\eta<\dt/2\}\le d_\tau(a)-\tau(f_{\dt_0}(a))
\le \omega(a)+\ep.
\eneq
We then  have, for all $\tau\in \Tw,$
\beq
d_\tau(f_\dt(a))+d_\tau(e_C) &\ge& d_\tau(e_B)-\liminf_{\eta\to \dt}\{d_\tau(f_{\eta}(a))-d_\tau(f_\dt(a)):0<\eta<\dt/2\}\\
&>& d_\tau(e_B)-\omega(a)+\ep.
\eneq

For the last part of the statement,   since $d_{mn}\in C,$ 
we may choose $e_C$ so that
%\beq
$d_{m,n}\le e_C.$
%(1-f_\eta(a)^{1/n})e_A(1-f_\eta(a)^{1/n}).
%\eneq
Then, by I.1.11 of \cite{BH},
\beq
(1/2)e_B\le  d_{m,n}+f_{\eta}(a)^{1/n}\le 
f_\dt(a)+e_C+(f_{\eta}(a)^{1/n}-f_\dt(a)).
\eneq
It follows that
%\beq
$e_B\sim f_\dt(a))+e_C+(f_{\eta}(a)^{1/n}-f_\dt(a)).$
%\eneq
Therefore (by (1) of Proposition 4.4 of \cite{FLosc} for the first inequality below),
\beq
\omega(e_B)&=&\omega(f_\dt(a)+e_C+(f_{\eta}(a)^{1/n}-f_\dt(a)))\\
&\ge&  \omega(f_\dt(a)+e_C)-\sup\{d_\tau(f_\eta(a)^{1/n}-f_\dt(a)):\tau\in \Tw\}\\
&{\stackrel{\eqref{LOs-n1-nn1}}{>}}&\omega(f_\dt(a)+e_C)-\omega(a)-\ep/4.
\eneq
Hence
\beq\nonumber
\omega(f_\dt(a))+e_C)<\omega(e_B)+\omega(a)+\ep.
\eneq
\vspace{-0.1in}
\end{proof}

\begin{cor}\label{CCOs}
Let $A$ be a  compact \CA\, with  $\tilde T(A)\setminus \{0\} \not=\emptyset$
and let $e_B\in A_+\setminus \{0\}.$  
Suppose that $0\le a\le 1$ is in $B:={\rm Her}(e_B)$ with $\omega(a)=0.$
Then, for any $\ep>0,$ there exists $\dt_0>0$ such 
that, for all $0<\dt< \dt_0,$
\beq
d_\tau(f_\dt(a))+d_\tau(e_C)>
d_\tau(e_B)-\ep \tforal \tau\in \Tw,
\eneq
where $e_C$ is a strictly positive element of $C:={\rm Her}(f_\dt(a))^\perp\cap B.$
Moreover,
\beq\nonumber
\omega(f_\dt(a)+e_C)<\omega(e_B)+\ep.
\eneq
%Furthermore, 
%there exists an integer $N\ge 1$ such that, for all $n\ge N,$
%\beq
%\|e_C(f_\dt(a)+e_B)^{1/n}-e_C\|_{2, \Tw}<\ep.
%%%%%%\|e_C^2\|\cdot \sqrt{\omega(a)+\ep},
%\eneq
\end{cor}

\section{Lifting}

Throughout this section $\varpi\in \bt(\N)\setminus \N$ is a fixed free ultrafilter.

The most of the content  of this section is within the same lines as that of \cite{MS} 
which are also  reformulated in \cite{KR} and \cite{TWW-2}. But by the end of this 
section, we add some flavor of tracial approximate oscillation to these results. 

The following is  a version of a result of Y. Sato (see Lemma 2.1 of \cite{S12}). 
The unital version  (see Theorem 3.3 of \cite{KR}) of  the next lemma  can also be found  in \cite{KR}
(for the second statement, see also the end of Remark 3.11 of \cite{KR2}).

\begin{thm}\label{TSR}
Let $A$ be a separable \CA, $\tau$ a faithful tracial state,
$N$  the weak closure of $A$ under the GNS representation of $A$ with respect  to
the state $\tau.$ 
%and let $\vo$ be a free ultrafilter on $\N.$
Then there are natural  surjective \hm s
\beq
\Av\to l^\infty(N)/I_{\tau,\vo}(N)\tand  \Av\cap A'\to 
(\Nuv)\cap N'.
\eneq
Moreover,
\beq
\Avt \cong \Nuv\tand (\Avt)\cap A'\cong (\Nuv)\cap N'
\eneq
(see also the last line of \ref{Dultrafiler} for notation).
\end{thm}

\begin{proof}
The unital case of the theorem is known (see Theorem 3.3 of \cite{KR}).
In  fact the proof of the non-unital case is the same. 
We will not repeat the proof.
However
let us briefly look at the non-unital case for first part of the statement. 
%follows the exactly the same 
%However, since we will study 
%non-unital \CA s, we would like to repeat the proof of the first part of the statement only for the convenience 
%of the reader.

Let $\pi_A$ and $\pi_N$ denote the quotient maps 
$l^\infty(A)\to \Av$ and $l^\infty(N)\to \Nuv,$ respectively.
Denote the canonical map $\Av \to \Nuv$ by $\Phi,$ and let 
$\wtd \Phi: l^\infty(A)\to \Nv$  denote the map $\Phi\circ \pi_A.$

%We will only repeat the proof  for non-unital case  which is the same as the proof 
%of unital case for the map $\Phi: \Av\to \Nuv.$

Let $x=\pi_A(\{x_n\})\in \Nuv.$ Then, by Kaplanski's density theorem (we wrote this on 
a Thanksgivings Day which is not a Sunday--see 2.3.4 of \cite{Pbook}), there  exists $a_k\in A$ with $\|a_k\|\le \|x_k\|$ such that
$\|a_k-x_k\|_{2,\tau}<1/k.$ It follows that $a:=\{a_k\}\in l^\infty(A)$ 
and  $\wtd \Phi(a)=x$ ($A$ has unit or not). This shows that  the proof for the first surjectivity of 
Theorem 3.3 of \cite{KR} works 
for non-unital case. 

The proof of the second surjectivity in the proof of Theorem 3.3 of \cite{KR} also uses Kirchberg's 
$\ep$-test. However, let us point out that 
the exactly the same proof works for non-unital case without changing any words
(see the proof of Theorem 3.3 of \cite{KR}). 
%We are allowed not to repeat the same proof for the second surjectivity.
%
%
\end{proof}

%Let $A$ be a $\sigma$-unital simple \CA\, with ${\rm Ped}(A)=A$ and let 
%$e_A\in A$ with $\|e_A\|=1$ be a strictly positive element of $A.$ 
%Fix an element $e\in {\rm Ped}(M_2(A))_+^{\bf 1}$ with $\|e\|=1$ such that $\tau(e)\ge \tau(e_A)$
%for all $\tau\in \Tw.$  Define $T_e:=\{\tau\in {\wtd{T}}(A): \tau(e)=1\}.$ 
%Note that $\|\tau\|$ More importantly, $T_e$ is a Choquet simplex and a basis for the 
%cone ${\wtd{T}}(A).$ 

The following lemma is stated exactly the same as Lemma 2.9 of \cite{TWW-2}
for the  unital case.

\begin{lem}[Lemma 3.3 of \cite{MS} and Lemma 2.9 of \cite{TWW-2}]
%and Lemma 4.5 of \cite{Gamma}]
\label{LSato-2}

Let $A$ be a separable non-elementary simple nuclear \CA\, with continuous scale   and 
$\tau\in T(A).$
% be  an  extremal tracial state.
 Then, for any $k\ge 2,$ there exists a unital \hm\, 
$\phi: M_k\to (\Avt)\cap A' $ and an order zero \cpc\, $\psi: M_k\to  (\Av)\cap A'$
such that $\pi\circ \psi=\phi,$ where $\pi: (\Av)\cap A'\to (\Avt)\cap A'$  is the quotient map
(see also the last line of \ref{Dultrafiler} for notation).

\end{lem}
%\iffalse
\begin{proof}
\iffalse
The proof is the same as  that for the unital case. To be complete, let us repeat some of it.
Let $\rho_\tau$ denote the GNS-representation  associated with $\tau.$
Since $\tau$ is an extremal point of $T(A),$ $N:=\rho_\tau(A)''$  is a factor.
This is standard. Otherwise, let $p\in N'$ be a  non-zero projection such that $p\not=1.$
Define two non-zero traces $\tau_1, \tau_2$ as follows:
\beq
\tau_1(a):=\tau(p\pi_\tau(a))\andeqn \tau_2:=\tau((1-p)\pi_\tau(a))\rforal a\in A.
\eneq
Note that $0<\|\tau_i\|\le 1,$ $i=1,2.$ 
Put $t_i=\tau_i/\|\tau_i\|,$ $i=1,2.$ Then $t_i\in T(A),$ $i=1,2$ and
\beq
\tau=\|\tau_1\|t_1+\|\tau_2\|t_2.
\eneq
This contradicts the fact that $\tau$ is an extremal point of $T(A).$
\fi
Since $A$ is non-elementary and simple, $A$ has no finite dimensional quotient. 
Therefore, by Proposition 1.3 of \cite{CETW}, $N:=\pi_\tau(A)''$ is a type II$_{1}$ von Neumann 
algebra, where $\pi_\tau$ is  the  GNS-representation given by 
the tracial state $\tau.$
 By Connes' theorem \cite{Ca}, since $A$ is nuclear,  $N$ is hyperfinite.
 By Proposition 1.6 of \cite{CETW}, for each $n\in \N,$ there is a unital embedding
 $\phi_0: M_n\to (l^\infty(N)/I_{\tau,\varpi})\cap N'.$
 %
%Since $A$ is nuclear,  $N$ is an injective  ${\rm II}_1$-factor and hence McDuff (see  Theorem 5.1 of \cite{Ca}).
%It follows that there is a unital embedding $\phi: M_k\to (\Nuv)\cap N'\cong (\Avt)\cap A'.$
%
By Theorem \ref{TSR} above,  applying Proposition 1.8 of \cite{CETW} (see 
Proposition 1.2.4  of \cite{WZ}), 
  there is an order zero  \cpc\, $\psi: M_k\to (\Av)\cap A'$ such that
$\phi=\pi\circ \psi$ as desired.  
\end{proof}
%\fi

\begin{lem}\label{LLpproj-2}
Let $A$ be a separable simple \CA\, with ${\rm Ped}(A)=A,$  $S\subset T(A)$ a compact subset 
%$\tau\in T(A),$ 
and $p\in \AvS$  a projection.
Suppose that  $A$ has tracial approximate oscillation zero.
%(and all 2-quasitraces of $A$ are traces).
Then there is a permanent projection lifting 
$\{e_n\}\in l^\infty(A)$ of projection  $q=\Pi_\varpi(\{e_n\})$ such that 
$\pi_S(\{e_n\})=p$ and $\Phi_S(q)=p,$
where $\pi_S: l^\infty(A)\to \AvS$ and $\Phi_S: l^\infty(A)/I_{_{\Tw,\varpi}}\to \AvS$ 
and $\Pi_\varpi: l^\infty(A)\to l^\infty(A)/I_{_{\Tw, \varpi}}$
are quotient maps, respectively.

Moreover, if $\{a_n\}\in l^\infty(A)_+^{\bf 1}$ is a given lifting of $p,$
we may choose $e_n\in \Her(a_n)_+^{\bf 1},$ $n\in \N.$
\end{lem}

\begin{proof}
Let $\{a_n\}\in l^\infty(A)_+^{\bf 1}$ be such that 
$\pi_S(\{a_n\})=p.$ Since $A$ has tracial approximate oscillation zero, 
for each $n,$ there exists 
$b_n\in \Her(f_{1/4}(a_n))_+^{\bf 1}$ such that
\beq
\|f_{1/4}(a_n)-b_n\|_{_{2,\Tw}}<1/(n+1)\andeqn \omega(b_n)<1/(n+1)^2,\,\, n\in \N.
\eneq
It follows that 
\beq
p=\pi_S(f_{1/4}(a_n))=\pi_S(\{b_n\}) \andeqn  \Pi_\varpi(\{f_{1/4}(a_n)\})=\Pi_\varpi(\{b_n\}).
\eneq
Choose $m(n)\in \N$ such that (see \ref{DefOS1})
\beq\label{Lpproj-2-e-3}
d_\tau(b_n)-\tau(b_n^{2/m(n)})<1/n^2\rforal \tau\in \Tw.
\eneq
Since $f_{1/8}(a_n)b_n^{1/m(n)}=b_n^{1/m(n)}$ for all $n\in \N,$  we have 
%(recall that $\pi_\tau: l^\infty(A)\to
%l^\infty(A)/I_\tau$ is the quotient map)
\beq
p=\pi_S(\{b_n\})\le \pi_S(\{b_n^{1/m(n)}\})\le \pi_S(\{f_{1/8}(a_n)\})=f_{1/8}(\pi_S(\{a_n\}))=p.
\eneq
Put $e_n=\{b_n^{1/m(n)}\},$ $n\in \N.$ Then $p=\pi_S(\{e_n\}).$ 
Note that \eqref{Lpproj-2-e-3} also implies that 
\beq
\|b_n^{1/m(n)}-(b_n^{1/m(n)})^2\|_{_{2, \Tw}}<1/n,\,\, n\in \N.
\eneq
It follows that
\beq
\Pi_\varpi(\{e_n\})^2=\Pi_\varpi(\{e_n\}).
\eneq
Then $q=\Pi_\varpi(\{e_n\})$ is a projection. By \eqref{Lpproj-2-e-3} and (2) of Proposition \ref{Dpproj},
%6.2?? of \cite{FLosc}, 
$\{e_n\}$ is a permanent projection lifting of $q.$
Moreover 
$\Phi_S(q)=p.$  

For the last statement, just recall that $e_n\in \Her(f_{1/4}(a_n))_+^{\bf 1},$ $n\in \N.$
\end{proof}
%
%%%%%%%%%%%%%%%%
\iffalse
\begin{cor}
$\Avt$ has real rank zero. 
\end{cor}

\begin{proof}
One proves this directly.
However, it follows from ?, since $A$ has $T$-tracial oscillation zero, 
$l^\infty(A)/I_{_{\Tw}}$ has real rank zero.  It follows that 
$l^\infty(A)/I_{_{\Tw, \varpi}}$ has real rank zero. It follows that $J:=I_{\varpi,\tau}/I_{_{\Tw,\varpi}}$
has real rank zero.   By Lemma \ref{LLpproj-2}, projections from $\Avt$ lifts to a projection 
in $l^\infty(A)/I_{_{\Tw, \varpi}}.$ It follows from  Theorem 3.2 of \cite{sZhang-I} (see also Theorem  3.14
of \cite{BP})
 that $(l^\infty(A)/I_{_{\Tw, \varpi}})/J\cong 
\Avt$ has real rank zero. 

\end{proof}
\fi
%%%%%%%%%%%%%%%%%%%%
%
%%%%%%%%%%
It is proved in \cite{FLosc} (see Theorem 6.4 of \cite{FLosc}) that, when $A={\rm Ped}(A)$ and $A$ has tracial  approximate oscillation zero, 
$l^\infty(A)/I_{_{\Tw,\N}}$ has real rank zero\footnote{In \cite{FLosc},  it is actually 
shown that $l^\infty(A)/I_{_{\Qw}}$ has real rank zero. Since $l^\infty(A)/I_{_{\Tw}}$
is a quotient of $l^\infty(A)/I_{_{\Qw}},$ it also has real rank zero.}. 
It then follows that $I_{_{S, \varpi}}/I_{_{\Tw, \varpi}}$ and $\Cq$ have real rank zero.
The short exact sequence considered in the next lemma is 
\beq
{\small{0\to   I_{_{S,\varpi}}/I_{_{\Tw, \varpi}}\to     l^\infty(A)/I_{_{\Tw,\varpi}}\to l^\infty(A)/I_{_{S,\varpi}}\to 0.}}
\eneq
Thus  the next lemma is a 
% is  a  
%modified
% 
non-$\sigma$-unital  version of the Elliott Lifting Lemma (see Corollary 3.3 of  \cite{Ellaut}). 

\begin{lem}\label{LLlifting-2}
Let $A$ be a separable simple \CA\,  with ${\rm Ped}(A)=A$ 
which has tracial approximate oscillation zero 
%(and all quasi-traces are traces), 
and  let $S\subset T(A)$ be a compact subset.
%$\tau\in T(A).$
Suppose that $\phi: M_k\to l^\infty(A)/I_{S,\varpi}$ 
is a \hm.
Then,
there exists an order zero  \cpc\, $\Psi=\{\Psi_n\}: M_k\to l^\infty(A)$ 
and a \hm\,  $\psi: M_k\to l^\infty(A)/I_{_{\Tw, \varpi}}$ 
such that 
$\{\Psi_n(1_k)\}$ is a permanent projection lifting 
of the projection  $\psi(1_k),$  $\Phi_S\circ \psi=\phi$ 
(recall that $\Phi_S: \Cq\to \Avt$ is the quotient map) and 
$\Pi_\varpi\circ \Psi=\psi.$
\end{lem}

\begin{proof}
Define  a \hm\, $\phi_c: C_0((0,1])\otimes M_k\to l^\infty(A)/I_{_{S, \omega}}$ by
$\phi_c(f\otimes e_{i,j})=f(1)\phi(e_{i,j})$ for all $f\in C_0((0,1])$ and $1\le i,j\le 1.$ 
Since $C_0((0,1])\otimes M_k$ is projective, there 
exists a \hm\, $\Psi_c: C_0((0,1])\otimes M_k\to 
%an order zero \cpc\, map $\Psi': M_k\to 
l^\infty(A)$
such that  $\pi_S\circ \Psi_c=\phi_c.$
Define $\psi_c=\Pi_\varpi\circ \Psi_c: C_0((0,1])\otimes M_k\to l^\infty(A)/I_{_{\Tw,\varpi}}.$ 
%Let $\Phi_\tau: l^\infty(A)/I_{_{\Tw,\varpi}}\to \Avt$ be the quotient map.
Then $\Phi_S\circ \psi_c=\phi_c.$ 
We may write $\Psi_c=\{\wtd \Psi_c^{(m)}\}_{m\in \N},$ 
where $\wtd\Psi_c^{(m)}: C_0((0,1])\otimes M_k\to A$ is a \hm.
%Consider the \hm\, $H: C_0((0,1])\otimes M_k\to \Cq$ 
%such that $H(\iota\otimes e_{i,j})=\psi'(e_{i,j}),$ $1\le i,j\le j.$ 

In what follows we will denote by $\imath$ the identity function on $[0,1].$
Put $\{a_n\}=\Psi_c(\imath\otimes 1_k)\in l^\infty(A)_+^{\bf 1}$
and $\{a_n^{(1)}\}=\Psi_c(\imath\otimes e_{1,1})\in l^\infty(A)_+^{\bf 1}.$ 
%such that  $a_n=\Psi_c(\imath\otimes 1_k).$ 
Thus
$\pi_S(\{a_n\})=\psi_c(\imath\otimes 1_k)$  and $\pi_S(\{a_n^{(1)}\})=\psi_c(\imath\otimes e_{1,1})$
 (recall 
that $\pi_S: l^\infty(A)\to\AvS$ is the quotient map).
%
\iffalse
We may write (recall that $\Psi_c$ is a \hm)
\beq
f_{1/4}(a_n)=\diag(\Psi_c(f_{1/4}(\imath\otimes e_{1,1})), \Psi_c(f_{1/4}(\imath\otimes e_{2,2})),...,\Psi_c(f_{1/4}(\imath\otimes e_{k,k}))).
\eneq
\fi
%%%%%%%%%%%%%%%%
 By Lemma \ref{LLpproj-2}, there is $e_n^{(1)}\in \Her(f_{1/4}(a_n^{(1)}))_+^{\bf 1}$ 
 for each $n$ 
such that $\{e_n^{(1)}\}$ is a permanent projection lifting  of a projection $q_1=\Pi_\varpi(\{e_n^{(1)}\})$ 
and $\pi_S(\{e_n^{(1)}\})=\phi(e_{1,1})=\phi_c(\imath\otimes e_{1,1}).$ 
Put $w_{j,1}=\Psi_c(f_{1/8}(\imath)\otimes e_{j,1})$   and $w_{j,1,n}=\wtd \Psi_c^{(n)}(f_{1/8}(\imath)\otimes e_{j,1}),$
$n\in \N,$ 
($2\le j\le k$).
Let $w_{j,1,n}=v_{j,1,n}|w_{j,1,n}|$ be the polar decomposition of $w_{j,1,n}$ in $A^{**}.$
Then $\Pi_\varpi(\{w_{j,1,n}\}\{e_n^{(1)}\}\{w_{j,1,n}^*\})=\Pi_{\varpi}(\{v_{j,1,n}e_n^{(1)}v_{j,1,n}\})$ is a projection.
Then $\pi_S(w_{1,j}\{e_n^{(1)}\}w_{1,j}^*)=\phi(e_{j,j})=\phi_c(\imath\otimes e_{j,j})$ ($2\le j\le k$).
Define, for each $n\in \N,$ 
\beq
e_n=e_n^{(1)}+\sum_{j=2}^k w_{j,1,n}e_n^{(1)}w_{j,1,n}^*\in \Her(f_{1/8}(a_n)).
\eneq
Since $\wtd \Psi_c^{(n)}$ is a \hm\, for each $n,$  
$\{e_n^{(1)}, w_{2,1,n}e_n^{(1)}w_{2,1,n}^*, w_{3,1,n}e_n^{(1)}w_{3,1,n}^*,..., w_{k,1,n}e_n^{(1)}w_{k,1,n}^*\}$
are mutually orthogonal.
%and $w_{i,1,n}e_n^{(1)}w_{i,1,n}^*$ are mutually orthogonal (if $i\not=j$ and $i,j\ge 2$).
Moreover,  $\{\Pi_\varpi(\{e_n^{(1)}\}),\,\Pi_\varpi(\{w_{j,1,n}e_n^{(1)}w_{j,1,n}^*\}_{n\in \N}): 2\le j\le k\}$ 
is a set of 
%and $\Pi_\varpi(\{e_n^{(1)}\})$ are also 
mutually orthogonal projections.
%for each $j\ge 2.$ 
It follows that 
%One checks that 
$q=\Pi_\varpi(\{e_n\})$ is a projection (recall that $q_1$ is a projection).
% (as $\Psi_c$ is a \hm). 
%For each $m\in \N,$
Note that, since $e_n^{(1)}\in f_{1/4}(a_n^{(1)}),$ 
\beq
w_{j,1,n}e_n^{(1)}w_{j,1,n}^*w_{j,1,n} e_n^{(1)} w_{j, 1,n}^*&=&
w_{j,1,n}e_n^{(1)} \Psi_c^{(n)}(f_{1/8}(\iota)^2\otimes e_{1,1})e_n^{(1)}w_{j,1,n}^*\\
&=&w_{j,1,n}(e_n^{(1)})^2w_{j,1,n}^*.
\eneq
Hence, for any $m\in \N,$
\beq
(w_{j,1,n}e_n^{(1)}w_{j,1,n}^*)^m=w_{j,1,n}(e_n^{(1)})^m w_{j,1,n}^*.
\eneq
It follows that, for $f\in C_0((0, 1]),$
\beq
f(w_{j,1,n}e_n^{(1)}w_{j,1,n}^*)=w_{j,1,n}f(e_n^{(1)})w_{j,1,n}^*.
\eneq
In particular, for any $m(n)\ge 2,$ 
\beq
(w_{j,1,n}e_n^{(1)}w_{j,1,n}^*)^{1/m(n)}=w_{j,1,n}(e_n^{(1)})^{1/m(n)}w_{j,1,n}.
\eneq
Since $\{e_n^{(1)}\}$ is a permanent projection lifting,  
\beq
&&\Pi_\varpi(\{(w_{j,1,n}e_n^{(1)}w_{j,1,n}^*)^{1/m(n)}\})=\Pi_\varpi(\{w_{j,1,n}(e_n^{(1)})^{1/m(n)}w_{j,1,n}\})\\
&&=\Pi_\varpi(\{w_{j,1,n}\})\Pi_\varpi(e_n^{(1)})\Pi_\varpi(\{w_{j,1,n}\})=\Pi_\varpi(\{w_{j,1,n}\}\{e_n^{(1)}\}\{w_{j,1,n}^*\})\\
&&=\Pi_{\varpi}(\{v_{j,1,n}e_n^{(1)}v_{j,1,n}\}).
\eneq
Hence 
$\{w_{j,1,n}e_n^{(1)}w_{j,1,n}\}$ is also  permanent projection lifting,
% of 
%$\phi(e_{j,j}),$ 
$2\le j\le k.$ Recall that 
$$e_n^{(1)}, w_{2,1,n}e_n^{(1)}w_{2,1,n}^*, w_{3,1,n}e_n^{(1)}w_{3,1,n}^*,..., w_{k,1,n}e_n^{(1)}w_{k,1,n}^*$$
are mutually orthogonal. Hence 
%
%Then 
$\{e_n\}$ is a permanent  projection lifting of $q$
and $\pi_S(\{e_n\})=\phi(1_k)=\phi_c(\imath\otimes 1_k).$
Note that  
\beq\label{LL-2-e1}
&&\{e_n\}\Psi_c(f_{1/8}(\imath)\otimes 1_k)=\{e_n\}
\andeqn\\
&&\{e_n\}\Psi_c(f_{1/8}(\imath)\otimes e_{i,j})=\Psi_c(f_{1/8}(\imath)\otimes e_{i,j})\{e_n\},\,\,1\le i,j\le k.
\eneq

Define  an order zero \cpc\, $\Psi: M_k\to 
%an order zero \cpc\, map $\Psi': M_k\to 
l^\infty(A)$ by $\Psi(e_{i,j})=\{e_n^{1/2}\}\Psi_c(f_{1/8}(\imath)\otimes e_{i,j})\{e_n^{1/2}\}$
($1\le i,j\le 1$).   Define $\psi: M_k\to \Cq$ by $\psi=\Pi_{\varpi}\circ \Psi.$ 
It is an order zero \cpc. 
By \eqref{LL-2-e1}, $\psi(1_k)=q$ which is a projection. 
Hence 
$\psi$ is a \hm.
%
%
%
%Let $\psi'': M_k\to \Cq$ by $\psi''(e_{i,j})=H(f_{1/8}(\iota\otimes e_{i,j})),$ where 
%$\iota(r)=r$ for all $r\in [0,1],$ ($1\le i,j\le k$).
%Then $q\psi''(e_{i,j})=\psi''(e_{i,j})q$ for all $1\le i,j\le k.$  Define an order zero \cpc\,   $\psi: M_k\to \Cq$ by 
%$\psi(x)=q\psi''(x)q$  for all $x\in M_k.$  Note $\psi(1_k)=q$ 
%is a projection. 
%Hence 
% $\psi$ is a \hm\, with $\psi(1_k)=q.$
Moreover, since $\pi_S(\{e_n\})=\phi(1_k),$   we also have $\Phi_S\circ \psi=\phi.$ 

\end{proof}

\begin{lem}\label{LLlifting-2N}
Let $A$ be a 
%non-elementary 
separable simple \CA\, with ${\rm Ped}(A)=A$ 
which has tracial approximate oscillation zero
% (and all quasi-traces are traces) 
and let $S\subset T(A)$ be a compact subset.
%$\tau\in T(A).$
Suppose that $\phi: M_k\to (l^\infty(A)\cap A')/I_{_{S,\varpi}}$ 
is a \hm.

(1) There
are  order zero  \cpc s $\Psi'=\{\Psi_n'\}: M_k\to l^\infty(A)\cap A'$  and 
%and 
%an 
%order zero 
%\cpc\, map
$\psi': M_k\to \Cqc$
such that $\Pi_\varpi\circ \Psi'=\psi'$ and $\Phi_S\circ \psi'=\phi.$

(2) Moreover,  there is also an order zero \cpc\, $\Psi: M_k\to l^\infty(A)$ such that
$\{\Psi_n(1_k)\}$ is a permanent projection lifting 
of the projection  $q=\Pi_\varpi(\{\Psi_n(1_k)\}),$  $q\psi'(b)=\psi'(b)q$ for all $b\in M_k$ and
$\psi: M_k\to \Cq$ defined by $\psi(b)=q\psi'(b)$  for all $b\in M_k$ 
is a \hm. 

(3) Furthermore,
$\Phi_S\circ \psi=\phi$ and $\Pi_\varpi\circ \Psi=\psi.$ 
\end{lem}

\begin{proof}
The proof is contained in that of Lemma \ref{LLlifting-2}.
Let  us repeat some of the same argument.

%As in Lemma \ref{LSato-2},
Using  the fact that \CA\,  $C_0((0,1])\otimes M_k$ is projective, 
one obtains  
an order zero \cpc\, $\psi': M_k\to \Cqc$
such that $\Phi_S\circ \psi'=\phi.$ 
Let $\Psi'=\{\Psi_n'\}: M_k\to l^\infty(A)\cap A'$ be 
an order zero \cpc\, lifting of $\psi'.$  So (1) follows.

As in  the proof of Lemma \ref{LLlifting-2}, 
put
%$\{a_n\}\in l^\infty(A)_+^{\bf 1}$ such that 
$\pi_\tau(\{a_n\})=\psi_c(\imath\otimes 1_k)\in l^\infty(A)_+^{\bf 1}.$  
%By Lemma \ref{LLpproj-2}, t
There is $e_n\in \Her(f_{1/4}(a_n))_+^{\bf 1}$  for each $n$ 
such that $\{e_n\}$ is a permanent projection lifting  of a projection $q:=\Pi_\varpi(\{e_n\})$ 
and $\pi_S(\{e_n\})=\phi(1_k).$
%
%=\phi_c(\imath\otimes 1_k).$ 
%
Note that, as in the proof of Lemma \ref{LLlifting-2},     we may also require that
%\beq\label{LL-2N-e1}
%&&
$\{e_n\}\{f_{1/8}(\Psi_n'(1_k))\}=\{e_n\}$ and 
%\andeqn\\
%&&
$\{e_n\}\{(f_{1/8}(\Psi_n'))(e_{i,j}))=\{(f_{1/8}(\Psi_n'))(e_{i,j})\}\{e_n\},$
$1\le i,j\le k$
%\eneq
(using functional calculus for order zero maps  as in the proof of \ref{LLlifting-2}).
In particular, $q\psi'(b)=\psi'(b)q$ for all $b\in M_k.$

Define  an order zero \cpc\, $\Psi: M_k\to 
%an order zero \cpc\, map $\Psi': M_k\to 
l^\infty(A)$ by $\Psi(e_{i,j})=\{e_n^{1/2}\}( f_{1/8}(\Psi))(e_{i,j})\{e_n^{1/2}\}$
($1\le i,j\le 1$).   Then  $\{\Psi_n(1_k)\}=\{e_n\}$ is a permanent projection lifting of $q.$
%Moreover $\Pi_\varpi\circ \Psi=\Pi_\varpi\circ \Psi'=\psi'.$ 
%Thus (1) follows.
 Define $\psi: M_k\to \Cq$ by $\psi=\Pi_{\varpi}\circ \Psi.$ 
It is an order zero \cpc. Moreover, $\psi(b)=\psi'(b)q$ for all $b\in M_k.$
%By \eqref{LL-2N-e1}, 
%Recall that $\psi(1_k)=q$ which is a projection. 
%Hence 
Since $\psi(1_k)=q,$
$\psi$ is a \hm.
%
%
%
%Let $\psi'': M_k\to \Cq$ by $\psi''(e_{i,j})=H(f_{1/8}(\iota\otimes e_{i,j})),$ where 
%$\iota(r)=r$ for all $r\in [0,1],$ ($1\le i,j\le k$).
%Then $q\psi''(e_{i,j})=\psi''(e_{i,j})q$ for all $1\le i,j\le k.$  Define an order zero \cpc\,   $\psi: M_k\to \Cq$ by 
%$\psi(x)=q\psi''(x)q$  for all $x\in M_k.$  Note $\psi(1_k)=q$ 
%is a projection. 
%Hence 
% $\psi$ is a \hm\, with $\psi(1_k)=q.$
Moreover, since $\pi_S(\{e_n\})=\phi(1_k),$   we have $\Phi_S\circ \psi=\phi.$   So (2) and (3) hold.

\end{proof}

\begin{lem}\label{LIunit}
Let $A$ be a separable simple \CA\, with  nonempty compact $T(A).$
Suppose that
 (1) $\partial_e(T)=Y\cup Z$ and $Y\cap Z=\emptyset,$ 
 (2)   
 $Y$ is compact, $Z$ is Borel and,
 (3)   $M_Y=
 \overline{{\rm conv}(Y)}$ and  $M_Z= \overline{{\rm conv}(Z)}.$
 
 Let $I_{_{M_Z, \varpi}}=\{ \{x_n\}\in l^\infty(A): \lim_{n\to\varpi}\sup\{\|x_n\|_{_{2, \tau}}:\tau\in M_Z \}=0\}.$
 Suppose that $\{e^{\lambda}\}$ is an approximate identity for $I_{_{M_Z, \varpi}}.$
 Then, for any $\ep>0,$ there exists $\lambda_0$ such that
 \beq
&& \|\Pi_\varpi(e^\lambda)-\Pi_\varpi((e^\lambda)^2)\|<\ep\andeqn\\
 &&\lim_{n\to\varpi}\sup\{\tau(1-(e_n^\lambda)^2): \tau\in M_Y\}<\ep
% \tau(e^\lambda)>1-\ep\rforal \tau\in T_\varpi(A)
 \eneq
 for all $\lambda\ge \lambda_0,$
%Moreover, for any $\ep>0,$ if $\lambda\ge \lambda_0,$ 
%\beq
%\lim_{n\to\varpi}\sup\{\tau(1-(e_n^\lambda)^2): \tau\in M_Y\}<\ep,
%\eneq
where $\{e_n^\lambda\}\in l^\infty(A)_+^{\bf 1}$ is a lifting of $e^\lambda.$
\end{lem}

\begin{proof}
For any $\sigma\in (0,1/4),$ by Corollary \ref{Laff}, 
there is $f_0\in \Aff(T(A))$ such that
\beq
0\le f_0\le 1,\,\,(f_0)|_{M_Z}=0\andeqn  (f_0)|_{M_Y}=1-\sigma.
\eneq
For any $\eta\in (0, \sigma/4),$ then
$f_0+\eta\cdot 1_{T(A)}$ is a strictly positive affine  continuous function. 
By Theorem 9.3 of \cite{Linclr1}\footnote{The  unital condition of Theorem 9.3 of \cite{Linclr1} 
is not required. 
%assumes
%$A$ is unital. However, 
The proof  is based on  Proposition 2.7, Theorem 2.9   and   Corollary 6.4 of \cite{CP}   which do not require 
that $A$ is unital.}
, we obtain a sequence of $\{a_n\}\in A_+^{\bf 1}$ such 
that 
\beq
\tau(a_n)>1-1/n\rforal \tau\in M_Y\andeqn \tau(a_n)<1/n\rforal \tau\in M_Z.
\eneq
It follows that $a:=\{a_n\}\in I_{_{M_Z, \N}}.$ 

Note that $\lim_\lambda\|\Pi_\varpi((a-a^{1/2}(e^\lambda)^2a^{1/2})^{1/2})\|=0.$
Write $e^\lambda= \{e_n^{\lambda}\}\in l^\infty(A)_+^{\bf 1}.$ 
%such 
%that $\Pi_\varpi(\{e_n^\lambda\})=e^\lambda.$
It follows that, for any $\ep\in (0,1/4),$
there is $\lambda_0$ satisfying the following:  for any $\lambda\ge \lambda_0,$ 
there exists ${\cal P}\in \varpi$ such that, if $n\in {\cal P}_\lambda,$
\beq\label{LIunit-10}
\sup\{\tau(a_n-a_n^{1/2} (e_n^\lambda)^2 a_n^{1/2}): \tau\in T(A)\}<(\ep/8)^2.
\eneq
Since $(e^\lambda)^{1/2}\in I_{_{M_Z, \varpi}},$ \wilog, we may assume that,
for all $n\in {\cal P}_\lambda,$
\beq\label{LIunit-11}
\sup\{\tau(e_n^\lambda):\tau\in M_Z\}<(\ep/4)^2.
\eneq
We may assume that $1/n<(\ep/8)^2$ for all $n\in {\cal P}_\lambda.$ 
We may also assume that, for all $n\in {\cal P}_\lambda,$
\beq\label{LIunit-11+}
\sup\{\tau(a_n): \tau\in M_Y\}>1-(\ep/8)^2\andeqn \sup\{\tau(a_n): \tau\in M_Z\}<(\ep/8)^2.
\eneq
By \eqref{LIunit-10} and \eqref{LIunit-11+},  we have, for $n\in {\cal P}_\lambda,$ 
\beq
\sup\{1-\tau(a_n^{1/2}(e_n^\lambda)^2a_n^{1/2}): \tau\in M_Y\}<(\ep/4)^2.
\eneq
Hence, for $n\in {\cal P}_\lambda,$
\beq\label{LIunit-16}
\sup\{1-\tau((e_n^\lambda)^2):\tau\in M_Y\}\le \sup\{1-\tau(a_n^{1/2}(e_n^\lambda)^2a_n^{1/2}): \tau\in M_Y\}<(\ep/4)^2.
\eneq
It follows that, for all $n\in {\cal P}_\lambda,$
\beq
\sup\{\tau(e_n^\lambda-(e_n^\lambda)^2): \tau\in M_Y\}<(\ep/4)^2.
\eneq
We also have, by \eqref{LIunit-11}, for all $n\in {\cal P}_\lambda,$
\beq
\sup\{\tau(e_n^\lambda-(e_n^\lambda)^2): \tau\in M_Z\}<(\ep/2)^2.
\eneq
Since $T$ is the convex hull of $M_Y$ and $M_Z,$ 
we have, for all $n\in {\cal P}_\lambda,$ 
\beq
\sup\{\tau(e_n^\lambda-(e_n^\lambda)^2):\tau\in T(A)\}<(\ep/2)^2.
\eneq
Hence (for each $\lambda\ge \lambda_0$)
\beq
\|\Pi_\varpi(e^\lambda-(e^\lambda)^2)\|<\ep.
\eneq
This  and \eqref{LIunit-16}  hold for any $\lambda\ge \lambda_0.$ The lemma follows.
\end{proof}

\begin{lem}\label{LIproj1}
Let $A$ be a separable simple \CA\, with  nonempty compact $T(A)$ which satisfies  condition (C). 
Let $\partial_e(T(A))=\cup_{n=1}^\infty X_n,$ where 
$\{X_n\}$ satisfies (1) and (2) in \ref{Dc1}.
Suppose that $K=M_{\partial_e(T)\setminus \cup_{i=1}^JX_{i+l}}$ for some integer $J,l\in \N.$
Then, for any separable \SCA\, of $B\subset l^\infty(A),$  there exists a sequence 
$\{e_n\}\in 
(I_{_{K, \varpi}})_+^{\bf 1}$  satisfying the following:

(1) $\lim_{n\to \infty}\| {[b_n,\, e_n]}\|=0\tforal \{b_n\}\in B,$

(2) $\Pi_\varpi(\{e_n\})$ is a projection,

(3) $\lim_{n\to \varpi}\|c_ne_n-c_n\|_{_{2, T(A)}}=0,$  if $c=\{c_n\}\in B\cap I_{_{K, \varpi}}.$

(4) $\lim_{n\to\varpi} \sup\{1-\tau(e_n^2): \tau\in M_{\cup_{i=1}^JX_{i+l}}\}=0.$
\end{lem}

\begin{proof}
Note that, by \ref{Dc1},  $K$ is compact (as a closed subset of $T(A)$).
Using a quasi-central approximate identity of $I_{_{K, \varpi}},$  we obtain 
 a sequence $\{\{e_n^{(m)}\}_{n\in \N}\}\subset (I_{_{K, \varpi}})_+^{\bf 1}$ such that
(in $l^\infty(A)$)
\beq\label{LIproj1-1}
&&\lim_{m\to\infty}\|e^{(m)}b-be^{(m)}\|=0\rforal b\in B\andeqn\\\label{LIproj1-2}
&&\lim_{m\to\infty}\|e^{(m)}c-c\|=0\rforal c\in B\cap I_{_{K, \varpi}},
\eneq
where $e^{(m)}=\{e_n^{(m)}\}_{n\in \N}.$
Moreover,  for any $\ep>0,$ by Lemma \ref{LIunit},   there is $m_0\in \N$ such that, for $m\ge m_0,$
\beq\label{LIproj1-4}
&&\lim_{n\to \varpi}\|e_n^{(m)}-(e_n^{(m)})^2\|_{2,T(A)}<\ep/2\andeqn\\\label{LIproj1-5}
&&\lim_{n\to\varpi}\sup\{1-\tau((e_n^{(m)})^2): \tau\in M_{\cup_{i=1}^JX_{i+l}}\}<\ep/2.
\eneq
We will use Kirchberg's $\ep$-test (we will use the version of Lemma 3.1 of \cite{KR}).
 Let $\{\{b_n^{(k)}\}_{n\in \N}\}$ be a dense sequence of $B$ and 
$\{b_n^{(2k)}\}$ be a dense sequence of $B\cap I_{_{K, \varpi}}.$
Let $Y_n=
%( I_{_{K, \varpi}})_+^{\bf 1},
A_+^{\bf 1},$ $n\in \N$ (as $X_n$ in Lemma 3.1 \cite{KR}).
Define, for each $x\in Y_n,$ 
\beq
f_n^{(1)}(x)&=&\|x\|_{2, K}+\sup\{1-\tau(x^2):\tau\in M_{\cup_{i=1}^J X_{i+l}}\},\andeqn\\
f_n^{(k+1)}(x)&=&\|xb_n^{(k)}-b_n^{(k)}x\|+\|xb_n^{(2k)}-b_n^{(2k)}\|+\|x-x^2\|_{2, T(A)}, \,\,k\in \N.
\eneq
We have that, since $\{e_n^{(m)}\}_{n\in \N}\in I_{_{K, \varpi}},$ for each $m\in \N,$ 
\beq
\lim_{n\to \varpi}\|e_n^{(m)}\|_{2, K}=0
\rforal m.
\eneq
By \eqref{LIproj1-5}, \eqref{LIproj1-1}, \eqref{LIproj1-2} and  \eqref{LIproj1-4}, if $\ep>0$ is given, 
by choosing $x_n=e_n^{(m)}$ for some large $m\ge m_0$), we obtain 
\beq
\lim_{n\to \varpi} f_n^{(k)}(x_n)<\ep,\,\,1\le k\le K.
\eneq
Thus, by Kirchberg's $\ep$-test (see Lemma 3.1 of \cite{KR}),  we obtain $e=\{e_n\}\in l^{\infty}(A)_+^{\bf 1}$ such that
\beq
\lim_{n\to \varpi}f_n^{(k)}(e_n)=0\rforal k\in \N.
\eneq
By $\lim_{n\to\varpi}f_n^{(1)}(e_n)=0,$ we have $e=\{e_n\}\in (I_{_{K, \varpi}})_+^{\bf 1}.$  Moreover
\beq
\lim_{n\to\varpi}\sup\{1-\tau(e_n^2): \tau\in M_{\cup_{i=1}^JX_{i+l}}\}=0\andeqn \lim_{n\to\varpi}\|e_n-(e_n)^2\|_{2, T(A)}=0.
\eneq
Hence $\Pi_\varpi(\{e_n\})$ is a projection. Furthermore,  
since $\{b^{(k)}\}$ is dense in $B$ and $\{b^{(2k)}\}$ is dense in $B\cap I_{_{K, \varpi}},$ 
respectively, we conclude that
\beq
\lim_{n\to\varpi}\|b_ne_n-e_nb_n\|=0\rforal b=\{b_n\}\in B \andeqn\\
 \lim_{n\to\varpi}\|c_ne_n-c_n\|=0\rforal c=\{c_n\}\in B\cap I_{_{K, \varpi}}.
\eneq
The lemma the follows.
\end{proof}

\begin{lem}\label{LIJ}
Let $A$ be a separable simple \CA\, with nonempty compact $T(A)$ which satisfies  condition (C). 
Let $\partial_e(T(A))=\cup_{n=1}^\infty X_n,$ where 
$\{X_n\}$ satisfies (1) and (2) in \ref{Dc1}.
Let $K_m=\cup_{i=1}^m X_i,$ $Y_{k,m}=\cup_{j=1}^k X_{j+m}$  and $Z_{k,m}=\partial_e(T)\setminus Y_{k,m},$
$k\in \N.$   Suppose that $\{b_n\}\in (I_{M_{K_m}})_+^{\bf 1}.$ 
Then, for any $\ep>0,$ there is ${\cal P}\in \varpi$ satisfying the following:
for any $n\in {\cal P},$ there exists $m(n)\in \N$ such that, for all $k\ge m(n),$
\beq
\|b_n\|_{_{2, M_{Z_{k,m}}}}<\ep
\eneq
%
%Then 
%\beq
%I_{_{K_m, \varpi}}=\overline{\cup_{n=1}I_{_{Z_n, \varpi}}}.
%\eneq
\end{lem}

\begin{proof}
Fix $\ep\in (0,1/2)$ and $m\in \N.$ If the lemma is false, then
no ${\cal P}\in \varpi$ satisfies the said property. 
However, since $\{b_n\}\in (I_{M_{K_m}})_+^{\bf 1},$ 
there is ${\cal P}\in \varpi$ such that, for any $n\in {\cal P},$ 
\beq\label{LIJ-2}
%\sup\{\tau(b_n): \tau\in K_m\}
\|b_n\|_{_{2, K_m}}<\ep/4.
\eneq
Note that ${\cal P}\cap \{k>n: k\in \N\}\in \varpi.$
Since the  lemma is assumed to be false, for some $n_0\in {\cal P},$ there must be  a sequence 
$\{I(n_0,j)\}_{j\in \N}$ 
of integers  ($I(n_0,j)\to\infty$ as $j\to \infty$) such that
\beq
\|b_{n_0}\|_{_{2, M_{Z_{l(n_0,j),m}}}}\ge \ep.
\eneq
We obtain $\tau_{I(n_0,j)}\in M_{Z_{I(n_0,j),m}}$ such that 
\beq
\|b_{n_0}\|_{_{2, \tau_{I(n_0,j)}}}\ge \ep.
\eneq
Since $T(A)$ is compact, we may assume that $\tau_{I(n_0,j)}\to \tau\in T(A)$ (as $j\to\infty$).
Then, by Proposition \ref{Plimits}, $\tau\in M_{X_1}\subset M_{K_m}.$
It follows that 
\beq
\|b_{n_0}\|_{_{2, \tau}}\ge \ep.
\eneq
%Since $K_m\subset Z_n,$ we have $\overline{\cup_{n=1}I_{_{Z_n, \varpi}}}\subset I_{_{K_m, \varpi}},$ $n\in \N.$
A contradiction (to \eqref{LIJ-2}).
\end{proof}

We now combine an idea of Sato and the proof of  Theorem 3.3 of \cite{KR} (see also Proposition 3.8 of \cite{FLL})
with the notion of  tracial approximate oscillation zero. 

\begin{thm}\label{TLlifting}
Let $A$ be a separable  non-elementary algebraically simple \CA\, with nonempty compact $T(A)$  
satisfying condition (C)
such that $\partial_e(T(A))=\cup_{n=1}^\infty X_n$ which 
satisfies (1) and (2) in \ref{Dc1}. Suppose that $A$ has tracial approximate oscillation zero.
%(and all quasi-traces are traces).
 Let $S:=\cup_{j=1}^mX_n.$ 
 %be as described in \ref{Dc1}. 
%$\tau\in T(A).$
Suppose also that $\phi: M_N\to (l^\infty(A)\cap A')/I_{_{S, \varpi}}$ 
is a \hm\, (for some $N\in \N$).

Then
there exists a \hm\, $\psi: M_N\to (l^\infty(A)\cap A')/I_{_{T(A),\varpi}}$
such that $\Phi_S\circ \psi=\phi.$
%%%%%%%%%%%%%%%%%
\iffalse
(2) Moreover,  $\psi$ satisfies the following condition 
(Os):

there exists a permanent projection lifting $\{\td e_n\}\in (l^\infty(A)\cap A')_+^{\bf 1}$
such that $\Pi_\varpi(\{\td e_n\})=\psi(1_k)$ and, for any $\ep>0$  there exist  $\dt>0$  and 
${\cal Q}\in \varpi$
%$n_1\in \N$ 
such that
\beq
d_\tau(\td e_n)-\tau(f_\dt(\td e_n))<\ep\tforal \tau\in \Tw\andeqn  n\in {\cal Q}.
\eneq
\fi
%%%%%%%%%%%%%%%%%%%%%%%%%%%%%%%%
\end{thm}

\begin{proof}
%By Lemma \ref{TSR},
% \ref{LLlifting-2}, 
Using the fact that \CA\, $C_0((0, 1])\otimes M_N$ is projective, one obtains 
%there exists 
an order zero \cpc\, 
$\psi': M_N\to \Cqcc$ 
%and a permanent projection lifting  $\{e_n\}\in l^\infty(A)_+^{\bf 1}$ 
%of projection $\psi(1_k)$
 such that
$\Phi_S\circ \psi'=\phi.$ 
Let $\{a_n\}\in l^\infty(A)\cap A'$ be such that $\Pi_{\varpi}(\{a_n\})=\psi'(1_N).$
It follows from Lemma \ref{LLlifting-2N}
% \ref{LLpproj-2} 
that there is a  permanent projection lifting  $\{e_n\}\in l^\infty(A)_+^{\bf 1}$ 
of 
a projection $q\in \Ccq$ such that $\Phi_S(q)=\phi(1_N).$
 Moreover  it follows that $e_n\in \Her(f_{1/8}(a_n))_+^{\bf 1}.$
Since $\{e_n\}$ is a permanent projection lifting, by replacing $e_n$ by $e_n^{1/m(n)}$ for some 
$m(n)\in \N,$ we may assume that, for some $1/2>\dt_1>0,$ 
\beq
d_\tau(e_n)-f_{\dt_1}(e_n)<1/2^n\rforal \tau\in \Tw\andeqn n\in \N.
\eneq
Define $\psi_c': C_0((0,1])\otimes M_k\to \Cqcc$ by $\psi'_c(\imath\otimes e_{i,j})=\psi'(e_{i,j})$
($1\le i,j\le k$).
Define $\psi: M_k\to \Ccq$ by $\psi(e_{i,j})=q\psi_c'(f_{1/8}(\imath)\otimes e_{i,j})q.$ The proof of  Lemma \ref{LLlifting-2N}
shows that $\psi''$ is a \hm\,   and $\psi''(1_k)=\Pi_{\varpi}(\{e_n\}).$ 

We will  modify $\{e_n\}$ to obtain a central sequence.

%Put $F_n=\cup_{j=1}^{m+n}X_j$ and $Z_j=\cup_{i=j}^\infty X_i,$ $n, j\in \N.$

Let $B$ be the separable \SCA\, of $l^\infty(A)$  generated by $\iota(A)$ and $\{e_n\}.$ 
Let $\{a_k\}$ be a dense sequence of $A^{\bf 1}.$ 
Put $b_n^{(k)}=e_na_k-a_ke_n.$   Since $q=\Pi_{_{S, \varpi}}(\{e_n\})$ and $\Phi_S(q)=\pi(1_N),$
for each $k\in \N,$ $\{b_n^{(k)}\}_{n\in \N}\in I_{_{S, \varpi}}.$
Fix $\ep\in (0,1/2)$ and $K\in \N.$
By Lemma \ref{LIJ}, there is ${\cal P}\in \varpi,$ satisfying the following:
for any $n\in {\cal P},$ there is $m(n)>m+1$ such that 
\beq\label{TLlifting-12}
\|b_n^{(k)}\|_{_{2, M_{Z_n}}}<\ep/4, \,\, 1\le k\le K,
\eneq
where $Z_n=\cup_{i=m(n)}^\infty X_i,$ $n\in \N.$
Put $F_n=\cup_{i=m+1}^{m(n)-1} X_i$  and $Y_n=\partial_e(T(A))\setminus F_n,$ ($n\in \N$).
We may also write $Y_n=S\sqcup Z_n,$ $n\in \N.$  Hence
$I_{_{M_{Y_n}, \varpi}}\subset I_{_{S, \varpi}}.$
Since $\overline{{\rm conv}(S)}=M_S,$ by condition (C),   we have $I_{_{S, \varpi}}=I_{_{M_S, \varpi}}.$
For $n\not\in {\cal P}$ and $n<\min\{l: l\in {\cal P}\},$ define $Z_{n}=X_{m+1}.$
If $n_p=\max\{l<n:l\in {\cal P}\},$ define $Z_n=Z_{n_p}.$ Then, for $n\not\in {\cal P},$
define $F_n$ and $Y_n$ accordingly.

By Lemma \ref{LIproj1}, there is $\{p_l^{(n)}\}_{l\in \N}\in I_{_{M_{Y_n}, \varpi}}$ such that 
$\Pi_\varpi(\{p_l^{(n)}\})$ is a projection and ($1\le k\le K$)
\beq
&&\lim_{l\to\varpi}\|a_kp_l^{(n)}-p_l^{(n)}a_k\|=0,\\\label{TLlifting-15}
&&\lim_{l\to\varpi}\sup\{1-\tau((p_l^{(n)})^2): \tau\in M_{F_n}\}=0\andeqn\\
&&\lim_{l\to\varpi}\|p_l^{(n)}-(p_l^{(n)})^2\|_{_{2, T(A)}}=0.
\eneq
%In particular, by \eqref{TLlifting-15}, 
Define $d_l^{(1,n,k)}=b_n^{(k)}p_l^{(n)}$ and $d_l^{(2,n,k)}=b_n^{(k)}(1-p_l^{(n)}),$ $l\in \N.$
Then $\{d_l^{(1,n,k)}\}_{l\in \N}\in I_{_{M_{Y_n}, \varpi}}$ and, by \eqref{TLlifting-15}, for 
$1\le k\le K,$ 
\beq\label{TLlifting-16} 
\lim_{l\to \varpi}\|d_l^{(2,n,k)}\|_{_{2, M_{F_n}}}\le 2\lim_{l\to\varpi}\sup\{\tau(1-p_l^{(n)}):\tau\in M_{F_n}\}=0.
\eneq
Note that $\{d_l^{(1,n,k)}\}_{l\in \N}\in I_{_{M_{Y_n},\varpi}}.$

By Lemma \ref{LIproj1},  for each fixed $n,$  there is $\{q_l^{(n)}\}_{l\in \N}\in  (I_{_{M_{Y_n}, \varpi}})_+^{\bf 1}$ 
such that, for $1\le k\le K,$ 
\beq\label{TLlifting-17}
&&\lim_{l\to\infty}\|q_l^{(n)}\|_{_{M_{Y_n},\varpi}}=0,\\\label{TLlifting-18}
&&\lim_{l\to\infty}\|[q_l^{(n)},\, a_k]\|=0,\,\,
%\\\label{TLlifting-19}
\lim_{l\to\infty}\|[q_l^{(n)},\, e_n]\|=0,\\\label{TLlifting-19+}
&&\lim_{l\to\varpi}\|q_l^{(n)}-(q_l^{(n)})^2\|_{_{2, T(A)}}=0\andeqn\\\label{TLlifting-20}
&&\lim_{l\to \varpi}\|d_l^{(1,n,k)}q_l^{(n)}-d_l^{(1,n,k)}\|_{_{2, T(A)}}=0.
\eneq
We estimate that, if $n\in {\cal P},$  by \eqref{TLlifting-20} and \eqref{TLlifting-16}, for $1\le k\le K,$
\beq
&&\hspace{-1in}\lim_{l\to \varpi}\|b_n^{(k)}q_l^{(n)}-b_n^{(k)}\|_{_{2, M_{F_n}}}\le
\lim_{l\to\varpi}\|d_l^{(1,n,k)}q_l^{(n)}-d_l^{(1,n,k)}\|_{_{2,T(A)}}\\\label{Llifting-30}
&&\hspace{1in}+\lim_{l\to\varpi}\|d_l^{(2,n,k)}q_l^{(n)}-d_2^{(2,n,k)}\|_{_{2, M_{F_n}}}=0.
%\\
%&&=0
%\|d_l^{(1,n,k)}q_l^{(n)}-d_l^{(1,n,k)}\|_{_{2,T(A)}}
%+\lim_{l\to\varpi}\|b_n^{(k)}(1-p_l^{(m(n)})(1-q_l^{(n)})\|_{_{2,F_n}}\\
%\le  \|d_l^{(1,n,k)}q_l^{(n)}-d_l^{(1,n,k)}\|_{_{2,T(A)}}+
\eneq
Also, by \eqref{TLlifting-20} and \eqref{TLlifting-12}, if $n\in {\cal P},$ we have that, for $1\le k\le K,$
\beq
&&\hspace{-1in}\lim_{l\to \varpi}\|b_n^{(k)}q_l^{(n)}-b_n^{(k)}\|_{_{2, M_{Z_n}}}\le
\lim_{l\to\varpi}\|d_l^{(1,n,k)}q_l^{(n)}-d_l^{(1,n,k)}\|_{_{2,T(A)}}\\\label{Llifting-31}
&&\hspace{1in}+\lim_{l\to\varpi}\|d_l^{(2,n,k)}q_l^{(n)}-d_2^{(2,n,k)}\|_{_{2, M_{Z_n}}}<\ep/4.
\eneq
Since $b_n^{(k)}\in I_{_{S, \varpi}},$    we also assume that, for all $n\in {\cal P},$ 
\beq\label{Llifting-32}
\|b_n^{(k)}\|_{_{2, S}}<\ep/4, \,\,{\rm whence}\,\,\|b_n^{(k)}q_l^{(n)}-b_n^{(k)}\|_{_{2,, S}}<\ep/2.
\eneq
Note that $T(A)$ is the convex hull of $S, M_{F_n}$ and $M_{Z_n},$ by condition (C). 
We obtain, combining \eqref{Llifting-30}, \eqref{Llifting-31} and \eqref{Llifting-32},
that, for any $n\in {\cal P},$
\beq\label{TLift60}
\lim_{l\to \varpi}\|b_n^{(k)}q_l^{(n)}-b_n^{(k)}\|_{_{2, T(A)}}<\ep/2.
\eneq
Thus, combining with \eqref{TLlifting-17},  \eqref{TLlifting-18},  \eqref{TLlifting-19+}, and  \eqref{TLift60},
%{TLlifting-20}, 
for each $n\in {\cal P},$ there exists $l(n)\in \N$ such that, 
for $1\le k\le K,$ 
\beq\label{TLlifting-25}
&& \|q_{l(n)}^{(n)}\|_{_{2, S}}<1/n^2,\,\, \|[q_{l(n)}^{(n)},\, a_k]\|<1/n^2\andeqn \|[q_{l(n)}^{(n)},\, e_n]\|<1/n^2, \\\label{TLlifting-26}
&&\|q_{l(n)}^{(n)}-(q_{l(n)}^{(n)})^2\|_{_{2, T(A)}}<1/n^2 \andeqn
\|b_n^{(k)}q_{l(n)}^{(n)}-b_n^{(k)}\|_{_{2, T(A)}}<\ep/2.
\eneq

We will apply Kirchberg's $\ep$-test.
Let $W_n=A_+^{\bf 1}$  (as $X_n$ in Lemma 3.1 of \cite{KR}), $n\in \N.$
Define 
\beq\nonumber
f_n^{(1)}(x)&=&\|x\|_{_{2, S}} \,\,\,{\rm{and,\,\,\, for}}\,\, k\in  N, \\\nonumber
%+\sup\{1-\tau(x):\tau\in K\},\andeqn\\
\hspace{-0,2in}f_n^{(k+1)}&=&\|xa_k-a_kx\|+\|xe_n-e_nx\|+\|(1-x)b_n^{(k)}(1-x)\|_{_{2, T(A)}}+\|x-x^2\|_{_{2, T(A)}}.
\eneq
By \eqref{TLlifting-25} and \eqref{TLlifting-26}, for any $\ep>0$ and any $1\le k\le K,$  choosing $q_{l(n)}^{(n)}\in W_n,$ $n\in \N,$ 
as above,  we have that
\beq
\lim_{n\to\varpi}f_n^{(k)}(q_{l(n)}^{(n)})<\ep.
\eneq
Applying Kirchberg's $\ep$-test (see Lemma 3.1 of \cite{KR}), we obtain $c=\{c_n\}\in l^\infty(A)_+^{\bf 1}$ such 
that
\beq
\lim_{n\to\varpi}f_n^{(k)}(c_n)=0\rforal k\in \N.
\eneq
Put ${\bar c}=\Pi_\varpi(\{c_n\}).$ 
Then (since $\lim_{n\to\varpi} f_n^{(1)}(c_n)=0$) we have that  $\{c_n\}\in I_{_{S, \varpi}}$
and  ${\bar c}=\Pi_\varpi(\{c_n\})$ is a projection (as $\lim_{n\to \varpi}\|c_n-c_n^2\|_{_{2, T(A)}}=0$).
Since $\{a_k\}$ is dense in $A,$  we conclude that
$\{c_n\}\in A'$  (since $\lim_{n\to\varpi}\|c_na_k-a_kc_n\|=0$)  and 
\beq\label{6665}
&&{\bar c}\psi''(1_N)=\psi''(1_N){\bar c} \,\,\,\,\,\,\,\,\,({\rm since}\,\,\lim_{n\to\varpi}\|[c_n,e_n]\|=0),\\\label{6666}
&&\hspace{-0.4in}(1-{\bar c})\psi''(1_N)(1-{\bar c})\Pi_\varpi(\iota(a))=\Pi_\varpi(\iota(a))(1-{\bar c}))\psi''(1_N)(1-{\bar c})\rforal a\in A
\eneq
(for the last equality, recall that $b_n^{(k)}=e_na_k-a_ke_n$ 
and $\lim_{n\to\varpi}\|(1-c_n)b_n^{(k)}(1-c_n)\|_{_{2, T(A)}}=0$).

Let ${\bar e}_n=(1-c_n)e_n(1-c_n),$ $n\in \N,$ and 
 ${\bar q}=\Pi_\varpi(\{\bar e_n\}).$   By \eqref{6665}, $q\bar c={\bar c}q.$ Since both $q$ and ${\bar c}$ 
 are projections, so is ${\bar q}.$  We also have 
 ${\bar q}\le q.$ 
 Moreover, $\Pi_\varpi(\{\bar e_n\})=(1-{\bar c})\psi''(1_N)(1-{\bar c})$
and $\{{\bar e_n}\}\in \Pi_\varpi(\iota(A))'$ (by \eqref{6666}). 
  Define $\psi: M_N\to \Ccq$ by 
$\psi(e_{i,j})={\bar q}\psi_c'(f_{1/8}(\imath)\otimes e_{i,j}){\bar q}$ $(1\le i,j\le N$). 
It is an order zero \cpc.
Then (recall that  ${\bar q}$ is a projection and ${\bar q}\le q$), as in the first 
paragraph of this proof,  ${\bar \psi}(1_N)={\bar q}.$ Since $\{\bar e_n\}\in \Pi_\varpi(\iota(A))'$ 
and ${\bar q}$ is a projection, 
$\psi$ is a \hm.  Recall that  $\psi_c'$ maps into $\Cqcc$ and $\{\bar e_n\}\in \Pi_\varpi(\iota(A))',$
 the map $\psi$ is actually a \hm\, from $M_N$ into $(l^\infty(A)/I_{_{T(A),\varpi}})\cap \Pi_\varpi(\iota(A))'.$ 
 %\Cqcc.$  
 By Central Surjectivity  (see  Lemma 1.2 of \cite{CETW}, also Proposition 4.5 (iii) and Proposition 4.6 of \cite{KR}),
 $(l^\infty(A)/I_{_{T(A),\varpi}})\cap \Pi_\varpi(\iota(A))'=\Cqcc.$
 Therefore $\psi$ is a \hm\, from $M_N$  to $\Cqcc.$
 Since $\{c_n\}\in I_{_{S,\varpi}},$ we have
$\Phi_S\circ\psi=\phi.$  The lemma then follows.
\end{proof}
%%%%%%%%%%%%%%%%%%%%%%%
\iffalse
\begin{rem}
It should be warned that while $\{e_n\}$ in Theorem \ref{TLlifting} is a permanent projection 
lifting and $\Pi_\varpi(\{e_n^{1/m(n)}\})=\Pi_\varpi(\{e_n\})$ 
for any integers $\{m(n)\},$ $\{e_n^{1/m(n)}\}$ may not be a central sequence. 
But, of course, $\{f_\dt(e_n)\}$ is a central sequence.
\end{rem}
\fi
%%%%%%%%%%%%%%%%%%%%

\begin{df}\label{DOS}
Let $A$ be a separable simple \CA\, with $A={\rm Ped}(A),$   $S\subset \Tw$ be a compact subset 
and $\phi: M_k\to l^\infty(A)\cap A'/I_{_{S, \varpi}}$ be an order zero  \cpc. 
We say that $\phi$ has  property (Os), if there exists 
an order zero \cpc\, $\Psi=\{\wtd\psi_n\}: M_k\to l^\infty(A)\cap A'$ 
such that
$\{\wtd\psi_n(1_k)\}$ is 
a permanent projection lifting of $\psi(1_k),$
%such that $\Pi_\varpi(\{\td e_n\})=\psi(1_k)$ 
and, for any $\ep>0$  there exist  $\dt>0$  and 
${\cal Q}\in \varpi$
%$n_1\in \N$ 
such that
\beq
d_\tau(\wtd \psi_n(1_k))-\tau(f_\dt(\wtd\psi_n(1_k)))<\ep\tforal \tau\in \Tw 
\tand  n\in {\cal Q}.
\eneq
\end{df}

\begin{prop}\label{LLliftos}
Let $A$ be a separable  non-elementary simple \CA\, with $A={\rm Ped}(A)$ 
and $T(A)\not=\emptyset.$
%and $S\subset \Tw$ be a compact subset. 
%which has $T$-tracial approximate oscillation zero and $\tau\in T(A).$
Suppose that $\phi: M_k\to (l^\infty(A)\cap A')/I_{_{\Tw, \varpi}}$
%{_{\Tw,\varpi}}$ 
is a \hm.
Then $\phi$ has  property 
(Os).
%
\iffalse:
there exists an order zero \cpc\, $\Psi=\{\wtd\psi_n\}: M_k\to l^\infty(A)\cap A'$ 
such that $\wtd\psi_n(1_k)\}$ is 
a permanent projection lifting of $\psi(1_k),$
%such that $\Pi_\varpi(\{\td e_n\})=\psi(1_k)$ 
and, for any $\ep>0$  there exist  $\dt>0$  and 
${\cal Q}\in \varpi$
%$n_1\in \N$ 
such that
\beq
d_\tau(\wtd \psi_n(1_k))-\tau(f_\dt(\wtd\psi_n(1_k)))<\ep\tforal \tau\in \Tw 
\tand  n\in {\cal Q}.
\eneq
\fi
%%%%%%%%%%%%%
\end{prop}

\begin{proof}
Let $\psi_c: C_0((0,1])\otimes M_k\to \Cqcc$ be defined 
by $\psi_c(g\otimes e_{i,j})=g(1)\psi(e_{i,j})$ for all $g\in C_0((0,1])$ and $1\le i,j\le k.$
Since $C_0((0,1])\otimes M_k$ is projective, there 
is a \hm\, $\Psi_c=\{\wtd\psi_{c,n}\}: C_0((0,1])\otimes M_k\to l^\infty(A)\cap A'$ 
such that 
$\Pi_\varpi\circ \Psi_c=\psi_c.$
Note that\\
 $\Pi_\varpi(\{\wtd\psi_{c,n}(\imath\otimes 1_k)\})=\psi(1_k).$
Define $\Psi=\{\wtd\psi_n\}: M_k\to l^\infty(A)\cap A'$ 
by $\wtd\psi_n(e_{i,j})=\wtd\psi_{c,n}(f_{1/4}(\imath)\otimes e_{i,j})$ for $1\le i,j\le k.$
Then $\Psi$ is an order zero \cpc\,
and 
%Since $\Psi(1_k)=f_{1/4}(\Pi_\varpi\circ \Psi_c(\imath\otimes 1_k))=\psi(1_k),$
$\Pi_\varpi
%$\Phi_S
\circ \Psi=\phi.$ Moreover 
$\{\wtd\psi_n(1_k)\}=\{f_{1/4}(\wtd\psi_{c,n}(1_k))\}$  is a permanent projection 
lifting of $\phi(1_k)$ (see Proposition \ref{Dpproj}).   The proposition then follows from Proposition \ref{Dpproj}.
\end{proof}

\begin{cor}\label{CCLLliftos}
Let $A$ be a separable  non-elementary simple \CA\, with nonempty compact $T(A)$
satisfying  condition (C) such that $\partial_e(T(A))=\cup_{n=1}^\infty X_n$  satisfies  (1) and (2) 
in \ref{Dc1}. 
 Suppose that $A$ has tracial approximate oscillation zero and 
%which has $T$-tracial approximate oscillation zero and $\tau\in T(A).$
suppose that $\phi: M_k\to (l^\infty(A)\cap A')/I_{_{\cup_{i=1}^n X_i, \varpi}}$ (for some $n\in \N$)
%{_{\Tw,\varpi}}$ 
is a \hm.
Then $\phi$ has  property 
(Os).
%
%%%%%%%%%%%%%
\end{cor}

\begin{proof}
Since $A$ has tracial approximate oscillation zero, by  Theorem \ref{TLlifting}, 
%(1) of Lemma \ref{LLlifting-2N}, 
there is a \hm\, $\phi: M_k\to \Cqcc$ such that $\Phi_S\circ \psi=\phi.$ 
Then the corollary follows from Proposition \ref{LLliftos}.
\end{proof}

\section{Sums of order zero maps}

We fix a free ultrafilter $\varpi\in \bt(\N)\setminus \N.$

%%%%%%%%%%%
\iffalse
\begin{lem}\label{L2pproj}
Let $A$ be a separable simple \CA\, with  ${\rm Ped}(A)=A$ and with $T(A)\not=\emptyset.$
Let $p=\Pi_\varpi(\{e_n\})$ and $q=\Pi_\varpi(\{f_n\})$ be projections 
and $\{e_n\}$ and $\{f_n\}$ be permanent projection liftings. 
Let $c_n=(1-f_n)^{1/2}$ and $c=\Pi_\varpi(\{c_m\}).$  Suppose that
$pc=cp.$ 
Then,  for any $\eta\in (0,1/2),$ 
 $\{f_\eta(c_n^{1/2}e_nc_n^{1/2})\}$ is a permanent projection lifting of the projection $cp.$
In particular, for any $\ep>0,$ there exists $\dt>0$  and ${\cal P}\in \varpi$ such that,  for all  $n\in {\cal P},$
\beq
\sup\{d_\tau(b_n)-\tau(f_\dt(b_n)):\tau\in T(A)\}<\ep\tforal n\in {\cal P},
\eneq
where $b_n=f_\eta(c_n^{1/2}e_nc_n^{1/2}),$ $n\in \N.$
\end{lem}

\begin{proof}
We first  note that   $c=1-q$ and 
$pc=(1-q)p$ is a projection. By Proposition \ref{Dpproj}, $\{f_\eta(c_n^{1/2} e_n e_n^{1/2})\}$ is 
a permanent projection lifting.  Then, by Proposition \ref{Dpproj},  again, 
for any $\dt\in (0,1/2),$ 
\beq
\lim_{n\to\varpi}\sup\{d_\tau(b_n)-\tau(f_\dt(b_n)): \tau\in T(A)\}= 0. 
\eneq
In other words,  $\ep\in (0,1/2),$ there exists  ${\cal P}\in \varpi$ such that
\beq
\sup\{d_\tau(b_n)-\tau(f_\dt(b_n)):\tau\in T(A)\}<\ep\tforal n\in {\cal P}.
\eneq

\end{proof}
\fi
%%%%%%%%%%%%%%%%%%%%%

\begin{lem}\label{Lortho}
Let $A$ be a simple \CA\, with compact $T(A)\not=\emptyset.$
Suppose that $a,b\in A_+^{\bf 1}\setminus \{0\}$ such 
that 
\beq
\tau(a)>1-\ep
%\inf\{\tau(b^2):\tau\in T(A)\} \ep
\eneq
for some $\tau\in T(A)$ and $\ep\in (0,1/2).$
Then 
\beq
\tau(bab)>\tau(b^2)-\sqrt{\ep}.
\eneq
\end{lem}

\begin{proof}
We have 
\beq
\tau(b^2-bab)^2&=&\tau(b(1-a)b)^2\le \tau(b^2)\tau((1-a)b^2(1-a))\\
&<&\tau(b^2)\tau(1-a)=\tau(b^2)(1-\tau(a))<\ep.
%&<&\tau(b^2)\inf\{\tau(b^2):\tau\in T(A)\}\ep\le \tau(b^2)^2\ep.
\eneq
It follows that
\beq
\tau(bab)=\tau(b^2)-\tau(b^2-bab)>\tau(b^2)-\sqrt{\ep}
%=\tau(b^2)-\sqrt{\ep}
\eneq
as desired.
\end{proof}

\begin{lem}\label{LLnmaps}
Let $A$ be a separable amenable algebraically simple \CA\, with 
%continuous scale and  
nonempty compact $T(A)$
%\not=\emptyset,$ 
%Suppose that $A$ has T-tracial approximate oscillation zero 
and let $k\in \N.$
Let $\{S_n\}\subset 
%=\{\tau_n\}\subset 
\partial_e(T(A))$ be an increasing sequence of compact subsets.
%countable set.
%
%which satisfies condition (C1),
%Suppose that $A$ has T-tracial approximate oscillation zero 
%and let $k\in \N.$
%Let 
%$\partial_e(T(A))=\cup_{n=1}^\infty S_n,$ 
% satisfy  condition (1) and (2) in \ref{Dc1}.
%be an increasing sequence of compact subsets.
%countable set.
Suppose that, for each $n,$
there exists a unital \hm\, 
 $\psi_n: M_k\to (l^\infty(A)\cap A')/I_{_{S_n, \varpi}}$ which has property (Os).

Then, for any $\ep>0$ and any finite subset 
${\cal F}\subset A,$ there exists a sequence 
of order zero \cpc s $\phi_n: M_k\to A$ such that
\beq 
&&\|[\phi_n(b), \, a]\|<\sum_{j=1}^n \ep/2^{i+2}\tforal  a\in {\cal F}\tand b\in M_k^{\bf 1},\\
&&\tau(\phi_n(1_k))>1-(\ep/2^{n+5})^2\tforal \tau\in \cup_{j=1}^n S_j,\\
%(\sum_{i=1}^j \ep/2^{i+2})^2
%,\,\,1\le j\le n,\\
\eneq
and, if $\tau(\phi_n(1_k))>1-(\sigma/2)^2
%(\sum_{i=1}^j \ep/2^{i+2})^2
$ for some $\tau\in T(A)$ and $\sigma\in (0,1/2),$   then 
\beq
\tau(\phi_{n+1}(1_k))>1-(\ep/2^{n+4})^2-\sigma/2
%\sum_{i=1}^j \ep/2^{i+1}
\tforal n\in \N.
%>j,\,\, j=1,2,....
\eneq
\end{lem}

\begin{proof}
Since $C_0((0,1])\otimes M_k$ is projective, 
for any $\ep>0,$ there is a universal $\eta_n>0,$ for each $n\in \N$ satisfying the following:
for any  \cpc\, $L: M_k\to B$ (for any \CA\, $B$) such that
\beq\label{Ltwomaps-5}
\|L(a)L(b)\|<\eta_n
\eneq
for any pair of mutually orthogonal elements $a, b\in (M_k)_+^{\bf 1},$ 
there exists an order zero map $\Phi: M_k\to B$ such 
that
\beq
\|L-\Phi\|<(\ep/2^{2(n+2)})^2
\eneq
(recall the unit ball of $M_k$ is compact). 
\Wlog, we may assume that $\eta_n<\ep/2^{3(n+1)}$.

We will construct $\{\phi_n\}$ by induction.  
%Put $S_n=\cup_{j=1}^n X_j,$ $n\in \N.$

For $n=1,$  
% by
% Lemma \ref{Lonemap}, there exists a \hm\, $\psi_1: M_k\to \Cqcc$
% such that 
 %$(\tau_1)_\varpi(\psi_1(1_k))=1.$
Let 
$\Psi_1=\{\psi_1^{(m)}\}:M_k\to l^\infty(A)\cap A'$  be an order zero \cpc\, such 
 that $\Pi_\varpi\circ \Psi_1=\psi_1.$
 Since $\psi_1:M_k\to (l^\infty(A)\cap A')/I_{S_1, \varpi}$ is unital, there exists ${\cal P}_1\in \varpi$
 such that, for any $m\in {\cal P}_1,$ 
 \beq
 s(\psi_1^{(m)}(1_k))>1-(\ep/2^{2+5})^2\rforal s\in S_1.
 \eneq
 Therefore we obtain
 $m_1\in \N$ such that
 \beq\label{nmaps-1-1}
&& \|[\psi_1^{(m_1)}(b),\, a]\|<\eta_1\rforal a\in {\cal F}\andeqn b\in M_k^{\bf 1}\andeqn\\\label{nmaps-1-2}
 &&s(\phi_1^{(m_1)}(1_k))>1-(\ep/2^{1+5})^2\rforal s\in S_1.
 \eneq
Put $\phi_1=\psi_1^{(m_1)}$  and ${\cal F}_1={\cal F}\cup \phi_1(M_k^{\bf 1}).$

%By Lemma \ref{Lonemap}, there is \hm\, $\psi_2: M_k\to \Cqcc$ such 
%that $(\tau_j)_\varpi(\psi_2(1_k))=1$ ($j=1,2$) and it satisfies (Os). 
Let $\Psi_2=\{\psi_2^{(m)}\}: M_k\to l^\infty(A)\cap A'$ 
be an order zero \cpc\, such that $\Pi_\varpi\circ \Psi_2=\psi_2.$
We assume, since $\psi_2$ has property (Os), there is $\dt_2\in (0,1/4)$ and 
${\cal P}_2\in \varpi$ such that, for all $m\in {\cal P}_2,$
\beq\label{Lnmaps-4}
\sup\{d_\tau(\psi_2^{(m)}(1_k))-\tau(f_{\dt_2}(\psi_2^{(m)}(1_k))):\tau\in T(A)\}<(\ep/2^{2+8})^2.
\eneq
By (1) of Proposition \ref{Dpproj}, we may assume that, for all $m\in {\cal P}_2,$
\beq\label{Lnmaps-5}
\sup\{\tau(\psi_2^{(m)}(1_k))-\tau(f_{\dt_2}(\psi_2^{(m)}(1_k))\psi_2^{(m)}(1_k)):\tau\in T(A)\}<(\ep/2^{2+8})^2.
\eneq
\Wlog\, (recall that $\psi_2: M_k\to l^\infty(A)\cap A'/I_{_{S_2, \varpi}}$ is unital), we may choose $m_2\in {\cal P}_2$  such  that (recall that $\phi_1((M_k)^{\bf 1})$ is compact)
%that
\beq\label{nmaps-6}
&& s(\psi_2^{(m_2)}(1_k))>1-(\ep/2^{2+8})^2\rforal s\in S_2,\\\label{nmaps-7}
&&\|[\psi_2^{(m_2)}(b),\, a]\|<\eta_2\rforal a\in {\cal F}_1\andeqn b\in M_k^{\bf 1}, \\\label{nmaps-8}
&&\|[f_{\dt_2}(\psi_2^{(m_2)}(1_k)),\, a]\|<\eta_2 \rforal a\in {\cal F}_1
%\cup\{\psi_1^{(m_1)}(M_k^{\bf 1})\} 
\andeqn\\\label{nmaps-9}
&&\|[(1-f_{\dt_2/2}(\psi_2^{(m_2)}(1_k)))^{1/2}, \, a]\|<\eta_2  \rforal a\in {\cal F}_1.
%\cup\{\psi_1^{(m_1)}(M_k^{\bf 1})\}
%&&\sup\{d_\tau(\psi_2^{(m_2)}(1_k))-\tau(f_{\dt_2}(\psi_2^{(m_2)}(1_k))):\tau\in T(A)\}<\ep/2^{2+3}.
\eneq
Put 
\beq
a_2=f_{\dt_2}(\psi_2^{(m_2)}(1_k))\andeqn a_2^\perp=(1-f_{\dt_2/2}(\psi_2^{(m_2)}(1_k)))^{1/2}.
\eneq
Let $H_2: C_0((0,1])\otimes M_k\to A$ be a \hm\,  such that
$H_2(\imath\otimes e_{i,j})=\psi_2^{(m_2)}(e_{i,j})$ ($1\le i,j\le k$).
Then, for any $\dt\in (0,1/2),$ 
$f_\dt(\psi_2^{(m_2)}(1_k))=H_2(f_\dt(\imath)\otimes 1_k).$
It follows that $a_2$ and $a_2^\perp$ commutes with elements 
of the form $\psi_2^{(m_2)}(b)$ for $b\in M_k$ (this will be used for \eqref{Lnmaps-13++}).
Define
$L_2:M_k\to A$ by 
\beq\nonumber
L_2(b)&=&a_2^\perp\phi_1(b)a_2^\perp+
a_2\psi_2^{(m_2)}(b)a_2\rforal b\in M_k.
\eneq
Then $L_2$ is a \cpc. 

By  \eqref{Lnmaps-5}, we compute that, for all $\tau\in T(A),$
\beq\nonumber
&&\hspace{-0.32in}\sup\{\tau(\psi_2^{(m_2)}(1_k)-f_{\dt_2}(\psi_2^{(m_2)}(1_k))\psi_2^{(m_2)}(1_k)f_{\dt_2}(\psi_2^{(m_2)}(1_k))):\tau\in T(A)\}\\\nonumber
&&\le \sup\{\tau(\psi_2^{(m_2)}(1_k)-f_{\dt_2}(\psi_2^{(m_2)}(1_k))\psi_2^{(m_2)}(1_k)):
\tau\in T(A)\}+\\\nonumber
 &&\hspace{0.2in}\sup\{\tau(f_{\dt_2}(\psi_2^{(m_2)}(1_k))\psi_2^{(m_2)}(1_k))-\tau(f_{\dt_2}(\psi_2^{(m_2)}(1_k))\psi_2^{(m_2)}(1_k)f_\dt(\psi_2^{(m_2)}(1_k))):
 \tau\in T(A)\}\\\nonumber
 &&<(\ep/2^{2+8})^2+\|f_{\dt_2}(\psi_2^{(m_2)}(1_k))\|\sup\{\tau(\psi_2^{(m_2)}(1_k)-\psi_2^{(m_2)}(1_k)f_{\dt_2}(\psi_2^{(m_2)}(1_k))):
 \tau\in T(A)\}\\\label{Lnmaps-10-1}
 &&<(\ep/2^{2+8})^2+(\ep/2^{2+8})^2<(\ep/2^{2+7})^2.
\eneq
Similarly (using \eqref{Lnmaps-4}),  for all $\tau\in T(A),$
\beq\label{Lnmaps-10m}
&&\hspace{-0.75in}\sup\{\tau(f_{\dt_2}(\psi_2^{(m_2)}(1_k))-f_{\dt_2}(\psi_2^{(m_2)}(1_k))\psi_2^{(m_2)}(1_k)f_\dt(\psi_2^{(m_2)}(1_k)):\tau\in T(A)\}
<(\ep/2^{2+7})^2.
\eneq
Then, for $\tau\in S_2,$   by \eqref{Lnmaps-10-1} and \eqref{nmaps-6},  
\beq\label{Lnmaps-10}
\tau(L_2(1_k))>\tau(a_2\psi_2^{(m_2)}(1_k)a_2)
%\tau_j(1-f_{\dt_2/2}(\psi_2^{(m_2)}(1_k))\phi_1(1_k)(1-f_{\dt/2}(\psi_2^{(m_2)}(1_k)))\\
%\tau_j(f_{\dt_2}(\phi_2^{(m_2)})\psi_2^{(m_2)}(1_k)f_{\dt_2}(\phi_2^{(m_2)}(1_k))\\
>\tau(\psi_2^{(m_2)}(1_k))-(\ep/2^{2+7})^2>1-(\ep/2^{2+6})^2.
\eneq
Moreover,  if, for some $\tau\in T(A),$  $\tau(\phi_1(1_k))>1-(\sigma/2)^2,$  then, by Lemma \ref{Lortho},
\eqref{Lnmaps-10m}
%{Lnmaps},  
%and \eqref{nmaps-1-2},  
%and by \eqref{Lnmaps-10} 
and \eqref{Lnmaps-4}, 
\beq\nonumber
\tau(L_2(1_k))&=&\tau(a_2^\perp \phi_1(1_k)a_2^\perp)+
%\tau_j(1-f_{\dt_2/2}(\psi_2^{(m_2)}(1_k))^{1/2}\phi_1(1_k)(1-f_{\dt/2}(\psi_2^{(m_2)}(1_k))^{1/2})\\\nonumber
\tau(a_2\psi_2^{(m_2)}(1_k)a_2)\\\nonumber
&>&\tau((a_2^\perp)^2)-\sigma/2+\tau(f_{\dt_2}(\phi_2^{(m_2)})\psi_2^{(m_2)}(1_k)f_{\dt_2}(\phi_2^{(m_2)}(1_k))\\\nonumber
&=&\tau((1-f_{\dt_2/2}(\psi_2^{(m_2)}(1_k))+f_{\dt_2}(\phi_2^{(m_2)}(1_k))\psi_2^{(m_2)}(1_k)f_{\dt_2}(\phi_2^{(m_2)}(1_k))))-\sigma/2\\\nonumber
&\ge& \tau(1-f_{\dt_2/2}(\psi_2^{(m_2)}(1_k))+f_{\dt_2}(\phi_2^{(m_2)}(1_k)))-(\ep/2^{2+7})^2-\sigma/2\\\label{nmaps-12}
&> &1-(\ep/2^{2+8})^2-(\ep/2^{2+7})^2-\sigma/2.
%=1-\ep/2^5.
\eneq
For any $a\in {\cal F}$ and $b\in M_k^{\bf 1},$   by 
\eqref{nmaps-9},  \eqref{nmaps-8} and by \eqref{nmaps-1-1},
\beq\label{Lnmaps-13}
xL_2(b)\approx_{2\eta_2} a_2^\perp x \phi_1(b)a_2^\perp+
a_2x\psi_2^{(m_2)}(b)a_2
 \approx_{\eta_1+3\eta_2} L_2(b)x.
\eneq
Suppose that $b, c\in (M_k)^{\bf 1}_+$ such that $bc=cb=0.$
Then, since both $\phi_1$ and $\psi_2^{(m_2)}$ are order zero maps (see \eqref{nmaps-8}), 
\beq
L_2(b)L_2(c)&=&(a_2^\perp \phi_1(b)a_2^\perp+a_2\psi_2^{(m_2)}(b)a_2)(a_2^\perp \phi_1(c)a_2^\perp+a_2\psi_2^{(m_2)}(c)a_2)\\
&=&a_2^\perp \phi_1(b)a_2^\perp a_2^\perp \phi_1(c)a_2^\perp+a_2\psi_2^{(m_2)}(b)a_2^2\psi_2^{(m_2)}(c)a_2\\\label{Lnmaps-13++}
&\approx_{\eta_2}&(a_2^\perp)^3\phi_1(b)\phi_1(c)a_2^\perp=0.
\eneq
By the choice of $\eta_2,$ we obtain an order  zero \cpc\, $\phi_2: M_k\to A$ such 
that
\beq
\|L_2-\phi_2\|<(\ep/2^{2(2+2)})^2=(\ep/2^8)^2.
\eneq
By \eqref{Lnmaps-13}, 
\beq
\|[x,\, \phi_2(b)]\|<\eta_1+5\eta_2+ 2(\ep/2^8)^2<\ep/2^4 \rforal x\in {\cal F}\andeqn b\in M_k^{\bf 1}.
\eneq
Therefore, by \eqref{Lnmaps-10}, for  $s\in S_2,$ 
%$j=1,2,$ 
\beq
s(\phi_2(1_k))>1-(\ep/2^6)^2-(\ep/2^8)^2>1-(\ep/2^{2+5})^2.
\eneq
If $\tau\in T(A)$ such that $\tau(\phi_1(1_k))>1-(\sigma/2)^2,$ then, by \eqref{nmaps-12}, 
\beq
\tau(\phi_2(1_k))>1-(\ep/2^{2+8})^2-(\ep/2^{2+7})^2-\sigma/2-(\ep/2^8)^2>1-(\ep/2^{2+5})^2-\sigma/2.
\eneq
Suppose that order zero \cpc  s $\phi_1, \phi_2,...,\phi_n: M_k\to A$ have been 
constructed which meet the requirements of the lemma.

Let ${\cal F}_n={\cal F}\cup \{\phi_n(b): b\in M_k^{\bf 1}\}.$
% By Lemma \ref{Lonemap}, there is \hm\, $\psi_{n+1}: M_k\to \Cqcc$ such 
%that $(\tau_j)_\varpi(\psi_{n+1}(1_k))=1$ ($1\le j\le n+1$) and it satisfies (Os). 
Let $\Psi_{n+1}=\{\psi_{n+1}^{(m)}\}: M_k\to l^\infty(A)\cap A'$ 
be an order zero \cpc\, such that $\Pi_\varpi\circ \Psi_{n+1}=\psi_{n+1}.$
We assume, since $\psi_{n+1}$ has property (Os), there is $\dt_{n+1}\in (0,1/4)$ and 
${\cal P}_{n+1}\in \varpi$ such that, for all $m\in {\cal P}_{n+1},$
\beq\label{Lnmaps-n+1}
\sup\{d_\tau(\psi_{n+1}^{(m)}(1_k))-\tau(f_{\dt_{n+1}}(\psi_{n+1}^{(m)}(1_k))):\tau\in T(A)\}<({\ep\over{2^{n+1+8}}})^2.
\eneq
By (1) of Proposition \ref{Dpproj}, we may assume that
\beq\label{Lnmaps-n-5}
\sup\{\tau(\psi_{n+1}^{(m)}(1_k))-\tau(f_{\dt_{n+1}}(\psi_2^{(m)}(1_k))\psi_2^{(m)}(1_k)):\tau\in T(A)\}<({\ep\over{2^{n+1+8}}})^2.
\eneq
\Wlog\, (recall that $\psi_{n+1}: M_k\to l^\infty(A)\cap A'/I_{_{S_{n_1}, \varpi}}$ is unital), 
we may choose $m_{n+1}\in {\cal P}_{n+1}$  such
that
\beq\label{nmaps-n-6}
&& s(\psi_{n+1}^{(m_{n+1})}(1_k))>1-(\ep/2^{n+1+8})^2\rforal s\in S_{n+1},\\\label{nmaps-n-7}
&&\|[\psi_{n+1}^{(m_{n+1})}(b),\, a]\|<\eta_{n+1}\rforal a\in {\cal F}_n,\\\label{nmaps-n-8}
&&\|[f_{\dt_{n+1}}(\psi_{n+1}^{(m_{n+1})}(1_k)),\, a]\|<\eta_{n+1} \rforal a\in {\cal F}_n
%\cup\{\psi_1^{(m_1)}(M_k^{\bf 1})\} 
\andeqn\\\label{nmaps-n-9}
&&\|[(1-f_{\dt_{n+1}/2}(\psi_{n+1}^{(m_{n+1})}(1_k)))^{1/2}, \, a]\|<\eta_{n+1}  \rforal a\in {\cal F}_n.
%\cup\{\psi_1^{(m_1)}(M_k^{\bf 1})\}
%&&\sup\{d_\tau(\psi_2^{(m_2)}(1_k))-\tau(f_{\dt_2}(\psi_2^{(m_2)}(1_k))):\tau\in T(A)\}<\ep/2^{2+3}.
\eneq
Put 
\beq
a_{n+1}=f_{\dt_{n+1}}(\psi_{n+1}^{(m_{n+1})}(1_k))\andeqn a_{n+1}^\perp=(1-f_{\dt_{n+1}/2}(\psi_{n+1}^{(m_{n+1})}(1_k)))^{1/2}.
\eneq
Let $H_{n+1}: C_0((0,1])\otimes M_k\to A$ be a \hm\,  such that
$H_{n+1}(\imath\otimes e_{i,j})=\psi_2^{(m_2)}(e_{i,j})$ ($1\le i,j\le k$).
Then, for any $\dt\in (0,1/2),$ 
$f_\dt(\psi_2^{(m_2)}(1_k))=H_{n+1}(f_\dt(\imath)\otimes 1_k).$
It follows that $a_{n+1}$ and $a_{n+1}^\perp$ commutes with elements 
of the form $\psi_{n+1}^{(m_{n+1})}(b)$ for $b\in M_k.$
Define
$L_{n+1}:M_k\to A$ by 
\beq\nonumber
L_{n+1}(b)&=&a_{n+1}^\perp \phi_n(b)a_{n+1}^\perp+
a_{n+1}\psi_{n+1}^{(m_{n+1})}(b)a_{n+1}\rforal b\in M_k.
\eneq
Then $L_{n+1}$ is a \cpc. 
By  \eqref{Lnmaps-n-5}, we compute that, for all $\tau\in T(A),$
\beq\label{nmaps-nn-10}
&&\hspace{-0.3in}\sup\{\tau(\psi_{n+1}^{(m_{n+1})}(1_k))-a_{n+1}
%f_{\dt_{n+1}}
\phi_{n+1}^{(m_{n+1})}(1_k) a_{n+1}
%\psi_{n+1}^{(m_{n+1})}(1_k)
%f_{\dt_{n+1}}(\phi_{n+1}^{(m_{n+1})}(1_k))
:\tau\in T(A)\}<({\ep\over{2^{n+1+7}}})^2\andeqn\\\label{nmaps-nn-11}
&&\hspace{-0.3in}\sup\{\tau(f_{\dt_{n+1}}(\psi_{n+1}^{(m_{n+1})}(1_k)))-a_{n+1}\psi_{n+1}^{(m_{n+1})}(1_k)a_{n+1}
%f_{\dt_{n+1}}(\phi_{n+1}^{(m_{n+1})}(1_k))
:\tau\in T(A)\}<({\ep\over{2^{n+1+7}}})^2.
\eneq
It follows from  \eqref{nmaps-nn-11} that, for all $\tau\in T(A),$
\beq\label{Lnmaps-n-13}
\tau((a_{n+1}^\perp)^2+a_{n+1}\psi_{n+1}^{(m_{n+1})}(1_k)a_{n+1})>1
%-({\ep\over{2^{n+1+8}}})^2
-({\ep\over{2^{n+1+7}}})^2.
\eneq
 For   $\tau\in S_{n+1},$ 
 %$1\le j\le n+1$,  
 by \eqref{nmaps-nn-10} and \eqref{nmaps-n-6},
 %(see also the inequality  two lines above),
\beq\label{Lnmaps-n-14}
&&\hspace{-0.6in}s(L_{n+1}(1_k))>s(a_{n+1}\psi_{n+1}^{(m_{n+1})}(1_k)a_{n+1})
%\tau_j(1-f_{\dt_2/2}(\psi_2^{(m_2)}(1_k))\phi_1(1_k)(1-f_{\dt/2}(\psi_2^{(m_2)}(1_k)))\\
%\tau_j(f_{\dt_2}(\phi_2^{(m_2)})\psi_2^{(m_2)}(1_k)f_{\dt_2}(\phi_2^{(m_2)}(1_k))\\
>s(\psi_2^{(m_2)}(1_k))-({\ep\over{2^{n+7}}})^2>1-({\ep\over{2^{n+6}}})^2.
\eneq
Moreover,  if $\tau(\phi_n(1_k))>1-(\sigma/2)^2$ for some 
$\tau\in T(A),$   then, by Lemma \ref{Lortho}
%{Lnmaps},
% and  inductive assumption,
%\eqref{nmaps-1-2},  
and by \eqref{Lnmaps-n-13},
%\eqref{Lnmaps-10} 
%and \eqref{Lnmaps-4}, 
\beq\nonumber
\tau(L_{n+1}(1_k))&=&\tau(a_{n+1}^\perp \phi_n(1_k)a_{n+1}^\perp)+
%\tau_j(1-f_{\dt_2/2}(\psi_2^{(m_2)}(1_k))^{1/2}\phi_1(1_k)(1-f_{\dt/2}(\psi_2^{(m_2)}(1_k))^{1/2})\\\nonumber
\tau(a_{n+1}\psi_{n+1}^{(m_{n+1})}(1_k)a_{n+1})\\\nonumber
&>&\tau((a_{n+1}^\perp)^2)-\sigma/2+\tau(a_{n+1}\psi_2^{(m_{n+1})}(1_k)a_{n+1})\\\nonumber
&=&\tau((a_{n+1}^\perp)^2+a_{n+1}\psi_{n+1}^{(m_{n+1})}(1_k)a_{n+1})-\sigma/2\\\label{Lnmaps-nn}
%&\ge& \tau(1-f_{\dt_2/2}(\psi_2^{(m_2)}(1_k))+f_{\dt_2}(\phi_2^{(m_2)}(1_k)))-\ep/2^{2+4}-\ep/8\\\label{nmaps-n-12}
&> &1-(\ep/2^{n+7})^2-\sigma/2.
%>1-(\ep/2^{n+6})^2-\sigma/2.
%=1-19\ep/128.
\eneq
For any $a\in {\cal F}$ and $b\in M_k^{\bf 1},$    by 
\eqref{nmaps-n-9},  \eqref{nmaps-n-8} and by  inductive assumption,
%\eqref{nmaps-1-1},
\beq
xL_{n+1}(b)&\approx_{2\eta_{n+1}}& a_{n+1}^\perp x \phi_n(b)a_{n+1}^\perp+
a_{n+1}x\psi_{n+1}^{(m_{n+1})}(b)a_{n+1}\\\label{nmapsn-n100}
 &&\approx_{(\sum_{i=1}^n\ep/2^{i+2})+3\eta_{n+1}} L_{n+1}(b)x.
\eneq
Suppose that $b, c\in (M_k)^{\bf 1}_+$ such that $bc=cb=0.$
Then, since both $\phi_n$ and $\psi_{n+1}^{(m_2)}$ are order zero maps (see \eqref{nmaps-n-8}),
\beq\nonumber
&&\hspace{-0.5in}L_{n+1}(b)L_{n+1}(c)=\\\nonumber
&&(a_{n+1}^\perp \phi_n(b)a_{n+1}^\perp+a_{n+1}\psi_{n+1}^{(m_{n+1})}(b)a_{n+1})(a_{n+1}^\perp \phi_n(c)a_{n+1}^\perp+a_{n+1}\psi_{n+1}^{(m_{n+1})}(c)a_{n+1})\\\nonumber
&&=a_{n+1}^\perp \phi_n(b)a_{n+1}^\perp a_{n+1}^\perp \phi_n(c)a_{n+1}^\perp+a_{n+1}\psi_{n+1}^{(m_{n+1})}(b)a_{n+1}^2\psi_{n+1}^{(m_{n+1})}(c)a_{n+1}\\
&&\approx_{\eta_{n+1}}(a_{n+1}^\perp)^3\phi_n(b)\phi_n(c)a_{n+1}^\perp=0.
\eneq
By the choice of $\eta_{n+1},$ we obtain an order zero \cpc\, $\phi_{n+1}: M_k\to A$ such that
\beq
\|L_{n+1}-\phi_{n+1}\|<(\ep/2^{2(n+1+2)})^2.
\eneq
Then, by \eqref{nmapsn-n100}, for all $x\in {\cal F}$ and $b\in {\cal F},$ 
\beq
\|[x,\, \phi_{n+1}(b)]\|<\sum_{i=1}^n\ep/2^{i+2}+5\eta_{n+1}+2(\ep/2^{2(n+1+2)})^2<\sum_{i=1}^{n+1}\ep/2^{i+2}.
\eneq
Also, by \eqref{Lnmaps-n-14},  for $s\in S_{n+1},$ 
%$1\le j\le n+1,$ 
\beq
s(\phi_{n+1}(1_k))>1-({\ep\over{2^{n+6}}})^2-(\ep/2^{2(n+1+2)})^2>1-(\ep/2^{n+5})^2.
\eneq
If $\tau(\phi_n(1_k))>1-(\sigma/2)^2,$ then, by \eqref{Lnmaps-nn},
\beq
\tau(\phi_{n+1}(1_k))>1-(\ep/2^{n+7})^2-\sigma/2-(\ep/2^{2(n+1+2)})^2>1-(\ep/2^{n+1+4})^2-\sigma/2.
\eneq
This completes the induction  and the lemma follows.
\end{proof}

\section{The main results}
Let us first recall a result of W. Zhang:

\begin{lem}[W. Zhang, Lemma 6.5 of \cite{Zw}]\label{LwZ}
Let $d\ge 0$ and $k\ge 2$  be integers and let $A$ be a 
non-elementary separable simple 
unital amenable \CA\, with $T(A)\not=\emptyset.$  Then, for any compact subset 
$Y\subset \partial_e(T(A))$ with ${\rm dim}(Y)\le d,$ any finite subset 
${\cal F}\subset A$ and $\ep>0,$ there exists 
an order zero \cpc\, $\phi: M_k\to A$ such that
\beq
&&\|[\phi(b),\, y]\|<\ep\|b\|\tforal b\in  M_k,\,\, y\in {\cal F}\tand\\
&&\tau(\phi(1_k))>1-\ep\tforal \tau\in Y.
\eneq
\end{lem}

\begin{proof}
Since the  proof of this is not provided in \cite{Zw}, 
let us explain briefly.   As pointed out  in \cite{Zw}, the proof only 
depends on Lemma 6.4 of \cite{Zw} (whose proof is also omitted in \cite{Zw})  which is  referred to  Lemma 3.5 of \cite{TWW-2}. 
Note that 
 the finite open cover argument remains 
exactly the same as  that of the proof of Lemma 3.5 of \cite{TWW-2}
(replacing $\partial_e(T(A))$ by compact subset $Y,$ and  
open subsets $V_j^{(i)}$ by relatively open subsets of $Y$).
One point should be made  clearer is the extension of functions $f_j^{(i)}$ 
(part of the partition of unity) in the proof of Lemma 3.5 of \cite{TWW-2}. Indeed, one only defines 
$f_j^{(i)}$ on the compact subset $Y.$
Applying Corollary 11.15 of \cite{kG},
one may extend  $f_j^{(i)}$ to real affine continuous functions 
on $T(A).$   So the rest of the proof of Lemma 3.5 of \cite{TWW-2} 
remains the same (but restricting on $Y$ only).  
In other words, Lemma 6.4 of \cite{Zw}
follows. Consequently, as pointed out in \cite{Zw},  Lemma 6.5 of \cite{Zw} holds.
\end{proof}

%%%%%%%%%%%%%%
\iffalse
%
The following is  a result of Y. Sato  which is proved in \cite{S-2}.

\begin{lem}[Y. Sato, Proposition 5.1 of \cite{S-2}]\label{LSato}
Let  $A$ be a separable unital  amenable  simple \CA\, 
%with continuous scale 
and $T(A)\not=\emptyset$ such that 
$\partial_e(T(A))$  is compact and $S\subset \partial_e(T(A))$ is a compact subset with covering dimension 
$d\in \Z_+.$
Then, for any $k\in \N,$ 
%and any $\ep>0,$  
there exist order zero \cpc s $\phi_l: M_k\to l^\infty(A)\cap A'/I_{_{T(A),\N}},$  $0\le l\le d,$ 
such that
\beq\label{LSato-1}
\max_{\tau\in S}\{\|1-\sum_{l=0}^d\phi_l(1_k)\|_\tau\}=0\tand [\phi_l(a), \phi_m(b)]=0\tforal l\not= m,\,\, a, b\in M_k.
\eneq

\end{lem}

The proof of the above lemma is contained in  the induction process of the proof of 
Proposition 5.1 of \cite{S-2}.
\fi
%%%%%%%%%%%%%

\begin{rem}\label{RwZ}
%Corollary \ref{Csato} 
Lemma \ref{LwZ}
may also be 
%similarly 
extracted from \cite{S-2} and
% \cite{TWW-2}. 
%One may also do this 
%from 
\cite{KR}.
In
% \cite{S-2}, 
\cite{TWW-2}  as well as in \cite{KR},   the condition that $T(A)$ is a Bauer simplex is required.
It is worth noting that Ken Goodearl's Corollary 11.15  of \cite{kG} plays a key role in the statements 
in \cite{Zw} which allows Wei Zhang to work on a compact subset of $\partial_e(T(A))$
instead of $\partial_e(T(A)).$   Next we will explain that W. Zhang's lemma also holds without 
assuming $A$ is unital. 
% Once Lemma 6.4 of \cite{Zw} is established (and $d$ is given), under the assumption 
%that $A$ has T-tracial approximate oscillation zero,  $\omega(\psi^{(i)}(1_k))$ can be made arbitrarily small in 
%Lemma 6.4 of \cite{Zw}.
%In that case,  a  variation of results  in  section 4  can be applied to add these $d+1$ maps to obtain Lemma \ref{LwZ}
%(or rather Corollary \ref{Csato}). 
\end{rem}
%%%%%%%%%%%%

Let us state the following lemma which is known in the unital case.

\begin{lem}[Lemma 4.2 of \cite{S-2}, Lemma 3.3 and 3.4 of \cite{TWW-2}]\label{LLsato}
Let $A$ be a separable simple amenable \CA\, with nonempty compact $T(A)$ and 
$F\subset \partial_e(T(A))$ be a compact subset.  
For any mutually orthogonal positive functions  $f_1, f_2,...,f_m$ in $C(F),$
there exist $\{a_{i,n}\}_{n\in\N}\in A',$ $1\le i\le m,$ such that
\beq
\lim_{n\to\infty}\max_{\tau\in F}|\tau(a_{i,n})-f_i(\tau)|=0\andeqn
\lim_{n\to\infty}\|a_{i,n}a_{j,n}\|=0,\,\, i\not=j, \,\,1\le j\le m.
\eneq
\end{lem}
\begin{proof}
First one observes that Lemma 3.3 and Lemma 3.4   in \cite{TWW-2} hold for  the case that $A$ is not necessarily unital 
but $T(A)$ is compact (see also the footnote on the page with Lemma \ref{LIunit}, also Lemma 4.1 of \cite{S-2}).
Then, 
by Theorem 11.14
%Corollary 11.15  
of \cite{kG},  there are    $g_1, g_2,...,g_m\in \Aff(T(A))$  with $0\le g_i\le 1$ such that $(g_i)|_{F}=f_i,$
$1\le i\le m.$  By (non-unital version of) Lemma 3.3 of \cite{TWW} mentioned,  one obtains $\{e_{n}^{(1)}\}\in A',$ 
%$1\le i\le m,$ 
such that 
\beq
\lim_{n\to\infty}\max_{\tau\in T(A)}|\tau(e_n^{(1)})-g_1(\tau)|=0.
\eneq
Then, by  Proposition 3.1 of \cite{CETW} and (non-unital version of) Lemma 3.3 of \cite{TWW}, one obtains $\{e_{n}^{(2)}\}\in A',$ 
%$1\le i\le m,$ 
such that 
\beq
\lim_{n\to\infty}\max_{\tau\in T(A)}|\tau(e_n^{(2)})-g_2(\tau)|=0\andeqn \lim_{n\to \infty}\max_{\tau\in F}|\tau(e^{(1)}_ne^{(2)}_n)|=0.
\eneq
Repeating this, 
one obtains $\{e_{n}^{(i)}\}\in A',$ 
$1\le i\le m,$ 
such that 
\beq
\lim_{n\to\infty}\max_{\tau\in T(A)}|\tau(e_n^{(i)})-g_i(\tau)|=0\andeqn \lim_{n\to \infty}\max_{\tau\in F}|\tau(e^{(i)}_ne^{(j)}_n)|=0,
\,\,i\not=j.
\eneq
The lemma then follows by applying (non-unital version of) Lemma 3.4 of  in \cite{TWW} mentioned above.
\end{proof}

%In fact Lemma \ref{LwZ} holds without assuming $A$ is unital. 
We would like to 
state the non-unital version of  W. Zhang's Lemma \ref{LwZ} as follows. 
%which is actually proved by Sato in \cite{S-2}
%(with some modification). 

\begin{lem}[Proposition 5.1 of \cite{S-2} and Lemma 6.5 of  \cite{Zw}]\label{Csato}
Let  $A$ be a non-elementary separable amenable  {{algebraically}} simple \CA\, 
%with continuous scale and 
with  compact $T(A)\not=\emptyset$  and let 
%$\partial_e(T(A))$ 
% is compact and 
$S\subset \partial_e(T(A))$ be  a compact subset with covering dimension 
$d\in \Z_+.$  
%Suppose that $A$ has T-tracial approximate oscillation zero. 
Then, for any $k\in \N,$ 
%and any $\ep>0,$  
there exists a unital  \hm\, $\phi: M_k\to (l^\infty(A)\cap A')/I_{_{S, \varpi}}.$ 
%which has property (Os).  
\end{lem}

\begin{proof}
This is actually proved in Proposition 5.1 of \cite{S-2}.
We will use  the second part of Proposition 3.1 of \cite{CETW}  and Lemma \ref{LLsato}
to replace Lemma 4.2 of \cite{S-2}. 
It suffices to show that the conclusion of Proposition 5.1 of \cite{S-2} holds
but replacing $A_{t, \infty}$ (the notation used in Proposition 5.1 of \cite{S-2})
by $(l^\infty(A)\cap A')/_{I_{S, \N}}$ (choosing $\phi:=\sum_{l=1}^d\phi_l$). 
In other words, it suffices to show that there exist 
$\phi_{l,n}: M_k\to A$ which satisfies  (i)' (replacing $1_ {A_\infty}$ by
$1_{l^\infty(A)/I_{_{S, \N}}}$),  (ii)',  (iii)', (iv)' and (v)'  (stated in the proof of \cite{S-2}) hold
(for $B=S$). 

Now as in the proof of Proposition 5.1, let $B\subset \partial_e(T(A))$ be a compact
subset with covering dimension $c\le d.$  As 
in  the proof of Proposition 5.1 of \cite{S-2}, using induction on $d,$ it suffices to prove  that 
there exists a \cpc\, $\phi_{l,n}: M_k \to A,$ $l=0,1,2,...,c\, (\le d),$  $n\in \N$
such that 
(i)', (ii)',  (iii)', (iv)' and (v)'  (stated in the proof of \cite{S-2}) hold.
%We also note that Lemma 3.3 works for separable simple \CA s with  Lemma 3.4 of \cite{TWW} works for 
In fact the proof of Proposition 5.1 of \cite{S-2} does  just that
(with one obvious modification: (a-vii) holds when we replace $\partial_e(T(A))$ by $B$ as  we apply the second part of Proposition 3.1 of \cite{CETW} instead of part (i) of Lemma 4.2 
of \cite{S-2}
and Lemma \ref{LLsato} instead of part (ii) of Lemma 4.2 of \cite{S-2}).
\end{proof}

\begin{lem}\label{Lcover2}
Let $A$ be a separable  non-elementary
%unital  
amenable  {{algebraically}} simple \CA\, with 
% with continuous scale, 
$T(A)\not=\emptyset$  which is compact
and with T-tracial approximate oscillation zero such that 
$\partial_e(T(A))$  satisfying condition (C). 
%=\cup_{n=1}^\infty X_n$ which satisfies (1) and (2) in \ref{Dc1}.
% has strong countable dimension. 
%
%Suppose also that $A$ has $T$-tracial approximate oscillation zero.
Then, 
for any integer $k\in\N,$ any $\ep>0$ and 
any finite subset ${\cal F}\subset A,$  there 
is an order zero \cpc\, $\phi: M_k\to A$ such that
\beq
\tau(\phi(1_k))>1-\ep\tforal \tau\in T(A)\tand \|[f, \, \phi(b)]\|<\ep
\eneq
for all $f\in {\cal F}$ and $b\in M_k^{\bf 1}.$
\end{lem}

\begin{proof}
Since  $T(A)$ satisfies condition (C), 
we may write $\partial_e(T(A))=\cup_{n=1}^\infty X_n$
which satisfies conditions (1) and (2) in \ref{Dc1}.
%$\partial_e(T(A))$ has countable dimension, 
Write $S_n=\cup_{i=1}^n X_i.$
%there  is an increasing sequence of 
%compact subsets $S_n$ of $\partial_e(T(A))$
%such that $\cup_{n=1}^\infty S_n=\partial_e(T(A))$ and 
Note that each $S_n$ has finite covering dimension and compact.
By Lemma \ref{Csato}, there exists, for each $n\in \N,$  a unital \hm\, 
${\bar \psi}: M_k\to (l^\infty(A)\cap A')/I_{_{S_n, \varpi}}.$ 
Since $A$ is assumed to have T-tracial approximate oscillation zero, by Lemma \ref{TLlifting},
there exists, for each $n\in\N,$  a \hm\, $\psi: M_k\to \Cqcc$  such that $\Phi_{S_n}\circ \psi={\bar \psi},$
where $\Phi_{S_n}: l^\infty(A)/_{_{T(A),\varpi}}\to l^\infty(A)/I_{_{S_n, \varpi}}$ is the quotient map,
and $\psi$ has property (Os). 

Fix $k\in \N,$ $\ep\in (0,1/2)$ and a finite subset ${\cal F}\subset A.$ 

Applying
% Lemma  \ref{Csato}  (with $S_n$ instead of $S$ for each $n$) and 
Lemma \ref{LLnmaps}, we obtain 
a sequence of  order zero \cpc s  $\phi_n: M_k\to A$ 
which satisfies the conclusion of Lemma \ref{LLnmaps} (with respective to $\{S_n\},$ ${\cal F}$ and $\ep$). 

Put $G_n=\{\tau\in T(A): \tau(\phi_n(1_k))>1-(\ep/4)^2\},$ $n\in \N.$

Let $\tau\in T(A).$ Then, by The Choquet  Theorem, there is a probability  Borel measure 
$\mu_\tau$ on $\partial_e(T(A))$ 
such that
\beq
\tau(f)=\int_{\partial_e(T(A))} f d\mu_\tau\rforal f\in \Aff(T(A)).
\eneq
%Put $F_1=S_1$ and $F_{n+1}=S_{n+1}\setminus S_n,$ $n=1,2,....$ 
Let $\mu_{\tau, n}=\mu_\tau|_{X_n}.$ 
Recall that $X_i\cap X_j=\emptyset,$ if $i\not=j$ and $i,j\ge 2.$
Then 
\beq
%\tau(f)=
\sum_{n=2}^\infty \int_{X_n} d\mu_{\tau, n}
\le \|\mu_\tau\|.
%\rforal f\in \Aff(T(A)).
\eneq
Since $\|\mu_\tau\|=1,$ there is $n_1\in \N$ such that
\beq\label{Lcover-n3}
\sum_{m>n_1}^{\infty} \|\mu_{\tau, m}\|<(\ep/8)^2\andeqn  \mu_\tau(S_{n_1})>1-(\ep/8)^2.
\eneq
We may assume that $n_1>2.$
Then (as $\{\phi_n\}$ satisfies the conclusion of Lemma \ref{LLnmaps}), if $n\ge n_1,$ 
\beq\label{Lcover-4}
\tau(\phi_n(1_k))&\ge &\int_{S_{n_1}} \widehat{\phi_n(1_k)}(s)d\mu_\tau\\
 %+\sum_{m>n_1}\int_{F_m}\widehat{\phi_n(1_k)}(s) d\mu_{\tau, m}\\
&\ge & (1-(\ep/2^{n+5})^2)\mu_\tau(S_{n_1})  >1-(\ep/4)^2.
\eneq
In other words, $\tau\in G_n$ (for $n\ge n_1$). 
It follows that $\cup_{n=1}^\infty G_n\supset T(A).$ 
Since $T(A)$ is compact,  there exists $n_0\in \N$ such that
\beq
T(A)\subset \cup_{n=1}^{n_0}G_n.
\eneq
Let $\tau\in T(A).$ Suppose that $\tau\in G_j$ for some $j\le n_0.$
Then
\beq
\tau(\phi_j(1_k))>1-(\ep/4)^2.
\eneq
It follows from the conclusion of Lemma \ref{LLnmaps} that
\beq
\tau(\phi_{n_0}(1_k))>1-\sum_{i=j+1}^{n_0}\ep/2^{i+1+4}-\ep/4>1-\ep.
\eneq
Choose $\phi=\phi_{n_0}.$ Then
\beq
\hspace{-0.2in}\|[x,\, \phi(b)]\|<\ep\rforal x\in {\cal F}\andeqn b\in M_k^{\bf 1},  \andeqn
\tau(\phi(1_k))>1-\ep\rforal \tau\in T(A).
\eneq
The lemma follows.

\end{proof}

\begin{thm}\label{TM-1}
Let $A$ be a separable non-elementary amenable  simple \CA\, 
whose extremal boundaries of $\wtd{T}(A)\setminus \{0\}$ satisfies condition (C)
(see Definition \ref{DconC}).
%with ${\wtd T}(A)\setminus \{0\}\not=\emptyset$
%such that $\partial_e(T(A))$  satisfies condition (C1).
%has strong countable dimension. 
%
%which has a basis $S$  satisfying condition (C).
%such that  $\partial_e(S)$ has countably many points.
 %such that 
%$\overline{\partial_e(S)}$ has finitely many limit points. 
Then  following are equivalent.

(1) $A$ has strict comparison and T-tracial approximate oscillation zero,

(2) $A$ has strict comparison and stable rank one,

(3) $A\cong A\otimes {\cal Z},$

(4) $A$ has finite nuclear dimension.

\end{thm}

\begin{proof}
We only need to prove (1) $\Rightarrow$ (3). 
The equivalence of (1) and (2) follows from Theorem 1.1 of \cite{FLosc}.
The equivalence of (3) and (4)   has been proved  (see  \cite{CE}, \cite{CETWW}, 
\cite{Winter-Z-stable-02}, \cite{T-0-Z}, \cite{MS2}) in general (without the assumption that $T(A)$ 
satisfies condition (C)). 
That 
(3)  implies that $A$ has strict comparison 
 is   proved in \cite{Rrzstable} in general. In fact, it is also proved that, in unital case, $A$ has stable rank one.
For non-unital case, that $A$ is ${\cal Z}$-stable implies that  $A$ has stable rank one 
is proved in \cite{FLL} (Corollary 6.8 of \cite{FLL}).  In other words, (3) $\Rightarrow$ (2) holds 
(without the assumption that $T(A)$ 
satisfies condition(C)).

To show (1) $\Rightarrow$ (3), let us assume that $A$ has strict comparison and T-tracial approximate oscillation zero.
%Note that 
%Since $\partial_e(T(A))$ has only countably many points,
%It follows from 
%by Theorem 5.9 of \cite{FLosc}, 
%$A$ has T-tracial  approximate oscillation zero.
 It follows from  Theorem  7.12  of \cite{FLosc} that  the canonical map 
$\Gamma: \Cu(A)\to {\rm LAff}_+({\wtd{T}}(A))$ is surjective.
Then, choose $a\in {\rm Ped}(A)_+^{\bf 1}\setminus \{0\}$ such that 
$d_\tau(a)$ is continuous on ${\wtd{T}}(A).$ Define $A_1=\Her(a).$
Then $A_1$ has continuous scale  and $T(A_1)$ is compact (see Proposition 5.4 of \cite{eglnp}). 
Since $A_1\otimes {\cal K}\cong A\otimes {\cal K}$ (\cite{Br}), 
by Cor. 3.1 of \cite{TWst},
 it suffices to show that $A_1$ is ${\cal Z}$-stable.
 It follows part (2) of Proposition \ref{Pextrb} that $T(A_1)$  satisfies condition (C).

Fix an integer $k\ge 2.$
Let $\{{\cal F}_n\}$ be an increasing sequence of finite subsets of $A_1$ such that 
$\cup_{n=1}^\infty {\cal F}_n$ is dense in $A_1.$ By Lemma \ref{Lcover2},
%ref{Lcover}, 
for each $n\in \N,$ there exists 
an order zero \cpc\, $\phi_n: M_k\to A_1$ such that
\beq\label{TM-1-1}
&&\|[a,\, \phi_n(b)]\|<1/n\rforal a\in {\cal F}_n\andeqn b\in M_k^{\bf 1}\andeqn\\\label{TM-1-2}
&&\sup\{\tau(\phi_n(1_k)): \tau\in T(A_1)\}>1-1/n,\,\,\,n\in \N.
\eneq
Define $\Phi: M_k\to l^\infty(A_1)$ by $\Phi(b)=\{\phi_n(b)\}.$ Then, by \eqref{TM-1-1},
$\Phi$ maps $M_k$ into $l^\infty(A_1)\cap(A_1)'.$  By 
\eqref{TM-1-2},
\beq
\lim_{n\to\infty}\sup\{1-\tau(\phi_n(1_k)): \tau\in T(A_1)\}=0.
\eneq
It follows that $\Pi_{\varpi}\circ \Phi$ is a unital order zero \cpc. Therefore it is a unital \hm.
Hence   (2) follows from a result of Matui-Sato (see, explicitly, 
  Corollary 5.11 and Proposition 5.12  of \cite{KR}, for example) in the unital case.
%a version of Matui-Sato's theorem. 
For non-unital case,   we  take a detour and use the result in \cite{CLZ}. In this case,
since $T(A_1)$ is compact, $A_1$ is uniformly McDuff (see  Definition 4.1 of \cite{CLZ},
or Definition 4.2 of \cite{CETW}). 
Since $A_1$ has strict comparison,  by Proposition 4.4 of \cite{CLZ} (also a  version of Matui-Sato's result),
we conclude that
$A_1\cong A_1\otimes {\cal Z}.$
\end{proof}
%
%
%%%%%%%%%%%%%
%
%
\iffalse
%%%%%%%%%%%%%%%%%%%%%%%
\begin{thm}\label{TM-2}
Let $A$ be a separable amenable  simple \CA\, with $T(A)\not=\emptyset$
such that $\partial_e(T(A))$  satisfies condition (C).
%has strong countable dimension. 
%which has a basis $S$ such that  $\partial_e(S)$ has countably many points.
 %such that 
%$\overline{\partial_e(S)}$ has finitely many limit points. 
Then  following are equivalent.

(1) $A$ has strict comparison and T-tracial approximate oscillation zero,

(2) $A$ has strict comparison and stable rank one,

(3) $A\cong A\otimes {\cal Z},$

(4) $A$ has finite nuclear dimension.

\end{thm}

\begin{proof}
The equivalence of (1) and (2) follows from  Theorem 1.1 of \cite{FLosc} (see also \cite{Th}). 
As in the proof of Theorem \ref{TM-1}, it suffices to show (1) $\Rightarrow$ (3). 
It follows from Theorem 1.1 of \cite{FLosc} that the map $\Gamma: \Cu(A)\to {\rm LAff}_+(\wtd{T}(A))$ is surjective.
Choose $a\in {\rm Ped}(A)_+^{\bf 1}\setminus \{0\}$ such that $d_\tau(a)$ is a continuous function 
on $\wtd{T}(A).$   By Theorem 5.3 of \cite{eglnp}, $A_1$ has continuous scale. 
By Brown's stable isomorphism theorem, $\Her(a)\otimes {\cal K}\cong A\otimes {\cal K}.$
It suffices to show that $A_1=\Her(a)$ is ${\cal Z}$-stable. 
\Wlog, we may assume that $A$ has continuous scale. 
The rest of the  proof of this is exactly the same as that  of Theorem \ref{TM-1}, but applying Lemma \ref{Lcover2}
instead of Lemma \ref{Lcover}.
%
\end{proof}
%
\fi
%
%
%%%%%%%%%%%%%%%%%%%%%%%%%%%%%
%
%
%%%%%%%%%%%%%%%%%%%%%%%%%%%%%%%%%%%%
\begin{cor}\label{CCM-1}
Let $A$ be a separable non-elementary amenable  simple \CA\, 
%such that 
with ${\wtd{T}}(A)\not=\emptyset$
which has a basis $S$   satisfying 
%which satisfies 
condition (C) and
$\partial_e(S)$ has countably many points.
% and 
%satisfies condition (C). 
 %such that 
%$\overline{\partial_e(S)}$ has finitely many limit points. 
Then  following are equivalent.

(1) $A$ has strict comparison,

(2) $A\cong A\otimes {\cal Z},$

(3) $A$ has finite nuclear dimension.

\end{cor}

\begin{proof}
We note, by Theorem 1.1 of \cite{FLosc}, $A$ has T-tracial approximate oscillation zero. 

\end{proof}

\begin{rem}
In Theorem \ref{TM-1}, by (2) of Proposition \ref{Pextrb}, the condition on $\wtd{T}(A)$ does not depend on the choice of the basis
$S.$
If $A$ is unital, or $A$ has continuous scale, the condition 
can be stated as $T(A)$ 
satisfies condition (C)
%has 
%finitely 
%countably many  points 
(in particular, $T(A)$ may not  be a Bauer simplex and may not have finite covering dimension). 
Theorem \ref{TM-1} covers cases of simple \CA s $A$  whose 
$\partial_e(T(A))$ are as in (v) of  Remark \ref{Rmark1}, and  Examples (4), (5) and (6) in  \ref{exm-1} which have infinite dimension.   We also note that, by , there are Choquet simpleces 
without countable extremal boundaries which do not satisfy condition (C). 
%However, Corollary \ref{CCM-1} does not seem to cover all simple \CA s with countable extremal boundary 
%of tracial state spaces. We do not have a concrete example of Choquet simplex 
%with countable extremal boundary which does not satisfy condition (C). 
%
One day after the first version of this  paper was posted, we were informed by Kang Li the paper of Wei Zhang \cite{Zw}
which contains statements overlapping with our original statement.
This helps us to improve our main results. 
% (see Remark \ref{RwZ}).
%Had the issue in \cite{Zw} mentioned in Remark \ref{RwZ} been resolved and  we found a proof of Lemma 6.4 of \cite{Zw}, the condition that $\partial_e(T(A))$
%is compact in Theorem \ref{TM-2} could be removed.
% (see also Remark \ref{RwZ}). 
It is, perhaps, conceivable, that the condition  that a basis $S$ of $\wtd{T}(A)$ 
satisfies condition (C) can be  further weakened in 
%changed to 
%that $S$ has countable dimension 
%$\partial_e(T(A))$ is countable-dimensional
%from 
Theorem \ref{TM-1}. 
%At this moment, we still need it.
% to apply  the notion  of tracial approximate oscillation zero.
%On the other hand,  the condition (C) in  \ref{CCM-1} should  be removed soon. 
% without too much painful 
%repetition. 
%The notion of tracial approximate oscillation zero may be omitted in light of equivalence 
%of (1) and (2) in Theorem \ref{TM-2} (see Theorem 1.1 of \cite{FLosc}). Nevertheless it 
%is the idea of tracial approximate oscillation zero helps us to obtain Theorem \ref{TM-1} and 
%\ref{TM-2}. 
%whose result covers Theorem \ref{TM-1}
%when $A$ is unital.  We now correct this unfortunate oversight. 
\end{rem}

 \providecommand{\href}[2]{#2}

%&&&&

\noindent 
Research Center for Operator Algebras\\
School of Mathematics\\
East China Normal University\\
Shanghai 200062, China\\
and (current) Department of Mathematics\\
University of Oregon\\
Eugene, Oregon 97402,
U.S.A.

\noindent
hlin@uoregon.edu

\end{document}